\documentclass{amsart}
\usepackage{url}
\usepackage{amsmath,amsfonts,euscript,amsthm,amssymb,amscd}
\usepackage[T1]{fontenc}
\usepackage{lmodern}
\DeclareSymbolFont{greek}{U}{eur}{m}{n}
\SetSymbolFont{greek}{bold}{U}{eur}{b}{n}
\DeclareSymbolFontAlphabet{\gr}{greek}

\DeclareMathSymbol{\alpha}{\mathord}{greek}{"0B}
\DeclareMathSymbol{\beta}{\mathord}{greek}{"0C}
\DeclareMathSymbol{\gamma}{\mathord}{greek}{"0D}
\DeclareMathSymbol{\delta}{\mathord}{greek}{"0E}
\DeclareMathSymbol{\epsilon}{\mathord}{greek}{"0F}
\DeclareMathSymbol{\zeta}{\mathord}{greek}{"10}
\DeclareMathSymbol{\eta}{\mathord}{greek}{"11}
\DeclareMathSymbol{\theta}{\mathord}{greek}{"12}
\DeclareMathSymbol{\iota}{\mathord}{greek}{"13}
\DeclareMathSymbol{\kappa}{\mathord}{greek}{"14}
\DeclareMathSymbol{\lambda}{\mathord}{greek}{"15}
\DeclareMathSymbol{\mu}{\mathord}{greek}{"16}
\DeclareMathSymbol{\nu}{\mathord}{greek}{"17}
\DeclareMathSymbol{\xi}{\mathord}{greek}{"18}
\DeclareMathSymbol{\pi}{\mathord}{greek}{"19}
\DeclareMathSymbol{\rho}{\mathord}{greek}{"1A}
\DeclareMathSymbol{\sigma}{\mathord}{greek}{"1B}
\DeclareMathSymbol{\tau}{\mathord}{greek}{"1C}
\DeclareMathSymbol{\upsilon}{\mathord}{greek}{"1D}
\DeclareMathSymbol{\phi}{\mathord}{greek}{"1E}
\DeclareMathSymbol{\chi}{\mathord}{greek}{"1F}
\DeclareMathSymbol{\psi}{\mathord}{greek}{"20}
\DeclareMathSymbol{\omega}{\mathord}{greek}{"21}
\DeclareMathSymbol{\varepsilon}{\mathord}{greek}{"22}
\DeclareMathSymbol{\vartheta}{\mathord}{greek}{"23}
\DeclareMathSymbol{\varpi}{\mathord}{greek}{"24}
\let\varrho\rho
\let\varsigma\sigma
\DeclareMathSymbol{\varphi}{\mathord}{greek}{"27}
\theoremstyle{definition}
\newtheorem{Def}{Definition}[subsection]
\newtheorem{Def-Prop}[Def]{Definition-Proposition}

\newtheorem{Ex}[Def]{Example}

\newtheorem{remark}[Def]{Remark}
\newtheorem{Prop}[Def]{Proposition}
\newtheorem{Lemma}[Def]{Lemma}
\newtheorem{Cor}[Def]{Corollary}
\newtheorem{Assum}{Assumption}

\newtheorem{Const}{Construction}
\DeclareMathOperator*{\hlim}{\text{holim}}
\DeclareMathOperator*{\hclim}{\text{hocolim}}
\DeclareMathOperator*{\clim}{\text{colim}}

\newcommand{\hdot}{{\:\raisebox{3pt}{\text{\circle*{1.5}}}}}
\newcommand{\hdotc}{{\:\raisebox{1pt}{\text{\circle*{1.5}}}}}

\expandafter\ifx\csname ProvidesPackage\endcsname\relax\toks0=\expandafter{%
\fi\ProvidesPackage{diagrams}[2014/12/31 v3.94 Paul Taylor's commutative
diagrams]%%
\toks0=\bgroup}%%
\ifx\diagram\isundefined\else\expandafter\endinput
\fi

\edef\cdrestoreat{%%
\noexpand\catcode`\noexpand\@=\the\catcode`\@%%
\noexpand\catcode`\noexpand\#=\the\catcode`\#%%
\noexpand\catcode`\noexpand\$=\the\catcode`\$%%
\noexpand\catcode`\noexpand\<=\the\catcode`\<%%
\noexpand\catcode`\noexpand\>=\the\catcode`\>%%
\noexpand\catcode`\noexpand\:=\the\catcode`\:%% Johannes L. Braams's
\noexpand\catcode`\noexpand\;=\the\catcode`\;%% Babel languages package
\noexpand\catcode`\noexpand\!=\the\catcode`\!%% makes these \active.
\noexpand\catcode`\noexpand\?=\the\catcode`\?%%
\noexpand\catcode`\noexpand\+=\the\catcode'53%% texinfo @+ is @outer@active
}\catcode`\@=11 \catcode`\#=6 \catcode`\<=12 \catcode`\>=12 \catcode'53=12
\catcode`\:=12 \catcode`\;=12 \catcode`\!=12 \catcode`\?=12

\ifx\diagram@help@messages\CD@qK\let\diagram@help@messages y\fi

\def\cdps@Rokicki#1{\special{ps:#1}}\let\cdps@dvips\cdps@Rokicki\let
\cdps@RadicalEye\cdps@Rokicki\let\CD@HB\cdps@Rokicki\let\CD@IK\cdps@Rokicki
\let\CD@HB\cdps@Rokicki%%
\def\cdps@Bechtolsheim#1{\special{dvitps: Literal "#1"}}%
\let\cdps@dvitps\cdps@Bechtolsheim\let\cdps@IntegratedComputerSystems
\cdps@Bechtolsheim%%
\def\cdps@Clark#1{\special{dvitops: inline #1}}%%
\let\cdps@dvitops\cdps@Clark%%
\let\cdps@OzTeX\empty\let\cdps@oztex\empty\let\cdps@Trevorrow\empty%%
\def\cdps@Coombes#1{\special{ps-string #1}}%%
\count@=\year\multiply\count@12 \advance\count@\month%%
\ifnum\count@>24240 %% (December 2019)
\ifnum\count@>24243 %% (March 2020)
\errhelp{You may press RETURN and carry on for the time being.}\errmessage{! remain
compatible with your files. Please get it NOW.}\fi\fi

\def\CD@DE{\global\let}\def\CD@RH{\outer\def}

{\escapechar\m@ne\xdef\CD@o{\string\{}\xdef\CD@yC{\string\}}%%
\catcode`\&=4 \CD@DE\CD@Q=&\xdef\CD@S{\string\&}%%ascii three ands
\catcode`\$=3 \CD@DE\CD@k=$\CD@DE\CD@ND=$%%ascii three dollars
\xdef\CD@nC{\string\$}\gdef\CD@LG{$$}%%ascii three dollars
\catcode`\_=8 \CD@DE\CD@lJ=_%%ascii two underlines
\obeylines\catcode`\^=7 \CD@DE\@super=^%%ascii two carets
\ifnum\newlinechar=10 \gdef\CD@uG{^^J}%%ascii two carets
\else\ifnum\newlinechar=13 \gdef\CD@uG{^^M}%%ascii two carets
\else\ifnum\newlinechar=-1 \gdef\CD@uG{^^J}%%ascii two carets
\else\CD@DE\CD@uG\space\expandafter\message{! input error: \noexpand
\newlinechar\space is ASCII \the\newlinechar, not LF=10 or CR=13.}%%
\fi\fi\fi}%%

\mathchardef\lessthan='30474 \mathchardef\greaterthan='30476

\ifx\tenln\CD@qK%%
\font\tenln=line10\relax%% Hit return - who needs diagonals?
\fi\ifx\tenlnw\CD@qK\ifx\tenln\nullfont\let\tenlnw\nullfont\else%%
\font\tenlnw=linew10\relax%% Hit return - who needs diagonals?
\fi\fi%%

\ifx\inputlineno\CD@qK\csname newcount\endcsname\inputlineno\inputlineno\m@ne
\fi

\def\cd@shouldnt#1{\CD@KB{* THIS (#1) SHOULD NEVER HAPPEN! *}}

\def\get@round@pair#1(#2,#3){#1{#2}{#3}}%%ascii round brackets ()
\def\get@square@arg#1[#2]{#1{#2}}%%ascii square brackets []
\def\CD@AE#1{\CD@PK\let\CD@DH\CD@@E\CD@@E#1,],}%%ascii sq brackets
\def\CD@m{[}\def\CD@RD{]}\def\commdiag#1{{\let\enddiagram\relax\diagram[]#1%
\enddiagram}}

\def\CD@BF{{\ifx\CD@EH[\aftergroup\get@square@arg\aftergroup\CD@YH\else
\aftergroup\CD@JH\fi}}%%
\def\CD@CF#1#2{\def\CD@YH{#1}\def\CD@JH{#2}\futurelet\CD@EH\CD@BF}

\def\CD@KK{|}

\def\CD@PB{%% arguments to maps inside diagrams
\tokcase\CD@DD:\CD@y\break@args;\catcase\@super:\upper@label;\catcase\CD@lJ:%
\lower@label;\tokcase{~}:\middle@label;%%ascii tilde
\tokcase<:\CD@iF;%%ascii less-than
\tokcase>:\CD@iI;%%ascii greater-than
\tokcase(:\CD@BC;%%)%ascii open round bracket
\tokcase[:\optional@;%%]%ascii open square bracket
\tokcase.:\CD@JJ;%%ascii dot 12.7.94
\catcase\space:\eat@space;\catcase\bgroup:\positional@;\default:\CD@@A
\break@args;\endswitch}

\def\switch@arg{%% arguments to horizontal maps outside diagrams
\catcase\@super:\upper@label;\catcase\CD@lJ:\lower@label;\tokcase[:\optional@
;%%]%ascii open square bracket
\tokcase.:\CD@JJ;%%ascii dot 12.7.94 % ; was : before 15.6.97
\catcase\space:\eat@space;\catcase\bgroup:\positional@;\tokcase{~}:%
\middle@label;%%ascii tilde (questionable!)
\default:\CD@y\break@args;\endswitch}

\let\CD@tJ\relax\ifx\protect\CD@qK\let\protect\relax\fi\ifx\AtEndDocument
\CD@qK\def\CD@PG{\CD@gB}\def\CD@GF#1#2{}\else\def\CD@PG#1{\edef\CD@CH{#1}%
\expandafter\CD@oC\CD@CH\CD@OD}\def\CD@oC#1\CD@OD{\AtEndDocument{\typeout{%
\CD@tA: #1}}}\def\CD@GF#1#2{\gdef#1{#2}\AtEndDocument{#1}}\fi\def\CD@ZA#1#2{%
\def#1{\CD@PG{#2\CD@mD\CD@W}\CD@DE#1\relax}}\def\CD@uF#1\repeat{\def\CD@p{#1}%
\CD@OF}\def\CD@OF{\CD@p\relax\expandafter\CD@OF\fi}\def\CD@sF#1\repeat{\def
\CD@q{#1}\CD@PF}\def\CD@PF{\CD@q\relax\expandafter\CD@PF\fi}\def\CD@tF#1%
\repeat{\def\CD@r{#1}\CD@QF}\def\CD@QF{\CD@r\relax\expandafter\CD@QF\fi}\def
\CD@tG#1#2#3{\def#2{\let#1\iftrue}\def#3{\let#1\iffalse}#3}\if y%
\diagram@help@messages\def\CD@rG#1#2{\csname newtoks\endcsname#1#1=%
\expandafter{\csname#2\endcsname}}\else\csname newtoks\endcsname\no@cd@help
\no@cd@help{See the manual}\def\CD@rG#1#2{\let#1\no@cd@help}\fi\chardef\CD@lF
=1 \chardef\CD@lI=2 \chardef\CD@MH=5 \chardef\CD@tH=6 \chardef\CD@sH=7
\chardef\CD@PC=9 \dimendef\CD@hI=2 \dimendef\CD@hF=3 \dimendef\CD@mF=4
\dimendef\CD@mI=5 \dimendef\CD@wJ=6 \dimendef\CD@tI=8 \dimendef\CD@sI=9
\skipdef\CD@uB=1 \skipdef\CD@NF=2 \skipdef\CD@tB=3 \skipdef\CD@ZE=4 \skipdef
\countdef\CD@JC=9 \countdef\CD@gD=8 \countdef\CD@A=7 \def\sdef#1#2{\def#1{#2}%
}\def\CD@L#1{\expandafter\aftergroup\csname#1\endcsname}\def\CD@RC#1{%
\expandafter\def\csname#1\endcsname}\def\CD@sD#1{\expandafter\gdef\csname#1%
\endcsname}\def\CD@vC#1{\expandafter\edef\csname#1\endcsname}\def\CD@nF#1#2{%
\expandafter\let\csname#1\expandafter\endcsname\csname#2\endcsname}\def\CD@EE
#1#2{\expandafter\CD@DE\csname#1\expandafter\endcsname\csname#2\endcsname}%
\def\CD@AK#1{\csname#1\endcsname}\def\CD@XJ#1{\expandafter\show\csname#1%
\endcsname}\def\CD@ZJ#1{\expandafter\showthe\csname#1\endcsname}\def\CD@WJ#1{%
\expandafter\showbox\csname#1\endcsname}\def\CD@tA{Commutative Diagram}\edef
\CD@kH{\string\par}\edef\CD@dC{\string\diagram}\edef\CD@HD{\string\enddiagram
}\edef\CD@EC{\string\\}\def\CD@eF{LaTeX}\ifx\@ignoretrue\CD@qK\expandafter
\CD@tG\csname if@ignore\endcsname\ignore@true\ignore@false\def\@ignoretrue{%
\global\ignore@true}\def\@ignorefalse{\global\ignore@false}\fi

\def\CD@g{{\ifnum0=`}\fi}\def\CD@wC{\ifnum0=`{\fi}}\def\catcase#1:{\ifcat
\noexpand\CD@EH#1\CD@tJ\expandafter\CD@kC\else\expandafter\CD@dJ\fi}\def
\tokcase#1:{\ifx\CD@EH#1\CD@tJ\expandafter\CD@kC\else\expandafter\CD@dJ\fi}%
\def\CD@kC#1;#2\endswitch{#1}\def\CD@dJ#1;{}\let\endswitch\relax\def\default:%
#1;#2\endswitch{#1}\ifx\at@\CD@qK\def\at@{@}\fi\edef\CD@P{\CD@o pt\CD@yC}%
\CD@RC{\CD@P>}#1>#2>{\CD@z\rTo\sp{#1}\sb{#2}\CD@z}\CD@RC{\CD@P<}#1<#2<{\CD@z
\lTo\sp{#1}\sb{#2}\CD@z}\CD@RC{\CD@P)}#1)#2){\CD@z\rTo\sp{#1}\sb{#2}\CD@z}%
\CD@RC{\CD@P(}#1(#2({\CD@z\lTo\sp{#1}\sb{#2}\CD@z}%%ascii brack
\def\CD@O{\def\endCD{\enddiagram}\CD@RC{\CD@P A}##1A##2A{\uTo<{##1}>{##2}%
\CD@z\CD@z}\CD@RC{\CD@P V}##1V##2V{\dTo<{##1}>{##2}\CD@z\CD@z}\CD@RC{\CD@P=}{%
\CD@z\hEq\CD@z}\CD@RC{\CD@P\CD@KK}{\vEq\CD@z\CD@z}\CD@RC{\CD@P\string\vert}{%
\vEq\CD@z\CD@z}\CD@RC{\CD@P.}{\CD@z\CD@z}\let\CD@z\CD@Q}\def\CD@IE{\let\tmp
\CD@JE\ifcat A\noexpand\CD@CH\else\ifcat=\noexpand\CD@CH\else\ifcat\relax
\noexpand\CD@CH\else\let\tmp\at@\fi\fi\fi\tmp}\def\CD@JE#1{\CD@nF{tmp}{\CD@P
\string#1}\ifx\tmp\relax\def\tmp{\at@#1}\fi\tmp}\def\CD@z{}\begingroup
\aftergroup\def\aftergroup\CD@T\aftergroup{\aftergroup\def\catcode`\@\active
\aftergroup @\endgroup{\futurelet\CD@CH\CD@IE}}\newcount\CD@uA\newcount\CD@vA
\newcount\CD@wA\newcount\CD@xA\newdimen\CD@OA\newdimen\CD@PA\CD@tG\CD@gE
\CD@@A\CD@y\CD@tG\CD@hE\CD@EA\CD@BA\newdimen\CD@RA\newdimen\CD@SA\newcount
\CD@yA\newcount\CD@zA\newdimen\CD@QA\newbox\CD@DA\CD@tG\CD@lE\CD@dA\CD@bA
\newcount\CD@LH\newcount\CD@TC\def\CD@V#1#2{\ifdim#1<#2\relax#1=#2\relax\fi}%
\def\CD@X#1#2{\ifdim#1>#2\relax#1=#2\relax\fi}\newdimen\CD@XH\CD@XH=1sp
\newdimen\CD@zC\CD@zC\z@\def\CD@cJ{\ifdim\CD@zC=1em\else\CD@nJ\fi}\def\CD@nJ{%
\CD@zC1em\def\CD@NC{\fontdimen8\textfont3 }\CD@@J\CD@NJ\setbox0=\vbox{\CD@t
\noindent\CD@k\null\penalty-9993\null\CD@ND\null\endgraf\setbox0=\lastbox
\unskip\unpenalty\setbox1=\lastbox\global\setbox\CD@IG=\hbox{\unhbox0\unskip
\unskip\unpenalty\setbox0=\lastbox}\global\setbox\CD@KG=\hbox{\unhbox1\unskip
\unpenalty\setbox1=\lastbox}}}\newdimen\CD@@I\CD@@I=1true in \divide\CD@@I300
\def\CD@zH#1{\multiply#1\tw@\advance#1\ifnum#1<\z@-\else+\fi\CD@@I\divide#1%
\tw@\divide#1\CD@@I\multiply#1\CD@@I}\def\MapBreadth{\afterassignment\CD@gI
\CD@LF}\newdimen\CD@LF\newdimen\CD@oI\def\CD@gI{\CD@oI\CD@LF\CD@V\CD@@I{4%
\CD@XH}\CD@X\CD@@I\p@\CD@zH\CD@oI\ifdim\CD@LF>\z@\CD@V\CD@oI\CD@@I\fi\CD@cJ}%
\def\CD@RJ#1{\CD@zD\count@\CD@@I#1\ifnum\count@>\z@\divide\CD@@I\count@\fi
\CD@gI\CD@NJ}\def\CD@NJ{\dimen@\CD@QC\count@\dimen@\divide\count@5\divide
\count@\CD@@I\edef\CD@OC{\the\count@}}\def\CD@AJ{\CD@QJ\z@}\def\CD@QJ#1{%
\CD@tI\axisheight\advance\CD@tI#1\relax\advance\CD@tI-.5\CD@oI\CD@zH\CD@tI
\CD@sI-\CD@tI\advance\CD@tI\CD@LF}\newdimen\CD@DC\CD@DC\z@\newdimen\CD@eJ
\CD@eJ\z@\def\CD@CJ#1{\CD@sI#1\relax\CD@tI\CD@sI\advance\CD@tI\CD@LF\relax}%
\def\horizhtdp{height\CD@tI depth\CD@sI}\def\axisheight{\fontdimen22\the
\textfont\tw@}\def\script@axisheight{\fontdimen22\the\scriptfont\tw@}\def
\ss@axisheight{\fontdimen22\the\scriptscriptfont\tw@}\def\CD@NC{0.4pt}\def
\CD@VK{\fontdimen3\textfont\z@}\def\CD@UK{\fontdimen3\textfont\z@}\newdimen
\PileSpacing\newdimen\CD@nA\CD@nA\z@\def\CD@RG{\ifincommdiag1.3em\else2em\fi}%
\newdimen\CD@YB\def\CellSize{\afterassignment\CD@kB\DiagramCellHeight}%
\newdimen\DiagramCellHeight\DiagramCellHeight-\maxdimen\newdimen
\DiagramCellWidth\DiagramCellWidth-\maxdimen\def\CD@kB{\DiagramCellWidth
\DiagramCellHeight}\def\CD@QC{3em}\newdimen\MapShortFall\def\MapsAbut{%
\MapShortFall\z@\objectheight\z@\objectwidth\z@}\newdimen\CD@iA\CD@iA\z@
\CD@tG\CD@vE\CD@aB\CD@ZB\expandafter\ifx\expandafter\iftrue\csname
ifUglyObsoleteDiagrams\endcsname\CD@ZB\else\CD@aB\fi\CD@nF{%
ifUglyObsoleteDiagrams}{relax}\newif\ifUglyObsoleteDiagrams\def\CD@nK{\CD@aB
\UglyObsoleteDiagramsfalse}\def\CD@oK{\CD@ZB\UglyObsoleteDiagramstrue}\CD@vE
\CD@nK\else\CD@oK\fi\CD@tG\CD@hK\CD@dK\CD@cK\CD@cK\def\CD@sK{\ifx\pdfoutput
\CD@qK\else\ifx\pdfoutput\relax\else\ifnum\pdfoutput>\z@\CD@pK\fi\fi\fi} \def
\CD@pK{\global\CD@dK\global\CD@aB\global\UglyObsoleteDiagramsfalse\global\let
\CD@n\empty\global\let\CD@oK\relax\global\let\CD@pK\relax\global\let\CD@sK
\relax}\def\CD@tK#1{}\ifx\pdfliteral\CD@qK\else\ifx
\pdfliteral\relax\else\let\CD@tK\pdfliteral\fi\fi\ifx\XeTeXrevision\CD@qK
\else\ifx\XeTeXrevision\relax\else\ifdim\XeTeXrevision pt<.996pt \expandafter
\message{! XeTeX version \XeTeXrevision\space does not support PDF literals,
so diagonals will not work!}\else\expandafter\message{RUNNING UNDER XETEX
\XeTeXrevision}\CD@pK\fi\fi\fi\CD@sK\def\newarrowhead{\CD@mG h\CD@BG\CD@GG>}%
\def\newarrowtail{\CD@mG t\CD@BG\CD@GG>}\def\newarrowmiddle{\CD@mG m\CD@BG
\hbox@maths\empty}\def\newarrowfiller{\CD@mG f\CD@bE\CD@MK-}\def\CD@mG#1#2#3#%
4#5#6#7#8#9{\CD@RC{r#1:#5}{#2{#6}}\CD@RC{l#1:#5}{#2{#7}}\CD@RC{d#1:#5}{#3{#8}%
}\CD@RC{u#1:#5}{#3{#9}}\CD@vC{-#1:#5}{\expandafter\noexpand\csname-#1:#4%
\endcsname\noexpand\CD@MC}\CD@vC{+#1:#5}{\expandafter\noexpand\csname+#1:#4%
\endcsname\noexpand\CD@MC}}\CD@ZA\CD@MC{\CD@eF\space diagonals are used unless
PostScript is set}\def\defaultarrowhead#1{\edef\CD@sJ{#1}\CD@@J}\def\CD@@J{%
\CD@IJ\CD@sJ<>ht\CD@IJ\CD@sJ<>th}\def\CD@IJ#1#2#3#4#5{\CD@HJ{r#4}{#3}{l#5}{#2%
}{r#4:#1}\CD@HJ{r#5}{#2}{l#4}{#3}{l#4:#1}\CD@HJ{d#4}{#3}{u#5}{#2}{d#4:#1}%
\CD@HJ{d#5}{#2}{u#4}{#3}{u#4:#1}}\def\CD@HJ#1#2#3#4#5{\begingroup\aftergroup
\CD@GJ\CD@L{#1+:#2}\CD@L{#1:#2}\CD@L{#3:#4}\CD@L{#5}\endgroup}\def\CD@GJ#1#2#%
3#4{\csname newbox\endcsname#1\def#2{\copy#1}\def#3{\copy#1}\setbox#1=\box
\voidb@x}\def\CD@sJ{}\CD@@J\def\CD@GJ#1#2#3#4{\setbox#1=#4}\ifx\tenln
\nullfont\def\CD@sJ{vee}\else\let\CD@sJ\CD@eF\fi\def\CD@xF#1#2#3{\begingroup
\aftergroup\CD@wF\CD@L{#1#2:#3#3}\CD@L{#1#2:#3}\aftergroup\CD@yF\CD@L{#1#2:#3%
-#3}\CD@L{#1#2:#3}\endgroup}\def\CD@wF#1#2{\def#1{\hbox{\rlap{#2}\kern.4%
\CD@zC#2}}}\def\CD@yF#1#2{\def#1{\hbox{\rlap{#2}\kern.4\CD@zC#2\kern-.4\CD@zC
}}}\CD@xF lh>\CD@xF rt>\CD@xF rh<\CD@xF rt<\def\CD@yF#1#2{\def#1{\hbox{\kern-%
.4\CD@zC\rlap{#2}\kern.4\CD@zC#2}}}\CD@xF rh>\CD@xF lh<\CD@xF lt>\CD@xF lt<%
\def\CD@wF#1#2{\def#1{\vbox{\vbox to\z@{#2\vss}\nointerlineskip\kern.4\CD@zC#%
2}}}\def\CD@yF#1#2{\def#1{\vbox{\vbox to\z@{#2\vss}\nointerlineskip\kern.4%
\CD@zC#2\kern-.4\CD@zC}}}\CD@xF uh>\CD@xF dt>\CD@xF dh<\CD@xF dt<\def\CD@yF#1%
#2{\def#1{\vbox{\kern-.4\CD@zC\vbox to\z@{#2\vss}\nointerlineskip\kern.4%
\CD@zC#2}}}\CD@xF dh>\CD@xF ut>\CD@xF uh<\CD@xF ut<\def\CD@BG#1{\hbox{%
\mathsurround\z@\offinterlineskip\CD@k\mkern-1.5mu{#1}\mkern-1.5mu\CD@ND}}%
\def\hbox@maths#1{\hbox{\CD@k#1\CD@ND}}\def\CD@GG#1{\hbox to\CD@LF{\setbox0=%
\hbox{\offinterlineskip\mathsurround\z@\CD@k{#1}\CD@ND}\dimen0.5\wd0\advance
\dimen0-.5\CD@oI\CD@zH{\dimen0}\kern-\dimen0\unhbox0\hss}}\def\CD@sB#1{\hbox
to2\CD@LF{\hss\offinterlineskip\mathsurround\z@\CD@k{#1}\CD@ND\hss}}\def
\CD@vF#1{\hbox{\mathsurround\z@\CD@k{#1}\CD@ND}}\def\CD@bE#1{\hbox{\kern-.15%
\CD@zC\CD@k{#1}\CD@ND\kern-.15\CD@zC}}\def\CD@MK#1{\vbox{\offinterlineskip
\kern-.2ex\CD@GG{#1}\kern-.2ex}}\def\@fillh{\xleaders\vrule\horizhtdp}\def
\@fillv{\xleaders\hrule width\CD@LF}\CD@nF{rf:-}{@fillh}\CD@nF{lf:-}{@fillh}%
\CD@nF{df:-}{@fillv}\CD@nF{uf:-}{@fillv}\CD@nF{rh:}{null}\CD@nF{rm:}{null}%
\CD@nF{rt:}{null}\CD@nF{lh:}{null}\CD@nF{lm:}{null}\CD@nF{lt:}{null}\CD@nF{dh%
:}{null}\CD@nF{dm:}{null}\CD@nF{dt:}{null}\CD@nF{uh:}{null}\CD@nF{um:}{null}%
\CD@nF{ut:}{null}\CD@nF{+h:}{null}\CD@nF{+m:}{null}\CD@nF{+t:}{null}\CD@nF{-h%
:}{null}\CD@nF{-m:}{null}\CD@nF{-t:}{null}\def\CD@@D{\hbox{\vrule height 1pt
depth-1pt width 1pt}}\CD@RC{rf:}{\CD@@D}\CD@nF{lf:}{rf:}\CD@nF{+f:}{rf:}%
\CD@RC{df:}{\CD@@D}\CD@nF{uf:}{df:}\CD@nF{-f:}{df:}\def\CD@BD{\CD@U\null
\CD@@D\null\CD@@D\null}\edef\CD@lG{\string\newarrow}\def\newarrow#1#2#3#4#5#6%
{\begingroup\edef\@name{#1}\edef\CD@oJ{#2}\edef\CD@iD{#3}\edef\CD@QG{#4}\edef
\CD@jD{#5}\edef\CD@LE{#6}\let\CD@HE\CD@sG\let\CD@FK\CD@BH\let\@x\CD@AH\ifx
\CD@oJ\CD@iD\let\CD@oJ\empty\fi\ifx\CD@LE\CD@jD\let\CD@LE\empty\fi\def\CD@LI{%
r}\def\CD@SF{l}\def\CD@IC{d}\def\CD@yJ{u}\def\CD@gH{+}\def\@m{-}\ifx\CD@iD
\CD@jD\ifx\CD@QG\CD@iD\let\CD@QG\empty\fi\ifx\CD@LE\empty\ifx\CD@iD\CD@aE\let
\@x\CD@yG\else\let\@x\CD@zG\fi\fi\else\edef\CD@a{\CD@iD\CD@oJ}\ifx\CD@a\empty
\ifx\CD@QG\CD@jD\let\CD@QG\empty\fi\fi\fi\ifmmode\aftergroup\CD@kG\else\CD@@A
\CD@oB rh{head\space\space}\CD@LE\CD@oB rf{filler}\CD@iD\CD@oB rm{middle}%
\CD@QG\ifx\CD@jD\CD@iD\else\CD@oB rf{filler}\CD@jD\fi\CD@oB rt{tail\space
\space}\CD@oJ\CD@gE\CD@HE\CD@FK\@x\CD@nG l-2+2{lu}{nw}\NorthWest\CD@nG r+2+2{%
ru}{ne}\NorthEast\CD@nG l-2-2{ld}{sw}\SouthWest\CD@nG r+2-2{rd}{se}\SouthEast
\else\aftergroup\CD@b\CD@L{r\@name}\fi\fi\endgroup}\def\CD@sG{\CD@vG\CD@LI
\CD@SF rl\Horizontal@Map}\def\CD@BH{\CD@vG\CD@IC\CD@yJ du\Vertical@Map}\def
\CD@AH{\CD@vG\CD@gH\@m+-\Vector@Map}\def\CD@yG{\CD@vG\CD@gH\@m+-\Slant@Map}%
\def\CD@zG{\CD@vG\CD@gH\@m+-\Slope@Map}\catcode`\/=\active\def\CD@vG#1#2#3#4#%
5{\CD@jG#1#3#5t:\CD@oJ/f:\CD@iD/m:\CD@QG/f:\CD@jD/h:\CD@LE//\CD@jG#2#4#5h:%
\CD@LE/f:\CD@jD/m:\CD@QG/f:\CD@iD/t:\CD@oJ//}\def\CD@jG#1#2#3#4//{\edef\CD@fG
{#2}\aftergroup\sdef\CD@L{#1\@name}\aftergroup{\aftergroup#3\CD@M#4//%
\aftergroup}}\def\CD@M#1/{\edef\CD@EH{#1}\ifx\CD@EH\empty\else\CD@L{\CD@fG#1}%
\expandafter\CD@M\fi}\catcode`\/=12 \def\CD@nG#1#2#3#4#5#6#7#8{\aftergroup
\sdef\CD@L{#6\@name}\aftergroup{\CD@L{#2\@name}\if#2#4\aftergroup\CD@CI\else
\aftergroup\CD@BI\fi\CD@L{#1\@name}%
\aftergroup(\aftergroup#3\aftergroup,\aftergroup#5\aftergroup)\aftergroup}}%
\def\CD@oB#1#2#3#4{\expandafter\ifx\csname#1#2:#4\endcsname\relax\CD@y\CD@gB{%
arrow#3 "#4" undefined}\fi}\CD@rG\CD@VE{All five components must be defined
before an arrow.}\CD@rG\CD@SE{\CD@lG, unlike \string\HorizontalMap, is a
declaration.}\def\CD@b#1{\CD@YA{Arrows \string#1 etc could not be defined}%
\CD@VE}\def\CD@kG{\CD@YA{misplaced \CD@lG}\CD@SE}\def\newdiagramgrid#1#2#3{%
\CD@RC{cdgh@#1}{#2,],}%% ASCII close square bracket
\CD@RC{cdgv@#1}{#3,],}}%% ASCII close square bracket
\CD@tG\ifincommdiag\incommdiagtrue\incommdiagfalse\CD@tG\CD@@F\CD@IF\CD@HF
\newcount\CD@VA\CD@VA=0 \def\CD@yH{\CD@VA6 }\def\CD@OB{\CD@VA1 \global\CD@yA1
\CD@DE\CD@YF\empty}\def\CD@YF{}\def\CD@nB#1{\relax\CD@MD\edef\CD@vJ{#1}%
\begingroup\CD@rE\else\ifcase\CD@VA\ifmmode\else\CD@YG\CD@E0\fi\or\CD@cE5\or
\CD@YG\CD@F5\or\CD@YG\CD@B5\or\CD@YG\CD@B5\or\CD@YG\CD@C5\or\CD@cE7\or\CD@YG
\CD@D7\fi\fi\endgroup\xdef\CD@YF{#1}}\def\CD@pB#1#2#3#4#5{\relax\CD@MD\xdef
\CD@vJ{#4}\begingroup\ifnum\CD@VA<#1 \expandafter\CD@cE\ifcase\CD@VA0\or#2\or
#3\else#2\fi\else\ifnum\CD@VA<6 \CD@tJ\CD@YG\CD@B#2\else\CD@YG\CD@G#2\fi\fi
\endgroup\CD@DE\CD@YF\CD@vJ\ifincommdiag\let\CD@ZD#5\else\let\CD@ZD\CD@LK\fi}%
\def\CD@yI{\global\CD@yA=\ifnum\CD@VA<5 1\else2\fi\relax}\def\CD@OI{\CD@VA
\CD@yA}\def\CD@cE#1{\aftergroup\CD@VA\aftergroup#1\aftergroup\relax}\def
\CD@HH{\def\CD@nB##1{\relax}\let\CD@pB\CD@FH\let\CD@yH\relax\let\CD@OB\relax
\let\CD@yI\relax\let\CD@OI\relax}\def\CD@FH#1#2#3#4#5{\ifincommdiag\let\CD@ZD
#5\else\xdef\CD@vJ{#4}\let\CD@ZD\CD@LK\fi}\def\CD@YG#1{\aftergroup#1%
\aftergroup\relax\CD@cE}\def\CD@B{\CD@YE\CD@S\CD@ME\CD@Q}\def\CD@G{\CD@YE{%
\CD@yC\CD@S}\CD@XE\CD@QD\CD@Q}\def\CD@F{\CD@YE{*\CD@S}\CD@RE\clubsuit\CD@Q}%
\def\CD@C{\CD@YE{\CD@S*\CD@S}\CD@RE\CD@Q\clubsuit\CD@Q}\def\CD@D{\CD@YE\CD@EC
\CD@TE\\}\def\CD@E{\CD@YE\CD@nC\CD@QE\CD@k}\def\CD@LK{\CD@YA{\CD@vJ\space
ignored \CD@dH}\CD@WE}\def\CD@FE{}\def\CD@d{\CD@YA{maps must never be enclosed
in braces}\CD@OE}\def\CD@dH{outside diagram}\def\CD@FC{\string\HonV, \string
\VonH\space and \string\HmeetV}\CD@rG\CD@ME{The way that horizontal and
vertical arrows are terminated implicitly means\CD@uG that they cannot be
mixed with each other or with \CD@FC.}\CD@rG\CD@XE{\string\pile\space is for
parallel horizontal arrows; verticals can just be put together in\CD@uG a cell%
. \CD@FC\space are not meaningful in a \string\pile.}\CD@rG\CD@RE{The
horizontal maps must point to an object, not each other (I've put in\CD@uG one
which you're unlikely to want). Use \string\pile\space if you want them
parallel.}\CD@rG\CD@TE{Parallel horizontal arrows must be in separate layers
of a \string\pile.}\CD@rG\CD@QE{Horizontal arrows may be used \CD@dH s, but
must still be in maths.}\CD@rG\CD@WE{Vertical arrows, \CD@FC\space\CD@dH s don%
't know where\CD@uG where to terminate.}\CD@rG\CD@OE{This prevents them from
stretching correctly.}\def\CD@YE#1{\CD@YA{"#1" inserted \ifx\CD@YF\empty
before \CD@vJ\else between \CD@YF\ifx\CD@YF\CD@vJ s\else\space and \CD@vJ\fi
\fi}}\count@=\year\multiply\count@12 \advance\count@\month\ifnum\count@>24247
\expandafter\endinput\fi\def
\Horizontal@Map{\CD@nB{horizontal map}\CD@LC\CD@TJ\CD@qD}\def\CD@TJ{\CD@GB-%
9999 \let\CD@ZD\CD@XD\ifincommdiag\else\CD@cJ\ifinpile\else\skip2\z@ plus 1.5%
\CD@VK minus .5\CD@UK\skip4\skip2 \fi\fi\let\CD@kD\@fillh\CD@nF{fill@dot}{rf:%
.}}\def\Vector@Map{\CD@HK4}\def\Slant@Map{\CD@HK{\CD@EF255\else6\fi}}\def
\Slope@Map{\CD@HK\CD@OC}\def\CD@HK#1#2#3#4#5#6{\CD@LC\def\CD@WK{2}\def\CD@aK{%
2}\def\CD@ZK{1}\def\CD@bK{1}\let\Horizontal@Map\CD@nI\def\CD@OG{#1}\def\CD@NI
{\CD@U#2#3#4#5#6}}\def\CD@nI{\CD@TJ\CD@JB\let\CD@ZD\CD@TD\CD@qD}\CD@tG\CD@pE
\CD@rA\CD@qA\CD@rA\def\cds@missives{\CD@rA}\def\CD@TD{\CD@vE\let\CD@OG\CD@OC
\CD@x\CD@zE\CD@WF\fi\setbox0\hbox{\incommdiagfalse\CD@HI}\CD@pE\CD@aD\else
\global\CD@YC\CD@bD\fi\ifvoid6 \ifvoid7 \CD@eE\fi\fi\CD@zE\else\CD@BD\global
\CD@YC\let\CD@CG\CD@IH\CD@YD\fi\else\CD@NI\CD@MI\global\CD@YC\CD@YD\fi}\def
\CD@LC{\begingroup\dimen1=\MapShortFall\dimen2=\CD@RG\dimen5=\MapShortFall
\setbox3=\box\voidb@x\setbox6=\box\voidb@x\setbox7=\box\voidb@x\CD@pD
\mathsurround\z@\skip2\z@ plus1fill minus 1000pt\skip4\skip2 \CD@TB}\CD@tG
\CD@tE\CD@UB\CD@TB\def\CD@U#1#2#3#4#5{\let\CD@oJ#1\let\CD@iD#2\let\CD@QG#3%
\let\CD@jD#4\let\CD@LE#5\CD@TB\ifx\CD@iD\CD@jD\CD@UB\fi}\def\CD@qD#1#2#3#4#5{%
\CD@U#1#2#3#4#5\CD@tD}\def\Vertical@Map{\CD@pB433{vertical map}\CD@cD\CD@LC
\CD@GB-9995 \let\CD@kD\@fillv\CD@nF{fill@dot}{df:.}\CD@qD}\def\break@args{%
\def\CD@tD{\CD@ZD}\CD@ZD\endgroup\aftergroup\CD@FE}\def\CD@MJ{\setbox1=\CD@oJ
\setbox5=\CD@LE\ifvoid3 \ifx\CD@QG\null\else\setbox3=\CD@QG\fi\fi\CD@@G2%
\CD@iD\CD@@G4\CD@jD}\def\CD@pF#1{\ifvoid1\else\CD@oF1#1\fi\ifvoid2\else\CD@oF
2#1\fi\ifvoid3\else\CD@oF3#1\fi\ifvoid4\else\CD@oF4#1\fi\ifvoid5\else\CD@oF5#%
1\fi} \def\CD@oF#1#2{\setbox#1\vbox{\offinterlineskip\box#1\dimen@\prevdepth
\advance\dimen@-#2\relax\setbox0\null\dp0\dimen@\ht0-\dimen@\box0}}\def\CD@@G
#1#2{\ifx#2\CD@kD\setbox#1=\box\voidb@x\else\setbox#1=#2\def#2{\xleaders\box#%
1}\fi}\CD@ZA\CD@BK{\string\HorizontalMap, \string\VerticalMap\space and
\string\DiagonalMap\CD@uG are obsolete - use \CD@lG\space to pre-define maps}%
\def\HorizontalMap#1#2#3#4#5{\CD@BK\CD@nB{old horizontal map}\CD@LC\CD@TJ\def
\CD@oJ{\CD@UH{#1}}\CD@SH\CD@iD{#2}\def\CD@QG{\CD@UH{#3}}\CD@SH\CD@jD{#4}\def
\CD@LE{\CD@UH{#5}}\CD@tD}\def\VerticalMap#1#2#3#4#5{\CD@BK\CD@pB433{vertical
map}\CD@cD\CD@LC\CD@GB-9995 \let\CD@kD\@fillv\def\CD@oJ{\CD@GG{#1}}\CD@VH
\CD@iD{#2}\def\CD@QG{\CD@GG{#3}}\CD@VH\CD@jD{#4}\def\CD@LE{\CD@GG{#5}}\CD@tD}%
\def\DiagonalMap#1#2#3#4#5{\CD@BK\CD@LC\def\CD@OG{4}\let\CD@kD\CD@qK\let
\CD@ZD\CD@YD\def\CD@WK{2}\def\CD@aK{2}\def\CD@ZK{1}\def\CD@bK{1}\def\CD@QG{%
\CD@vF{#3}}\ifPositiveGradient\let\mv\raise\def\CD@oJ{\CD@vF{#5}}\def\CD@iD{%
\CD@vF{#4}}\def\CD@jD{\CD@vF{#2}}\def\CD@LE{\CD@vF{#1}}\else\let\mv\lower\def
\CD@oJ{\CD@vF{#1}}\def\CD@iD{\CD@vF{#2}}\def\CD@jD{\CD@vF{#4}}\def\CD@LE{%
\CD@vF{#5}}\fi\CD@tD}\def\CD@aE{-}\def\CD@AD{\empty}\def\CD@SH{\CD@EG\CD@bE
\CD@aE\@fillh}\def\CD@VH{\CD@EG\CD@MK\CD@KK\@fillv}\def\CD@EG#1#2#3#4#5{\def
\CD@CH{#5}\ifx\CD@CH#2\let#4#3\else\let#4\null\ifx\CD@CH\empty\else\ifx\CD@CH
\CD@AD\else\let#4\CD@CH\fi\fi\fi}\def\CD@UH#1{\hbox{\mathsurround\z@
\offinterlineskip\def\CD@CH{#1}\ifx\CD@CH\empty\else\ifx\CD@CH\CD@AD\else
\CD@k\mkern-1.5mu{\CD@CH}\mkern-1.5mu\CD@ND\fi\fi}}\def\CD@yD#1#2{\setbox#1=%
\hbox\bgroup\setbox0=\hbox{\CD@k\labelstyle()\CD@ND}%% ASCII round brackets
\setbox1=\null\ht1\ht0\dp1\dp0\box1 \kern.1\CD@zC\CD@k\bgroup\labelstyle
\aftergroup\CD@LD\CD@xD}\def\CD@LD{\CD@ND\kern.1\CD@zC\egroup\CD@tD}\def
\CD@xD{\futurelet\CD@EH\CD@mJ}\def\CD@mJ{%% qualifiers on label arguments
\catcase\bgroup:\CD@v;\catcase\egroup:\missing@label;\catcase\space:\CD@TF;%
\tokcase[:\CD@XF;%%]%ascii close square bracket
\default:\CD@zJ;\endswitch}\def\CD@v{\let\CD@MD\CD@c\let\CD@CH}\def\CD@zJ#1{%
\let\CD@UF\egroup{\let\actually@braces@missing@around@macro@in@label\CD@ZH
\let\CD@MD\CD@xC\let\CD@UF\CD@VF#1%
\actually@braces@missing@around@macro@in@label}\CD@UF}\def
\actually@braces@missing@around@macro@in@label{\let\CD@CH=}\def\missing@label
{\egroup\CD@YA{missing label}\CD@PE}\def\CD@xC{\egroup\missing@label}\outer
\def\CD@ZH{}\def\CD@UF{}\def\CD@VF{\CD@wC\CD@UF}\def\CD@MD{}\def\CD@XF{\let
\CD@N\CD@xD\get@square@arg\CD@AE}\CD@rG\CD@PE{The text which has just been
read is not allowed within map labels.}\def\CD@c{\egroup\CD@YA{missing \CD@yC
\space inserted after label}\CD@PE}\def\upper@label{\CD@oD\CD@yD6}\def
\lower@label{\def\positional@{\CD@@A\break@args}\CD@yD7}\def\middle@label{%
\CD@yD3}\CD@tG\CD@yE\CD@pD\CD@oD\def\CD@iF{\ifPositiveGradient\CD@tJ
\expandafter\upper@label\else\expandafter\lower@label\fi}\def\CD@iI{%
\ifPositiveGradient\CD@tJ\expandafter\lower@label\else\expandafter
\upper@label\fi}\def\positional@{\CD@gB{labels as positional arguments are
obsolete}\CD@yE\CD@tJ\expandafter\upper@label\else\expandafter\lower@label\fi
-}\def\CD@tD{\futurelet\CD@EH\switch@arg}\def\eat@space{\afterassignment
\CD@tD\let\CD@EH= }\def\CD@TF{\afterassignment\CD@xD\let\CD@EH= }\def\CD@BC{%
\get@round@pair\CD@uD}\def\CD@uD#1#2{\def\CD@WK{#1}\def\CD@aK{#2}\CD@tD}\def
\optional@{\let\CD@N\CD@tD\get@square@arg\CD@AE}\def\CD@JJ.{\CD@sC\CD@tD}\def
\CD@sC{\let\CD@iD\fill@dot\let\CD@jD\fill@dot\def\CD@MI{\let\CD@iD\dfdot\let
\CD@jD\dfdot}}\def\CD@MI{}\def\CD@@E#1,{\CD@nH#1,\begingroup\ifx\@name\CD@RD
\CD@FF\aftergroup\CD@e\fi\aftergroup\CD@jC\else\expandafter\def\expandafter
\CD@RF\expandafter{\csname\@name\endcsname}\expandafter\CD@vD\CD@RF\CD@KD\ifx
\CD@RF\empty\aftergroup\CD@pC\expandafter\aftergroup\csname\CD@FB\@name
\endcsname\expandafter\aftergroup\csname\CD@FB @\@name\endcsname\else\gdef
\CD@GE{#1}\CD@gB{\string\relax\space inserted before `[\CD@GE'}\message{(I was
trying to read this as a \CD@tA\ option.)}\aftergroup\CD@H\fi\fi\endgroup}%
\def\CD@vD#1#2\CD@KD{\def\CD@RF{#2}}\def\CD@jC{\let\CD@CH\CD@N\let\CD@N\relax
\CD@CH}\def\CD@H#1],{%% ASCII close square bracket
\CD@jC\relax\def\CD@RF{#1}\ifx\CD@RF\empty\def\CD@RF{[\CD@GE]}%
\else\def\CD@RF{[\CD@GE,#1]}%% ASCII open and close square bracket
\fi\CD@RF}\def\CD@pC#1#2{\ifx#2\CD@qK\ifx#1\CD@qK\CD@gB{option `\@name'
undefined}\else#1\fi\else\CD@FF\expandafter#2\CD@GK\CD@PK\else\CD@QK\fi\fi
\CD@DH}\CD@tG\CD@FF\CD@QK\CD@PK\def\CD@nH#1,{\CD@FF\ifx\CD@GK\CD@qK\CD@e\else
\expandafter\CD@oH\CD@GK,#1,(,),(,)[]%
\fi\fi\CD@FF\else\CD@mH#1==,\fi}\def\CD@e{\CD@gB{option `\@name' needs (x,y)
value}\CD@PK\let\@name\empty}\def\CD@mH#1=#2=#3,{\def\@name{#1}\def\CD@GK{#2}%
\def\CD@RF{#3}\ifx\CD@RF\empty\let\CD@GK\CD@qK\fi}%
\def\CD@oH#1(#2,#3)#4,(#5,#6)#7[]{\def\CD@GK{{#2}{#3}}\def\CD@RF{#1#4#5#6}%
\ifx\CD@RF\empty\def\CD@RF{#7}\ifx\CD@RF\empty\CD@e\fi\else\CD@e\fi}\def
\CD@FB{cds@}\let\CD@N\relax\def\CD@zD#1{\ifx\CD@GK\CD@qK\CD@gB{option `\@name
' needs a value}\else#1\CD@GK\relax\fi}\def\CD@BE#1#2{\ifx\CD@GK\CD@qK#1#2%
\relax\else#1\CD@GK\relax\fi}\def\cds@@showpair#1#2{\message{x=#1,y=#2}}\def
\cds@@diagonalbase#1#2{\edef\CD@ZK{#1}\edef\CD@bK{#2}}\def\CD@DI#1{\def\CD@CH
{#1}\CD@nF{@x}{cdps@#1}\ifx\CD@CH\empty\CD@f\CD@CH{cannot be used}\else\ifx
\CD@CH\relax\CD@f\CD@CH{unknown}\else\let\CD@IK\@x\fi\fi}\def\CD@f#1#2{\CD@gB
{PostScript translator `#1' #2}}\def\CD@PH{}\def\CD@PJ{\CD@fA\edef\CD@PH{%
\noexpand\CD@KB{\@name\space ignored within maths}}}\def\diagramstyle{\CD@cJ
\let\CD@N\relax\CD@CF\CD@AE\CD@AE}\CD@tG\CD@sE
\CD@SB\CD@RB\CD@tG\CD@qE\CD@EB\CD@DB\CD@tG\CD@oE\CD@pA\CD@oA\CD@tG\CD@iE
\CD@HA\CD@GA\CD@HA\CD@tG\CD@jE\CD@JA\CD@IA\CD@tG\CD@kE\CD@LA\CD@KA\CD@tG
\CD@EF\CD@DK\CD@CK\CD@tG\CD@rE\CD@JB\CD@IB\CD@tG\CD@mE\CD@gA\CD@fA\CD@tG
\CD@nE\CD@kA\CD@jA\CD@tG\CD@AF\CD@iG\CD@hG\CD@RC{cds@ }{}\CD@RC{cds@}{}\CD@RC
{cds@1em}{\CellSize1\CD@zC}\CD@RC{cds@1.5em}{\CellSize1.5\CD@zC}\CD@RC{cds@2%
em}{\CellSize2\CD@zC}\CD@RC{cds@2.5em}{\CellSize2.5\CD@zC}\CD@RC{cds@3em}{%
\CellSize3\CD@zC}\CD@RC{cds@3.5em}{\CellSize3.5\CD@zC}\CD@RC{cds@4em}{%
\CellSize4\CD@zC}\CD@RC{cds@4.5em}{\CellSize4.5\CD@zC}\CD@RC{cds@5em}{%
\CellSize5\CD@zC}\CD@RC{cds@6em}{\CellSize6\CD@zC}\CD@RC{cds@7em}{\CellSize7%
\CD@zC}\CD@RC{cds@8em}{\CellSize8\CD@zC}\def\cds@abut{\MapsAbut\dimen1\z@
\dimen5\z@}\def\cds@alignlabels{\CD@IA\CD@KA}\def\cds@amstex{\ifincommdiag
\CD@O\else\def\CD{\diagram[amstex]}%%ascii square brackets []
\fi\CD@T\catcode`\@\active}\def\cds@b{\let\CD@dB\CD@bB}\def\cds@balance{\let
\CD@hA\CD@AA}\let\cds@bottom\cds@b\def\cds@center{\cds@vcentre\cds@nobalance}%
\let\cds@centre\cds@center\def\cds@centerdisplay{\CD@HA\CD@PJ\cds@balance}%
\let\cds@centredisplay\cds@centerdisplay\def\cds@crab{\CD@BE\CD@DC{.5%
\PileSpacing}}\CD@RC{cds@crab-}{\CD@DC-.5\PileSpacing}\CD@RC{cds@crab+}{%
\CD@DC.5\PileSpacing}\CD@RC{cds@crab++}{\CD@DC1.5\PileSpacing}\CD@RC{cds@crab%
--}{\CD@DC-1.5\PileSpacing}\def\cds@defaultsize{\CD@BE{\let\CD@QC}{3em}\CD@NJ
}\def\cds@displayoneliner{\CD@DB}\let\cds@dotted\CD@sC\def\cds@dpi{\CD@RJ{1%
truein}}\def\cds@dpm{\CD@RJ{100truecm}}\let\CD@XA\CD@qK\def\cds@eqno{\let
\CD@XA\CD@GK\let\CD@EJ\empty}\def\cds@fixed{\CD@qA}\CD@tG\CD@fE\CD@J\CD@I\def
\cds@flushleft{\CD@I\CD@GA\CD@PJ\cds@nobalance\CD@BE\CD@nA\CD@nA}\def\cds@gap
{\CD@AJ\setbox3=\null\ht3=\CD@tI\dp3=\CD@sI\CD@BE{\wd3=}\MapShortFall} \def
\cds@grid{\ifx\CD@GK\CD@qK\let\h@grid\relax\let\v@grid\relax\else\CD@nF{%
h@grid}{cdgh@\CD@GK}\CD@nF{v@grid}{cdgv@\CD@GK}\ifx\h@grid\relax\CD@gB{%
unknown grid `\CD@GK'}\else\CD@WB\fi\fi}\let\h@grid\relax\let\v@grid\relax
\def\cds@gridx{\ifx\CD@GK\CD@qK\else\cds@grid\fi\let\CD@CH\h@grid\let\h@grid
\v@grid\let\v@grid\CD@CH}\def\cds@h{\CD@zD\DiagramCellHeight}\def\cds@hcenter
{\let\CD@hA\CD@aA}\let\cds@hcentre\cds@hcenter\def\cds@heads{\CD@BE{\let
\CD@sJ}\CD@sJ\CD@@J\CD@vE\else\ifx\CD@sJ\CD@eF\else\CD@MC\fi\fi}\let
\cds@height\cds@h\let\cds@hmiddle\cds@balance\def\cds@htriangleheight{\CD@BE
\DiagramCellHeight\DiagramCellHeight\DiagramCellWidth1.73205%
\DiagramCellHeight}\def\cds@htrianglewidth{\CD@BE\DiagramCellWidth
\DiagramCellWidth\DiagramCellHeight.57735\DiagramCellWidth}\CD@tG\CD@zE\CD@eE
\CD@dE\CD@eE\def\cds@hug{\CD@eE} \def\cds@inline{\CD@gA\let\CD@PH\empty}\def
\cds@inlineoneliner{\CD@EB}\CD@RC{cds@l>}{\CD@zD{\let\CD@RG}\dimen2=\CD@RG}%
\def\cds@labelstyle{\CD@zD{\let\labelstyle}}\def\cds@landscape{\CD@kA}\def
\cds@large{\CellSize5\CD@zC}\let\CD@EJ\empty\def\CD@FJ{\refstepcounter{%
equation}\def\CD@XA{\hbox{\@eqnnum}}}\def\cds@LaTeXeqno{\let\CD@EJ\CD@FJ}\def
\cds@lefteqno{\CD@pA}\def\cds@leftflush{\cds@flushleft\CD@J}\def
\cds@leftshortfall{\CD@zD{\dimen1 }}\def\cds@lowershortfall{%
\ifPositiveGradient\cds@leftshortfall\else\cds@rightshortfall\fi}\def
\cds@loose{\CD@VB}\def\cds@midhshaft{\CD@JA}\def\cds@midshaft{\CD@JA}\def
\cds@midvshaft{\CD@LA}\def\cds@moreoptions{\CD@@A}\let\cds@nobalance
\cds@hcenter\def\cds@nohcheck{\CD@HH}\def\cds@nohug{\CD@dE} \def
\cds@nooptions{\def\CD@aC{\CD@WD}}\let\cds@noorigin\cds@nobalance\def
\cds@nopixel{\CD@@I4\CD@XH\CD@cJ}\def\cds@UO{\CD@oK\global\let\CD@n\empty}%
\def\cds@UglyObsolete{\cds@UO\let\cds@PS\empty}\def\CD@rK#1{\CD@gB{option `#1%
' renamed as `UglyObsolete'}}\def\cds@noPostScript{\CD@rK{noPostScript}}\def
\cds@noPS{\CD@rK{noPostScript}}\def\cds@notextflow{\CD@RB}\def\cds@noTPIC{%
\CD@CK}\def\cds@objectstyle{\CD@zD{\let\objectstyle}}\def\cds@origin{\let
\CD@hA\CD@iB}\def\cds@p{\CD@zD\PileSpacing}\let\cds@pilespacing\cds@p\def
\cds@pixelsize{\CD@zD\CD@@I\CD@gI}\def\cds@portrait{\CD@jA}\def
\cds@PostScript{\CD@nK\global\let\CD@n\empty\CD@BE\CD@DI\empty}\def\cds@PS{%
\CD@nK\global\let\CD@n\empty}\CD@GF\CD@n{\typeout{\CD@tA: try the PostScript
option for better results}}\def\cds@repositionpullbacks{\let\make@pbk\CD@fH
\let\CD@qH\CD@pH}\def\cds@righteqno{\CD@oA}\def\cds@rightshortfall{\CD@zD{%
\dimen5 }}\def\cds@ruleaxis{\CD@zD{\let\axisheight}}\def\cds@cmex{\let\CD@GG
\CD@sB\let\CD@QJ\CD@CJ}\def\cds@s{\cds@height\DiagramCellWidth
\DiagramCellHeight}\def\cds@scriptlabels{\let\labelstyle\scriptstyle}\def
\cds@shortfall{\CD@zD\MapShortFall\dimen1\MapShortFall\dimen5\MapShortFall}%
\def\cds@showfirstpass{\CD@BE{\let\CD@nD}\z@}\def\cds@silent{\def\CD@KB##1{}%
\def\CD@gB##1{}}\let\cds@size\cds@s\def\cds@small{\CellSize2\CD@zC}\def
\cds@snake{\CD@BE\CD@eJ\z@}\def\cds@t{\let\CD@dB\CD@fB}\def\cds@textflow{%
\CD@SB\CD@PJ}\def\cds@thick{\let\CD@rF\tenlnw\CD@LF\CD@NC\CD@BE\MapBreadth{2%
\CD@LF}\CD@@J}\def\cds@thin{\let\CD@rF\tenln\CD@BE\MapBreadth{\CD@NC}\CD@@J}%
\def\cds@tight{\CD@WB}\let\cds@top\cds@t\def\cds@TPIC{\CD@DK}\def
\cds@uppershortfall{\ifPositiveGradient\cds@rightshortfall\else
\cds@leftshortfall\fi}\def\cds@vcenter{\let\CD@dB\CD@cB}\let\cds@vcentre
\cds@vcenter\def\cds@vtriangleheight{\CD@BE\DiagramCellHeight
\DiagramCellHeight\DiagramCellWidth.577035\DiagramCellHeight}\def
\cds@vtrianglewidth{\CD@BE\DiagramCellWidth\DiagramCellWidth
\DiagramCellHeight1.73205\DiagramCellWidth}\def\cds@vmiddle{\let\CD@dB\CD@eB}%
\def\cds@w{\CD@zD\DiagramCellWidth}\let\cds@width\cds@w\def\diagram{\relax
\protect\CD@bC}\def\enddiagram{\protect\CD@SG}\def\CD@bC{\CD@g\CD@uI
\incommdiagtrue\edef\CD@wI{\the\CD@NB}\global\CD@NB\z@\boxmaxdepth\maxdimen
\everycr{}\CD@sK\everymath{}\everyhbox{}\ifx\pdfsyncstop\CD@qK\else
\pdfsyncstop\fi\CD@aC}\def\CD@aC{\CD@y\let\CD@N\CD@ZC\CD@CF\CD@AE\CD@WD}\def
\CD@ZC{\CD@gE\expandafter\CD@aC\else\expandafter\CD@WD\fi}\def\CD@WD{\let
\CD@EH\relax\CD@nE\CD@vE\else\CD@hK\else\CD@KB{landscape ignored without
PostScript}\CD@jA\fi\fi\fi\CD@EJ\setbox2=\vbox\bgroup\CD@JF\CD@VD}\def\CD@cH{%
\CD@nE\CD@fB\else\CD@dB\fi\CD@hA\nointerlineskip\setbox0=\null\ht0-\CD@pI\dp0%
\CD@pI\wd0\CD@kI\box0 \global\CD@QA\CD@kF\global\CD@yA\CD@XB\ifx\CD@NK\CD@qK
\global\CD@RA\CD@kF\else\global\CD@RA\CD@NK\fi\egroup\CD@zF\CD@nE\setbox2=%
\hbox to\dp2{\vrule height\wd2 depth\CD@QA width\z@\global\CD@QA\ht2\ht2\z@
\dp2\z@\wd2\z@\CD@hK\CD@tK{q 0 1 -1 0 0 0 cm}\else\global\CD@iG\CD@IK{0 1
bturn}\fi\box2\CD@gK\hss}\CD@DB\fi\ifnum\CD@yA=1 \else\CD@DB\fi\global
\@ignorefalse\CD@mE\leavevmode\fi\ifvmode\CD@TA\else\ifmmode\CD@PH\CD@GI\else
\CD@qE\CD@gA\fi\ifinner\CD@gA\fi\CD@mE\CD@GI\else\CD@sE\CD@QB\else\CD@TA\fi
\fi\fi\fi\CD@dD}\def\CD@dD{\global\CD@NB\CD@wI\relax\CD@xE\global\CD@ID\else
\aftergroup\CD@mC\fi\if@ignore\aftergroup\ignorespaces\fi\CD@wC\ignorespaces}%
\def\CD@fB{\advance\CD@pI\dimen1\relax}\def\CD@eB{\advance\CD@pI.5\dimen1%
\relax}\def\CD@bB{}\def\CD@cB{\CD@fB\advance\CD@pI\CD@YB\divide\CD@pI2
\advance\CD@pI-\axisheight\relax}\def\CD@aA{}\def\CD@iB{\CD@kF\z@}\def\CD@AA{%
\ifdim\dimen2>\CD@kF\CD@kF\dimen2 \else\dimen2\CD@kF\CD@kI\dimen0 \advance
\CD@kI\dimen2 \fi}\def\CD@QB{\skip0\z@\relax\loop\skip1\lastskip\ifdim\skip1>%
\z@\unskip\advance\skip0\skip1 \repeat\vadjust{\prevdepth\dp\strutbox\penalty
\predisplaypenalty\vskip\abovedisplayskip\CD@UA\penalty\postdisplaypenalty
\vskip\belowdisplayskip}\ifdim\skip0=\z@\else\hskip\skip0 \global\@ignoretrue
\fi}\def\CD@TA{\CD@LG\kern-\displayindent\CD@UA\CD@LG\global\@ignoretrue}\def
\CD@UA{\hbox to\hsize{\CD@fE\ifdim\CD@RA=\z@\else\advance\CD@QA-\CD@RA\setbox
2=\hbox{\kern\CD@RA\box2}\fi\fi\setbox1=\hbox{\ifx\CD@XA\CD@qK\else\CD@k
\CD@XA\CD@ND\fi}\CD@oE\CD@iE\else\advance\CD@QA\wd1 \fi\wd1\z@\box1 \fi\dimen
0\wd2 \advance\dimen0\wd1 \advance\dimen0-\hsize\ifdim\dimen0>-\CD@nA\CD@HA
\fi\advance\dimen0\CD@QA\ifdim\dimen0>\z@\CD@KB{wider than the page by \the
\dimen0 }\CD@HA\fi\CD@iE\hss\else\CD@V\CD@QA\CD@nA\fi\CD@GI\hss\kern-\wd1\box
1 }}\def\CD@GI{\CD@AF\CD@@F\else\CD@SC\global\CD@hG\fi\fi\kern\CD@QA\box2 }%
\CD@tG\CD@wE\CD@YC\CD@XC\def\CD@JF{\CD@cJ\ifdim\DiagramCellHeight=-\maxdimen
\DiagramCellHeight\CD@QC\fi\ifdim\DiagramCellWidth=-\maxdimen
\DiagramCellWidth\CD@QC\fi\global\CD@XC\CD@IF\let\CD@FE\empty\let\CD@z\CD@Q
\let\overprint\CD@eH\let\CD@s\CD@rJ\let\enddiagram\CD@ED\let\\\CD@cC\let\par
\CD@jH\let\CD@MD\empty\let\switch@arg\CD@PB\let\shift\CD@iA\baselineskip
\DiagramCellHeight\lineskip\z@\lineskiplimit\z@\mathsurround\z@\tabskip\z@
\CD@OB}\def\CD@VD{\penalty-123 \begingroup\CD@jA\aftergroup\CD@K\halign
\bgroup\global\advance\CD@NB1 \vadjust{\penalty1}\global\CD@FA\z@\CD@OB\CD@j#%
#\CD@DD\CD@Q\CD@Q\CD@OI\CD@j##\CD@DD\cr}\def\CD@ED{\CD@MD\CD@GD\crcr\egroup
\global\CD@JD\endgroup}\def\CD@j{\global\advance\CD@FA1 \futurelet\CD@EH\CD@i
}\def\CD@i{\ifx\CD@EH\CD@DD\CD@tJ\hskip1sp plus 1fil \relax\let\CD@DD\relax
\CD@vI\else\hfil\CD@k\objectstyle\let\CD@FE\CD@d\fi}\def\CD@DD{\CD@MD\relax
\CD@yI\CD@vI\global\CD@QA\CD@iA\penalty-9993 \CD@ND\hfil\null\kern-2\CD@QA
\null}\def\CD@cC{\cr}\def\across#1{\span\omit\mscount=#1 \global\advance
\CD@FA\mscount\global\advance\CD@FA\m@ne\CD@sF\ifnum\mscount>2 \CD@fJ\repeat
\ignorespaces}\def\CD@fJ{\relax\span\omit\advance\mscount\m@ne}\def\CD@qJ{%
\ifincommdiag\ifx\CD@iD\@fillh\ifx\CD@jD\@fillh\ifdim\dimen3>\z@\else\ifdim
\dimen2>93\CD@@I\ifdim\dimen2>18\p@\ifdim\CD@LF>\z@\count@\CD@bJ\advance
\count@\m@ne\ifnum\count@<\z@\count@20\let\CD@aJ\CD@uJ\fi\xdef\CD@bJ{\the
\count@}\fi\fi\fi\fi\fi\fi\fi}\def\CD@cG#1{\vrule\horizhtdp width#1\dimen@
\kern2\dimen@}\def\CD@uJ{\rlap{\dimen@\CD@@I\CD@V\dimen@{.182\p@}\CD@zH
\dimen@\advance\CD@tI\dimen@\CD@cG0\CD@cG0\CD@cG2\CD@cG6\CD@cG6\CD@cG2\CD@cG0%
\CD@cG0\CD@cG2\CD@cG6\CD@cG0\CD@cG0\CD@cG2\CD@cG2\CD@cG6\CD@cG0\CD@cG0\CD@cG2%
\CD@cG6\CD@cG2\CD@cG2\CD@cG0\CD@cG0}}\def\CD@bJ{10}\def\CD@aJ{}\def\CD@XD{%
\CD@gE\CD@TB\fi\CD@x\CD@WF\CD@HI}\def\CD@x{\CD@QJ\CD@DC\CD@MJ\ifdim\CD@DC=\z@
\else\CD@pF\CD@DC\fi\ifvoid3 \setbox3=\null\ht3\CD@tI\dp3\CD@sI\else\CD@V{\ht
3}\CD@tI\CD@V{\dp3}\CD@sI\fi\dimen3=.5\wd3 \ifdim\dimen3=\z@\CD@tE\else\dimen
3-\CD@XH\fi\else\CD@TB\fi\CD@V{\dimen2}{\wd7}\CD@V{\dimen2}{\wd6}\CD@qJ
\advance\dimen2-2\dimen3 \dimen4.5\dimen2 \dimen2\dimen4 \advance\dimen2%
\CD@eJ\advance\dimen4-\CD@eJ\advance\dimen2-\wd1 \advance\dimen4-\wd5 \ifvoid
2 \else\CD@V{\ht3}{\ht2}\CD@V{\dp3}{\dp2}\CD@V{\dimen2}{\wd2}\fi\ifvoid4 \else
\CD@V{\ht3}{\ht4}\CD@V{\dp3}{\dp4}\CD@V{\dimen4}{\wd4}\fi\advance\skip2\dimen
2 \advance\skip4\dimen4 \CD@tE\advance\skip2\skip4 \dimen0\dimen5 \advance
\dimen0\wd5 \skip3-\skip4 \advance\skip3-\dimen0 \let\CD@jD\empty\else\skip3%
\z@\relax\dimen0\z@\fi}\def\CD@WF{\offinterlineskip\lineskip.2\CD@zC\ifvoid6
\else\setbox3=\vbox{\hbox to2\dimen3{\hss\box6\hss}\box3}\fi\ifvoid7 \else
\setbox3=\vtop{\box3 \hbox to2\dimen3{\hss\box7\hss}}\fi}\def\CD@HI{\kern
\dimen1 \box1 \CD@aJ\CD@iD\hskip\skip2 \kern\dimen0 \ifincommdiag\CD@jE
\penalty1\fi\kern\dimen3 \penalty\CD@GB\hskip\skip3 \null\kern-\dimen3 \else
\hskip\skip3 \fi\box3 \CD@jD\hskip\skip4 \box5 \kern\dimen5}\def\CD@MF{\ifnum
\CD@LH>\CD@TC\CD@V{\dimen1}\objectheight\CD@V{\dimen5}\objectheight\else\CD@V
{\dimen1}\objectwidth\CD@V{\dimen5}\objectwidth\fi}\def\CD@Y{\begingroup
\ifdim\dimen7=\z@\kern\dimen8 \else\ifdim\dimen6=\z@\kern\dimen9 \else\dimen5%
\dimen6 \dimen6\dimen9 \CD@KJ\dimen4\dimen2 \CD@dG{\dimen4}\dimen6\dimen5
\dimen7\dimen8 \CD@KJ\CD@iC{\dimen2}\ifdim\dimen2<\dimen4 \kern\dimen2 \else
\kern\dimen4 \fi\fi\fi\endgroup}\def\CD@jJ{\CD@JI\setbox\z@\hbox{\lower
\axisheight\hbox to\dimen2{\CD@DF\ifPositiveGradient\dimen8\ht\CD@MH\dimen9%
\CD@mI\else\dimen8\dp3 \dimen9\dimen1 \fi\else\dimen8 \ifPositiveGradient
\objectheight\else\z@\fi\dimen9\objectwidth\fi\advance\dimen8
\ifPositiveGradient-\fi\axisheight\CD@Y\unhbox\z@\CD@DF\ifPositiveGradient
\dimen8\dp3 \dimen9\dimen0 \else\dimen8\ht\CD@MH\dimen9\CD@mF\fi\else\dimen8
\ifPositiveGradient\z@\else\objectheight\fi\dimen9\objectwidth\fi\advance
\dimen8 \ifPositiveGradient\else-\fi\axisheight\CD@Y}}}\def\CD@bD{\dimen6
\CD@aK\DiagramCellHeight\dimen7 \CD@WK\DiagramCellWidth\CD@jJ
\ifPositiveGradient\advance\dimen7-\CD@ZK\DiagramCellWidth\else\dimen7 \CD@ZK
\DiagramCellWidth\dimen6\z@\fi\advance\dimen6-\CD@bK\DiagramCellHeight\CD@mK
\setbox0=\rlap{\kern-\dimen7 \lower\dimen6\box\z@}\ht0\z@\dp0\z@\raise
\axisheight\box0 }\def\CD@mK{\setbox0\hbox{\ht\z@\z@\dp\z@\z@\wd\z@\z@\CD@hK
\expandafter\CD@tK{q \CD@eK\space\CD@lK\space\CD@kK\space\CD@eK\space0 0 cm}%
\else\global\CD@iG\CD@eD{\the\CD@TC\space\ifPositiveGradient\else-\fi\the
\CD@LH\space bturn}\fi\box\z@\CD@gK}}\def\CD@vB{\advance\CD@hF-\CD@mI\CD@wJ
\CD@hF\advance\CD@wJ\CD@hI\ifvoid\CD@sH\ifdim\CD@wJ<.1em\ifnum\CD@gD=\@m\else
\CD@aG h\CD@wJ<.1em:objects overprint:\CD@FA\CD@gD\fi\fi\else\ifhbox\CD@sH
\CD@SK\else\CD@TK\fi\advance\CD@wJ\CD@mI\CD@bH{-\CD@mI}{\box\CD@sH}{\CD@wJ}%
\z@\fi\CD@hF-\CD@mF\CD@gD\CD@FA\CD@hI\z@}\def\CD@SK{\setbox\CD@sH=\hbox{%
\unhbox\CD@sH\unskip\unpenalty}\setbox\CD@tH=\hbox{\unhbox\CD@tH\unskip
\unpenalty}\setbox\CD@sH=\hbox to\CD@wJ{\CD@OA\wd\CD@sH\unhbox\CD@sH\CD@PA
\lastkern\unkern\ifdim\CD@PA=\z@\CD@UB\advance\CD@OA-\wd\CD@tH\else\CD@TB\fi
\ifnum\lastpenalty=\z@\else\CD@JA\unpenalty\fi\kern\CD@PA\ifdim\CD@hF<\CD@OA
\CD@JA\fi\ifdim\CD@hI<\wd\CD@tH\CD@JA\fi\CD@jE\CD@hI\CD@wJ\advance\CD@hI-%
\CD@OA\advance\CD@hI\wd\CD@tH\ifdim\CD@hI<2\wd\CD@tH\CD@aG h\CD@hI<2\wd\CD@tH
:arrow too short:\CD@FA\CD@gD\fi\divide\CD@hI\tw@\CD@hF\CD@wJ\advance\CD@hF-%
\CD@hI\fi\CD@tE\kern-\CD@hI\fi\hbox to\CD@hI{\unhbox\CD@tH}\CD@HG}}\CD@tG
\ifinpile\inpiletrue\inpilefalse\inpilefalse\def\pile{\protect\CD@UJ\protect
\CD@uH}\def\CD@uH#1{\CD@l#1\CD@QD}\def\CD@UJ{\CD@nB{pile}\setbox0=\vtop
\bgroup\aftergroup\CD@lD\inpiletrue\let\CD@FE\empty\let\pile\CD@KF\let\CD@QD
\CD@PD\let\CD@GD\CD@FD\CD@yH\baselineskip.5\PileSpacing\lineskip.1\CD@zC
\relax\lineskiplimit\lineskip\mathsurround\z@\tabskip\z@\let\\\CD@wH}\def
\CD@l{\CD@DE\CD@YF\empty\halign\bgroup\hfil\CD@k\let\CD@FE\CD@d\let\\\CD@vH##%
\CD@MD\CD@ND\hfil\CD@Q\CD@R##\cr}\CD@rG\CD@NE{pile only allows one column.}%
\CD@rG\CD@UE{you left it out!}\def\CD@R{\CD@QD\CD@Q\relax\CD@YA{missing \CD@yC
\space inserted after \string\pile}\CD@NE}\def\CD@PD{\CD@MD\crcr\egroup
\egroup}\def\CD@GD{\CD@MD}\def\CD@FD{\CD@MD\relax\CD@QD\CD@YA{missing \CD@yC
\space inserted between \string\pile\space and \CD@HD}\CD@UE}\def\CD@QD{%
\CD@MD}\def\CD@lD{\vbox{\dimen1\dp0 \unvbox0 \setbox0=\lastbox\advance\dimen1%
\dp0 \nointerlineskip\box0 \nointerlineskip\setbox0=\null\dp0.5\dimen1\ht0-%
\dp0 \box0}\ifincommdiag\CD@tJ\penalty-9998 \fi\xdef\CD@YF{pile}}\def\CD@vH{%
\cr}\def\CD@wH{\noalign{\skip@\prevdepth\advance\skip@-\baselineskip
\prevdepth\skip@}}\def\CD@KF#1{#1}\def\CD@TK{\setbox\CD@sH=\vbox{\unvbox
\CD@sH\setbox1=\lastbox\setbox0=\box\voidb@x\CD@tF\setbox\CD@sH=\lastbox
\ifhbox\CD@sH\CD@rC\repeat\unvbox0 \global\CD@QA\CD@ZE}\CD@ZE\CD@QA}\def
\CD@rC{\CD@jE\setbox\CD@sH=\hbox{\unhbox\CD@sH\unskip\setbox\CD@sH=\lastbox
\unskip\unhbox\CD@sH}\ifdim\CD@wJ<\wd\CD@sH\CD@aG h\CD@wJ<\wd\CD@sH:arrow in
pile too short:\CD@FA\CD@gD\else\setbox\CD@sH=\hbox to\CD@wJ{\unhbox\CD@sH}%
\fi\else\CD@gJ\fi\setbox0=\vbox{\box\CD@sH\nointerlineskip\ifvoid0 \CD@tJ\box
1 \else\vskip\skip0 \unvbox0 \fi}\skip0=\lastskip\unskip}\def\CD@gJ{\penalty7
\noindent\unhbox\CD@sH\unskip\setbox\CD@sH=\lastbox\unskip\unhbox\CD@sH
\endgraf\setbox\CD@tH=\lastbox\unskip\setbox\CD@tH=\hbox{\CD@JG\unhbox\CD@tH
\unskip\unskip\unpenalty}\ifcase\prevgraf\cd@shouldnt P\or\ifdim\CD@wJ<\wd
\CD@tH\CD@aG h\CD@wJ<\wd\CD@sH:object in pile too wide:\CD@FA\CD@gD\setbox
\CD@sH=\hbox to\CD@wJ{\hss\unhbox\CD@tH\hss}\else\setbox\CD@sH=\hbox to\CD@wJ
{\hss\kern\CD@hF\unhbox\CD@tH\kern\CD@hI\hss}\fi\or\setbox\CD@sH=\lastbox
\unskip\CD@SK\else\cd@shouldnt Q\fi\unskip\unpenalty}\def\CD@cD{\CD@MJ\ifvoid
3 \setbox3=\null\ht3\axisheight\dp3-\ht3 \dimen3.5\CD@LF\else\dimen4\dp3
\dimen3.5\wd3 \setbox3=\CD@GG{\box3}\dp3\dimen4 \ifdim\ht3=-\dp3 \else\CD@TB
\fi\fi\dimen0\dimen3 \advance\dimen0-.5\CD@LF\setbox0\null\ht0\ht3\dp0\dp3\wd
0\wd3 \ifvoid6\else\setbox6\hbox{\unhbox6\kern\dimen0\kern2pt}\dimen0\wd6 \fi
\ifvoid7\else\setbox7\hbox{\kern2pt\kern\dimen3\unhbox7}\dimen3\wd7 \fi
\setbox3\hbox{\ifvoid6\else\kern-\dimen0\unhbox6\fi\unhbox3 \ifvoid7\else
\unhbox7\kern-\dimen3\fi}\ht3\ht0\dp3\dp0\wd3\wd0 \CD@tE\dimen4=\ht\CD@MH
\advance\dimen4\dp5 \advance\dimen4\dimen1 \let\CD@jD\empty\else\dimen4\ht3
\fi\setbox0\null\ht0\dimen4 \offinterlineskip\setbox8=\vbox spread2ex{\kern
\dimen5 \box1 \CD@iD\vfill\CD@tE\else\kern\CD@eJ\fi\box0}\ht8=\z@\setbox9=%
\vtop spread2ex{\kern-\ht3 \kern-\CD@eJ\box3 \CD@jD\vfill\box5 \kern\dimen1}%
\dp9=\z@\hskip\dimen0plus.0001fil \box9 \kern-\CD@LF\box8 \CD@kE\penalty2 \fi
\CD@tE\penalty1 \fi\kern\PileSpacing\kern-\PileSpacing\kern-.5\CD@LF\penalty
\CD@GB\null\kern\dimen3}\def\CD@cI{\ifhbox\CD@VA\CD@KB{clashing verticals}\ht
\CD@MH.5\dp\CD@VA\dp\CD@MH-\ht5 \CD@yB\ht\CD@MH\z@\dp\CD@MH\z@\fi\dimen1\dp
\CD@VA\CD@xA\prevgraf\unvbox\CD@VA\CD@wA\lastpenalty\unpenalty\setbox\CD@VA=%
\null\setbox\CD@lI=\hbox{\CD@JG\unhbox\CD@lI\unskip\unpenalty\dimen0\lastkern
\unkern\unkern\unkern\kern\dimen0 \CD@HG}\setbox\CD@lF=\hbox{\unhbox\CD@lF
\dimen0\lastkern\unkern\unkern\global\CD@QA\lastkern\unkern\kern\dimen0 }%
\CD@tF\ifnum\CD@xA>4 \CD@zI\repeat\unskip\unskip\advance\CD@mF.5\wd\CD@VA
\advance\CD@mF\wd\CD@lF\advance\CD@mI.5\wd\CD@VA\advance\CD@mI\wd\CD@lI\ifnum
\CD@FA=\CD@lA\CD@OA.5\wd\CD@VA\edef\CD@NK{\the\CD@OA}\fi\setbox\CD@VA=\hbox{%
\kern-\CD@mF\box\CD@lF\unhbox\CD@VA\box\CD@lI\kern-\CD@mI\penalty\CD@wA
\penalty\CD@NB}\ht\CD@VA\dimen1 \dp\CD@VA\z@\wd\CD@VA\CD@tB\CD@vB}\def\CD@zI{%
\ifdim\wd\CD@lF<\CD@QA\setbox\CD@lF=\hbox to\CD@QA{\CD@JG\unhbox\CD@lF}\fi
\advance\CD@xA\m@ne\setbox\CD@VA=\hbox{\box\CD@lF\unhbox\CD@VA}\unskip\setbox
\CD@lF=\lastbox\setbox\CD@lF=\hbox{\unhbox\CD@lF\unskip\unpenalty\dimen0%
\lastkern\unkern\unkern\global\CD@QA\lastkern\unkern\kern\dimen0 }}\def\CD@yB
{\dimen1\dp\CD@VA\ifhbox\CD@VA\CD@xB\else\CD@zB\fi\setbox\CD@VA=\vbox{%
\penalty\CD@NB}\dp\CD@VA-\dp\CD@MH\wd\CD@VA\CD@tB}\def\CD@zB{\unvbox\CD@VA
\CD@wA\lastpenalty\unpenalty\ifdim\dimen1<\ht\CD@MH\CD@aG v\dimen1<\ht\CD@MH:%
rows overprint:\CD@NB\CD@wA\fi}\def\CD@xB{\dimen0=\ht\CD@VA\setbox\CD@VA=%
\hbox\bgroup\advance\dimen1-\ht\CD@MH\unhbox\CD@VA\CD@xA\lastpenalty
\unpenalty\CD@wA\lastpenalty\unpenalty\global\CD@RA-\lastkern\unkern\setbox0=%
\lastbox\CD@tF\setbox\CD@VA=\hbox{\box0\unhbox\CD@VA}\setbox0=\lastbox\ifhbox
0 \CD@kJ\repeat\global\CD@SA-\lastkern\unkern\global\CD@QA\CD@JK\unhbox\CD@VA
\egroup\CD@JK\CD@QA\CD@bH{\CD@SA}{\box\CD@VA}{\CD@RA}{\dimen1}}\def\CD@kJ{%
\setbox0=\hbox to\wd0\bgroup\unhbox0 \unskip\unpenalty\dimen7\lastkern\unkern
\ifnum\lastpenalty=1 \unpenalty\CD@UB\else\CD@TB\fi\ifnum\lastpenalty=2
\unpenalty\dimen2.5\dimen0\advance\dimen2-.5\dimen1\advance\dimen2-%
\axisheight\else\dimen2\z@\fi\setbox0=\lastbox\dimen6\lastkern\unkern\setbox1%
=\lastbox\setbox0=\vbox{\unvbox0 \CD@tE\kern-\dimen1 \else\ifdim\dimen2=\z@
\else\kern\dimen2 \fi\fi}\ifdim\dimen0<\ht0 \CD@aG v\dimen0<\ht0:upper part of
vertical too short:{\CD@tE\CD@NB\else\CD@wA\fi}\CD@xA\else\setbox0=\vbox to%
\dimen0{\unvbox0}\fi\setbox1=\vtop{\unvbox1}\ifdim\dimen1<\dp1 \CD@aG v\dimen
1<\dp1:lower part of vertical too short:\CD@NB\CD@wA\else\setbox1=\vtop to%
\dimen1{\ifdim\dimen2=\z@\else\kern-\dimen2 \fi\unvbox1 }\fi\box1 \kern\dimen
6 \box0 \kern\dimen7 \CD@HG\global\CD@QA\CD@JK\egroup\CD@JK\CD@QA\relax}%
\countdef\CD@u=14 \newcount\CD@CA\newcount\CD@XB\newcount\CD@NB\let\CD@LB
\insc@unt\newcount\CD@FA\newcount\CD@lA\let\CD@mA\CD@XB\newcount\CD@MB\CD@tG
\CD@DF\CD@bI\CD@aI\CD@aI\def\CD@nD{-1}\def\CD@K{\ifnum\CD@nD<\z@\else
\begingroup\scrollmode\showboxdepth\CD@nD\showboxbreadth\maxdimen\showlists
\endgroup\fi\CD@bI\CD@zF\CD@CA=\CD@u\advance\CD@CA1 \CD@XB=\CD@CA\ifnum\CD@NB
=1 \CD@JA\fi\advance\CD@XB\CD@NB\dimen1\z@\skip0\z@\count@=\insc@unt\advance
\count@\CD@u\divide\count@2 \ifnum\CD@XB>\count@\CD@KB{The diagram has too
many rows! It can't be reformatted.}\else\CD@NG\CD@WI\fi\CD@cH}\def\CD@NG{%
\CD@NB\CD@CA\CD@uF\ifnum\CD@NB<\CD@XB\setbox\CD@NB\box\voidb@x\advance\CD@NB1%
\relax\repeat\CD@NB\CD@CA\skip\z@\z@\CD@uF\CD@GB\lastpenalty\unpenalty\ifnum
\CD@GB>\z@\CD@KE\repeat\ifnum\CD@GB=-123 \CD@tJ\unpenalty\else\cd@shouldnt D%
\fi\ifx\v@grid\relax\else\CD@NB\CD@XB\advance\CD@NB\m@ne\expandafter\CD@VJ
\v@grid\fi\CD@MB\CD@mA\CD@tB\z@\CD@XG\ifx\h@grid\relax\else\expandafter\CD@LJ
\h@grid\fi\count@\CD@XB\advance\count@\m@ne\CD@YB\ht\count@}\def\CD@KE{%
\ifcase\CD@GB\or\CD@MG\else\CD@uA-\lastpenalty\unpenalty\CD@vA\lastpenalty
\unpenalty\setbox0=\lastbox\CD@WG\fi\CD@wD}\def\CD@wD{\skip1\lastskip\unskip
\advance\skip0\skip1 \ifdim\skip1=\z@\else\expandafter\CD@wD\fi}\def\CD@MG{%
\setbox0=\lastbox\CD@pI\dp0 \advance\CD@pI\skip\z@\skip\z@\z@\advance\CD@NF
\CD@pI\CD@uE\ifnum\CD@NB>\CD@CA\CD@NF\DiagramCellHeight\CD@pI\CD@NF\advance
\CD@pI-\CD@qI\fi\fi\CD@qI\ht0 \CD@NF\CD@qI\setbox\CD@NB\hbox{\unhbox\CD@NB
\unhbox0}\dp\CD@NB\CD@pI\ht\CD@NB\CD@qI\advance\CD@NB1 }\def\CD@WG{\ifnum
\CD@uA<\z@\advance\CD@uA\CD@XB\ifnum\CD@uA<\CD@CA\CD@UG\else\CD@OA\dp\CD@uA
\CD@PA\ht\CD@uA\setbox\CD@uA\hbox{\box\z@\penalty\CD@vA\penalty\CD@GB\unhbox
\CD@uA}\dp\CD@uA\CD@OA\ht\CD@uA\CD@PA\fi\else\CD@UG\fi}\def\CD@UG{\CD@KB{%
diagonal goes outside diagram (lost)}}\def\CD@fI{\advance\CD@uA\CD@XB\ifnum
\CD@uA<\CD@CA\CD@UG\else\ifnum\CD@uA=\CD@NB\CD@VG\else\ifnum\CD@uA>\CD@NB
\cd@shouldnt M\else\CD@OA\dp\CD@uA\CD@PA\ht\CD@uA\setbox\CD@uA\hbox{\box\z@
\penalty\CD@vA\penalty\CD@GB\unhbox\CD@uA}\dp\CD@uA\CD@OA\ht\CD@uA\CD@PA\fi
\fi\fi}\def\CD@WI{\CD@t\CD@AJ\setbox\CD@PC=\hbox{\CD@k A\@super f\CD@lJ f%
\CD@ND}\CD@ZE\z@\CD@JK\z@\CD@kI\z@\CD@kF\z@\CD@NB=\CD@XB\CD@NF\z@\CD@uB\z@
\CD@uF\ifnum\CD@NB>\CD@CA\advance\CD@NB\m@ne\CD@qI\ht\CD@NB\CD@pI\dp\CD@NB
\advance\CD@NF\CD@qI\CD@rI\advance\CD@uB\CD@NF\CD@KC\CD@ZI\CD@w\ht\CD@NB
\CD@qI\dp\CD@NB\CD@pI\nointerlineskip\box\CD@NB\CD@NF\CD@pI\setbox\CD@NB\null
\ht\CD@NB\CD@uB\repeat\CD@wB\nointerlineskip\box\CD@NB\CD@gG\CD@ZE
\DiagramCellWidth{width}\CD@gG\CD@JK\DiagramCellHeight{height}\CD@VA\CD@LB
\advance\CD@VA-\CD@lA\advance\CD@VA\m@ne\advance\CD@VA\CD@mA\dimen0\wd\CD@VA
\CD@tI\axisheight\dimen1\CD@uB\advance\dimen1-\CD@YB\dimen2\CD@kI\advance
\dimen2-\dimen0 \advance\CD@XB-\CD@CA\advance\CD@LB-\CD@lA}\count@\year
\multiply\count@12 \advance\count@\month\ifnum\count@>24254 \loop\iftrue
\repeat\fi\def\CD@wB{\CD@qI-\CD@NF\CD@pI\CD@NF
\setbox\CD@MH=\null\dp\CD@MH\CD@NF\ht\CD@MH-\CD@NF\CD@mF\z@\CD@mI\z@\CD@lA
\CD@LB\advance\CD@lA-\CD@MB\advance\CD@lA\CD@mA\CD@FA\CD@LB\CD@VA\CD@MB\CD@sF
\ifnum\CD@FA>\CD@lA\advance\CD@FA\m@ne\advance\CD@VA\m@ne\CD@tB\wd\CD@VA
\setbox\CD@FA=\box\voidb@x\CD@yB\repeat\CD@w\ht\CD@NB\CD@qI\dp\CD@NB\CD@pI}%
\def\CD@gG#1#2#3{\ifdim#1>.01\CD@zC\CD@PA#2\relax\advance\CD@PA#1\relax
\advance\CD@PA.99\CD@zC\count@\CD@PA\divide\count@\CD@zC\CD@KB{increase cell #%
3 to \the\count@ em}\fi}\def\CD@rI{\CD@FA=\CD@LB\penalty4 \noindent\unhbox
\CD@NB\CD@sF\unskip\setbox0=\lastbox\ifhbox0 \advance\CD@FA\m@ne\setbox\CD@FA
\hbox to\wd0{\null\penalty-9990\null\unhbox0}\repeat\CD@lA\CD@FA\advance
\CD@FA\CD@MB\advance\CD@FA-\CD@mA\ifnum\CD@FA<\CD@LB\count@\CD@FA\advance
\count@\m@ne\dimen0=\wd\count@\count@\CD@MB\advance\count@\m@ne\CD@tB\wd
\count@\CD@sF\ifnum\CD@FA<\CD@LB\CD@DJ\CD@XG\dimen0\wd\CD@FA\advance\CD@FA1
\repeat\fi\CD@sF\CD@GB\lastpenalty\unpenalty\ifnum\CD@GB>\z@\CD@vA
\lastpenalty\unpenalty\CD@VG\repeat\endgraf\unskip\ifnum\lastpenalty=4
\unpenalty\else\cd@shouldnt S\fi}\def\CD@VG{\advance\CD@vA\CD@lA\advance
\CD@vA\m@ne\setbox0=\lastbox\ifnum\CD@vA<\CD@LB\setbox\CD@vA\hbox{\box0%
\penalty\CD@GB\unhbox\CD@vA}\else\CD@UG\fi}\def\CD@bG{}\CD@tG\CD@uE\CD@WB
\CD@VB\def\CD@DJ{\advance\dimen0\wd\CD@FA\divide\dimen0\tw@\CD@uE\dimen0%
\DiagramCellWidth\else\CD@V{\dimen0}\DiagramCellWidth\CD@pJ\fi\advance\CD@tB
\dimen0 }\def\CD@XG{\setbox\CD@MB=\vbox{}\dp\CD@MB=\CD@uB\wd\CD@MB\CD@tB
\advance\CD@MB1 }\def\CD@LJ#1,{\def\CD@GK{#1}\ifx\CD@GK\CD@RD\else\advance
\CD@tB\CD@GK\DiagramCellWidth\CD@XG\expandafter\CD@LJ\fi}\def\CD@VJ#1,{\def
\CD@GK{#1}\ifx\CD@GK\CD@RD\else\ifnum\CD@NB>\CD@CA\CD@NF\CD@GK
\DiagramCellHeight\advance\CD@NF-\dp\CD@NB\advance\CD@NB\m@ne\ht\CD@NB\CD@NF
\fi\expandafter\CD@VJ\fi}\def\CD@pJ{\CD@wE\CD@OA\dimen0 \advance\CD@OA-%
\DiagramCellWidth\ifdim\CD@OA>2\MapShortFall\CD@KB{badly drawn diagonals (see
manual)}\let\CD@pJ\empty\fi\else\let\CD@pJ\empty\fi}\def\CD@KC{\CD@VA\CD@mA
\CD@sF\ifnum\CD@VA<\CD@MB\dimen0\dp\CD@VA\advance\dimen0\CD@NF\dp\CD@VA\dimen
0 \advance\CD@VA1 \repeat}\def\CD@bH#1#2#3#4{\ifnum\CD@FA<\CD@LB\CD@OA=#1%
\relax\setbox\CD@FA=\hbox{\setbox0=#2\dimen7=#4\relax\dimen8=#3\relax\ifhbox
\CD@FA\unhbox\CD@FA\advance\CD@OA-\lastkern\unkern\fi\ifdim\CD@OA=\z@\else
\kern-\CD@OA\fi\raise\dimen7\box0 \kern-\dimen8 }\ifnum\CD@FA=\CD@lA\CD@V
\CD@kF\CD@OA\fi\else\cd@shouldnt O\fi}\def\CD@w{\setbox\CD@NB=\hbox{\CD@FA
\CD@lA\CD@VA\CD@mA\CD@PA\z@\relax\CD@sF\ifnum\CD@FA<\CD@LB\CD@tB\wd\CD@VA
\relax\CD@eI\advance\CD@FA1 \advance\CD@VA1 \repeat}\CD@V\CD@kI{\wd\CD@NB}\wd
\CD@NB\z@}\def\CD@eI{\ifhbox\CD@FA\CD@OA\CD@tB\relax\advance\CD@OA-\CD@PA
\relax\ifdim\CD@OA=\z@\else\kern\CD@OA\fi\CD@PA\CD@tB\advance\CD@PA\wd\CD@FA
\relax\unhbox\CD@FA\advance\CD@PA-\lastkern\unkern\fi}\def\CD@ZI{\setbox
\CD@sH=\box\voidb@x\CD@VA=\CD@MB\CD@FA\CD@LB\CD@VA\CD@mA\advance\CD@VA\CD@FA
\advance\CD@VA-\CD@lA\advance\CD@VA\m@ne\CD@tB\wd\CD@VA\count@\CD@LB\advance
\count@\m@ne\CD@hF.5\wd\count@\advance\CD@hF\CD@tB\CD@A\m@ne\CD@gD\@m\CD@sF
\ifnum\CD@FA>\CD@lA\advance\CD@FA\m@ne\advance\CD@hF-\CD@tB\CD@PI\wd\CD@VA
\CD@tB\advance\CD@hF\CD@tB\advance\CD@VA\m@ne\CD@tB\wd\CD@VA\repeat\CD@mF
\CD@kF\CD@mI-\CD@mF\CD@vB}\newcount\CD@GB\def\CD@s{}\def\CD@t{\mathsurround
\z@\hsize\z@\rightskip\z@ plus1fil minus\maxdimen\parfillskip\z@\linepenalty
9000 \looseness0 \hfuzz\maxdimen\hbadness10000 \clubpenalty0 \widowpenalty0
\displaywidowpenalty0 \interlinepenalty0 \predisplaypenalty0
\postdisplaypenalty0 \interdisplaylinepenalty0 \interfootnotelinepenalty0
\floatingpenalty0 \brokenpenalty0 \everypar{}\leftskip\z@\parskip\z@
\parindent\z@\pretolerance10000 \tolerance10000 \hyphenpenalty10000
\exhyphenpenalty10000 \binoppenalty10000 \relpenalty10000 \adjdemerits0
\doublehyphendemerits0 \finalhyphendemerits0 \baselineskip\z@\CD@IA\prevdepth
\z@}\newbox\CD@KG\newbox\CD@IG\def\CD@JG{\unhcopy\CD@KG}\def\CD@HG{\unhcopy
\CD@IG}\def\CD@iJ{\hbox{}\penalty1\nointerlineskip}\def\CD@PI{\penalty5
\noindent\setbox\CD@MH=\null\CD@mF\z@\CD@mI\z@\ifnum\CD@FA<\CD@LB\ht\CD@MH\ht
\CD@FA\dp\CD@MH\dp\CD@FA\unhbox\CD@FA\skip0=\lastskip\unskip\else\CD@OK\skip0%
=\z@\fi\endgraf\ifcase\prevgraf\cd@shouldnt Y \or\cd@shouldnt Z \or\CD@RI\or
\CD@XI\else\CD@QI\fi\unskip\setbox0=\lastbox\unskip\unskip\unpenalty\noindent
\unhbox0\setbox0\lastbox\unpenalty\unskip\unskip\unpenalty\setbox0\lastbox
\CD@tF\CD@GB\lastpenalty\unpenalty\ifnum\CD@GB>\z@\setbox\z@\lastbox\CD@lB
\repeat\endgraf\unskip\unskip\unpenalty}\def\CD@YJ{\CD@uA\CD@XB\advance\CD@uA
-\CD@NB\CD@vA\CD@FA\advance\CD@vA-\CD@lA\advance\CD@vA1 \expandafter\message{%
prevgraf=\the\prevgraf at (\the\CD@uA,\the\CD@vA)}}\def\CD@XI{\CD@CE\setbox
\CD@lI=\lastbox\setbox\CD@lI=\hbox{\unhbox\CD@lI\unskip\unpenalty}\unskip
\ifdim\ht\CD@lI>\ht\CD@PC\setbox\CD@MH=\copy\CD@lI\else\ifdim\dp\CD@lI>\dp
\CD@PC\setbox\CD@MH=\copy\CD@lI\else\CD@FG\CD@lI\fi\fi\advance\CD@mF.5\wd
\CD@lI\advance\CD@mI.5\wd\CD@lI\setbox\CD@lI=\hbox{\unhbox\CD@lI\CD@HG}\CD@bH
\CD@mF{\box\CD@lI}\CD@mI\z@\CD@yB\CD@vB}\def\CD@CE{\ifnum\CD@A>0 \advance
\dimen0-\CD@tB\CD@iA-.5\dimen0 \CD@A-\CD@A\else\CD@A0 \CD@iA\z@\fi\setbox
\CD@MH=\lastbox\setbox\CD@MH=\hbox{\unhbox\CD@MH\unskip\unskip\unpenalty
\setbox0=\lastbox\global\CD@QA\lastkern\unkern}\advance\CD@iA-.5\CD@QA\unskip
\setbox\CD@MH=\null\CD@mI\CD@iA\CD@mF-\CD@iA}\def\CD@Z{\ht\CD@MH\CD@tI\dp
\CD@MH\CD@sI}\def\CD@FG#1{\setbox\CD@MH=\hbox{\CD@V{\ht\CD@MH}{\ht#1}\CD@V{%
\dp\CD@MH}{\dp#1}\CD@V{\wd\CD@MH}{\wd#1}\vrule height\ht\CD@MH depth\dp\CD@MH
width\wd\CD@MH}}\def\CD@QI{\CD@CE\CD@Z\setbox\CD@lI=\lastbox\unskip\setbox
\CD@lF=\lastbox\unskip\setbox\CD@lF=\hbox{\unhbox\CD@lF\unskip\global\CD@yA
\lastpenalty\unpenalty}\advance\CD@yA9999 \ifcase\CD@yA\CD@VI\or\CD@YI\or
\CD@TI\or\CD@dI\or\CD@cI\or\CD@SI\else\cd@shouldnt9\fi}\def\CD@VI{\CD@FG
\CD@lI\CD@UI\setbox\CD@sH=\box\CD@lF\setbox\CD@tH=\box\CD@lI}\def\CD@YI{%
\CD@FG\CD@lF\setbox\CD@lI\hbox{\penalty8 \unhbox\CD@lI\unskip\unpenalty\ifnum
\lastpenalty=8 \else\CD@xH\fi}\CD@UI\setbox\CD@lF=\hbox{\unhbox\CD@lF\unskip
\unpenalty\global\setbox\CD@DA=\lastbox}\ifdim\wd\CD@lF=\z@\else\CD@xH\fi
\setbox\CD@sH=\box\CD@DA}\def\CD@xH{\CD@KB{extra material in \string\pile
\space cell (lost)}}\def\CD@UI{\CD@yB\ifvoid\CD@sH\else\CD@KB{Clashing
horizontal arrows}\CD@mI.5\CD@hF\CD@mF-\CD@mI\CD@vB\CD@mI\z@\CD@mF\z@\fi
\CD@hI\CD@hF\advance\CD@hI-\CD@mI\CD@hF-\CD@mF\CD@JC\CD@FA}\def\CD@RI{\setbox
0\lastbox\unskip\CD@iA\z@\CD@Z\ifdim\skip0>\z@\CD@tJ\CD@A0 \else\ifnum\CD@A<1
\CD@A0 \dimen0\CD@tB\fi\advance\CD@A1 \fi}\def\VonH{\CD@MA46\VonH{.5\CD@LF}}%
\def\HonV{\CD@MA57\HonV{.5\CD@LF}}\def\HmeetV{\CD@MA44\HmeetV{-\MapShortFall}%
}\def\CD@MA#1#2#3#4{\CD@pB34#1{\string#3}\CD@SD\CD@GB-999#2 \dimen0=#4\CD@tI
\dimen0\advance\CD@tI\axisheight\CD@sI\dimen0\advance\CD@sI-\axisheight\CD@CF
\CD@HC\CD@ZD}\def\CD@HC#1{\setbox0=\hbox{\CD@k#1\CD@ND}\dimen0.5\wd0 \CD@tI
\ht0 \CD@sI\dp0 \CD@ZD}\def\CD@SD{\setbox0=\null\ht0=\CD@tI\dp0=\CD@sI\wd0=%
\dimen0 \copy0\penalty\CD@GB\box0 }\def\CD@TI{\CD@GC\CD@yB}\def\CD@dI{\CD@GC
\CD@vB}\def\CD@SI{\CD@GC\CD@yB\CD@vB}\def\CD@GC{\setbox\CD@lI=\hbox{\unhbox
\CD@lI}\setbox\CD@lF=\hbox{\unhbox\CD@lF\global\setbox\CD@DA=\lastbox}\ht
\CD@MH\ht\CD@DA\dp\CD@MH\dp\CD@DA\advance\CD@mF\wd\CD@DA\advance\CD@mI\wd
\CD@lI}\CD@tG\ifPositiveGradient\CD@CI\CD@BI\CD@CI\CD@tG\ifClimbing\CD@rB
\CD@qB\CD@rB\newcount\DiagonalChoice\DiagonalChoice\m@ne\ifx\tenln\nullfont
\CD@tJ\def\CD@qF{\CD@KH\ifPositiveGradient/\else\CD@k\backslash\CD@ND\fi}%
\else\def\CD@qF{\CD@rF\char\count@}\fi\let\CD@rF\tenln\def\Use@line@char#1{%
\hbox{#1\CD@rF\ifPositiveGradient\else\advance\count@64 \fi\char\count@}}\def
\CD@cF{\Use@line@char{\count@\CD@TC\multiply\count@8\advance\count@-9\advance
\count@\CD@LH}}\def\CD@ZF{\Use@line@char{\ifcase\DiagonalChoice\CD@gF\or
\CD@fF\or\CD@fF\else\CD@gF\fi}}\def\CD@gF{\ifnum\CD@TC=\z@\count@'33 \else
\count@\CD@TC\multiply\count@\sixt@@n\advance\count@-9\advance\count@\CD@LH
\advance\count@\CD@LH\fi}\def\CD@fF{\count@'\ifcase\CD@LH55\or\ifcase\CD@TC66%
\or22\or52\or61\or72\fi\or\ifcase\CD@TC66\or25\or22\or63\or52\fi\or\ifcase
\CD@TC66\or16\or36\or22\or76\fi\or\ifcase\CD@TC66\or27\or25\or67\or22\fi\fi
\relax}\def\CD@uC#1{\hbox{#1\setbox0=\Use@line@char{#1}\ifPositiveGradient
\else\raise.3\ht0\fi\copy0 \kern-.7\wd0 \ifPositiveGradient\raise.3\ht0\fi
\box0}}\def\CD@jF#1{\hbox{\setbox0=#1\kern-.75\wd0 \vbox to.25\ht0{%
\ifPositiveGradient\else\vss\fi\box0 \ifPositiveGradient\vss\fi}}}\def\CD@jI#%
1{\hbox{\setbox0=#1\dimen0=\wd0 \vbox to.25\ht0{\ifPositiveGradient\vss\fi
\box0 \ifPositiveGradient\else\vss\fi}\kern-.75\dimen0 }}\CD@RC{+h:>}{%
\Use@line@char\CD@fF}\CD@RC{-h:>}{\Use@line@char\CD@gF}\CD@nF{+t:<}{-h:>}%
\CD@nF{-t:<}{+h:>}\CD@RC{+t:>}{\CD@jF{\Use@line@char\CD@fF}}\CD@RC{-t:>}{%
\CD@jI{\Use@line@char\CD@gF}}\CD@nF{+h:<}{-t:>}\CD@nF{-h:<}{+t:>}\CD@RC{+h:>>%
}{\CD@uC\CD@fF}\CD@RC{-h:>>}{\CD@uC\CD@gF}\CD@nF{+t:<<}{-h:>>}\CD@nF{-t:<<}{+%
h:>>}\CD@nF{+h:>->}{+h:>>}\CD@nF{-h:>->}{-h:>>}\CD@nF{+t:<-<}{-h:>>}\CD@nF{-t%
:<-<}{+h:>>}\CD@RC{+t:>>}{\CD@jF{\CD@uC\CD@fF}}\CD@RC{-t:>>}{\CD@jI{\CD@uC
\CD@gF}}\CD@nF{+h:<<}{-t:>>}\CD@nF{-h:<<}{+t:>>}\CD@nF{+t:>->}{+t:>>}\CD@nF{-%
t:>->}{-t:>>}\CD@nF{+h:<-<}{-t:>>}\CD@nF{-h:<-<}{+t:>>}\CD@RC{+f:-}{\CD@EF
\null\else\CD@cF\fi}\CD@nF{-f:-}{+f:-}\def\CD@tC#1#2{\vbox to#1{\vss\hbox to#%
2{\hss.\hss}\vss}}\def\hfdot{\CD@tC{2\axisheight}{.5em}}%
\def\vfdot{\CD@tC{1ex}\z@}%% % 1.46ex until 29.7.98
\def\CD@bF{\hbox{\dimen0=.3\CD@zC\dimen1\dimen0 \ifnum\CD@LH>\CD@TC\CD@iC{%
\dimen1}\else\CD@dG{\dimen0}\fi\CD@tC{\dimen0}{\dimen1}}}\newarrowfiller{.}%
\hfdot\hfdot\vfdot\vfdot\def\dfdot{\CD@bF\CD@CK}\CD@RC{+f:.}{\dfdot}\CD@RC{-f%
:.}{\dfdot}\def\CD@@K#1{\hbox\bgroup\def\CD@CH{#1\egroup}\afterassignment
\CD@CH%%
\count@='}\def\lnchar{\CD@@K\CD@qF}\def\CD@dF#1{\setbox#1=\hbox{\dimen5\dimen
#1 \setbox8=\box#1 \dimen1\wd8 \count@\dimen5 \divide\count@\dimen1 \ifnum
\count@=0 \box8 \ifdim\dimen5<.95\dimen1 \CD@gB{diagonal line too short}\fi
\else\dimen3=\dimen5 \advance\dimen3-\dimen1 \divide\dimen3\count@\dimen4%
\dimen3 \CD@dG{\dimen4}\ifPositiveGradient\multiply\dimen4\m@ne\fi\dimen6%
\dimen1 \advance\dimen6-\dimen3 \loop\raise\count@\dimen4\copy8 \ifnum\count@
>0 \kern-\dimen6 \advance\count@\m@ne\repeat\fi}}\def\CD@CG#1{\CD@EF\CD@xJ{#1%
}\else\CD@dF{#1}\fi}\def\CD@IH#1{}\newdimen\objectheight\objectheight1.8ex
\newdimen\objectwidth\objectwidth1em \def\CD@YD{\dimen6=\CD@aK
\DiagramCellHeight\dimen7=\CD@WK\DiagramCellWidth\CD@KJ\ifnum\CD@LH>0 \ifnum
\CD@TC>0 \CD@aF\else\aftergroup\CD@VC\fi\else\aftergroup\CD@UC\fi}\def\CD@VC{%
\CD@YA{diagonal map is nearly vertical}\CD@NA}\def\CD@UC{\CD@YA{diagonal map
is nearly horizontal}\CD@NA}\CD@rG\CD@NA{Use an orthogonal map instead}\def
\CD@aF{\CD@MJ\dimen3\dimen7\dimen7\dimen6\CD@iC{\dimen7}\advance\dimen3-%
\dimen7 \CD@MF\ifnum\CD@LH>\CD@TC\advance\dimen6-\dimen1\advance\dimen6-%
\dimen5 \CD@iC{\dimen1}\CD@iC{\dimen5}\else\dimen0\dimen1\advance\dimen0%
\dimen5\CD@dG{\dimen0}\advance\dimen6-\dimen0 \fi\dimen2.5\dimen7\advance
\dimen2-\dimen1 \dimen4.5\dimen7\advance\dimen4-\dimen5 \ifPositiveGradient
\dimen0\dimen5 \advance\dimen1-\CD@WK\DiagramCellWidth\advance\dimen1 \CD@ZK
\DiagramCellWidth\setbox6=\llap{\unhbox6\kern.1\ht2}\setbox7=\rlap{\kern.1\ht
2\unhbox7}\else\dimen0\dimen1 \advance\dimen1-\CD@ZK\DiagramCellWidth\setbox7%
=\llap{\unhbox7\kern.1\ht2}\setbox6=\rlap{\kern.1\ht2\unhbox6}\fi\setbox6=%
\vbox{\box6\kern.1\wd2}\setbox7=\vtop{\kern.1\wd2\box7}\CD@dG{\dimen0}%
\advance\dimen0-\axisheight\advance\dimen0-\CD@bK\DiagramCellHeight\dimen5-%
\dimen0 \advance\dimen0\dimen6 \advance\dimen1.5\dimen3 \ifdim\wd3>\z@\ifdim
\ht3>-\dp3\CD@TB\fi\fi\dimen3\dimen2 \dimen7\dimen2\advance\dimen7\dimen4
\ifvoid3 \else\CD@tE\else\dimen8\ht3\advance\dimen8-\axisheight\CD@iC{\dimen8%
}\CD@X{\dimen8}{.5\wd3}\dimen9\dp3\advance\dimen9\axisheight\CD@iC{\dimen9}%
\CD@X{\dimen9}{.5\wd3}\ifPositiveGradient\advance\dimen2-\dimen9\advance
\dimen4-\dimen8 \else\advance\dimen4-\dimen9\advance\dimen2-\dimen8 \fi\fi
\advance\dimen3-.5\wd3 \fi\dimen9=\CD@aK\DiagramCellHeight\advance\dimen9-2%
\DiagramCellHeight\CD@tE\advance\dimen2\dimen4 \CD@CG{2}\dimen2-\dimen0%
\advance\dimen2\dp2 \else\CD@CG{2}\CD@CG{4}\ifPositiveGradient\dimen2-\dimen0%
\advance\dimen2\dp2 \dimen4\dimen5\advance\dimen4-\ht4 \else\dimen4-\dimen0%
\advance\dimen4\dp4 \dimen2\dimen5\advance\dimen2-\ht2 \fi\fi\setbox0=\hbox to%
\z@{\kern\dimen1 \ifvoid1 \else\ifPositiveGradient\advance\dimen0-\dp1 \lower
\dimen0 \else\advance\dimen5-\ht1 \raise\dimen5 \fi\rlap{\unhbox1}\fi\raise
\dimen2\rlap{\unhbox2}\ifvoid3 \else\lower.5\dimen9\rlap{\kern\dimen3\unhbox3%
}\fi\kern.5\dimen7 \lower.5\dimen9\box6 \lower.5\dimen9\box7 \kern.5\dimen7
\CD@tE\else\raise\dimen4\llap{\unhbox4}\fi\ifvoid5 \else\ifPositiveGradient
\advance\dimen5-\ht5 \raise\dimen5 \else\advance\dimen0-\dp5 \lower\dimen0 \fi
\llap{\unhbox5}\fi\hss}\ht0=\axisheight\dp0=-\ht0\box0 }\def\NorthWest{\CD@BI
\CD@rB\DiagonalChoice0 }\def\NorthEast{\CD@CI\CD@rB\DiagonalChoice1 }\def
\SouthWest{\CD@CI\CD@qB\DiagonalChoice3 }\def\SouthEast{\CD@BI\CD@qB
\DiagonalChoice2 }\def\CD@aD{\vadjust{\CD@uA\CD@FA\advance\CD@uA
\ifPositiveGradient\else-\fi\CD@ZK\relax\CD@vA\CD@NB\advance\CD@vA-\CD@bK
\relax\hbox{\advance\CD@uA\ifPositiveGradient-\fi\CD@WK\advance\CD@vA\CD@aK
\hbox{\box6 \kern\CD@DC\kern\CD@eJ\penalty1 \box7 \box\z@}\penalty\CD@uA
\penalty\CD@vA}\penalty\CD@uA\penalty\CD@vA\penalty104}}\def\CD@eH#1{\relax
\vadjust{\hbox@maths{#1}\penalty\CD@FA\penalty\CD@NB\penalty\tw@}}\def\CD@lB{%
\ifcase\CD@GB\or\or\CD@bH{.5\wd0}{\box0}{.5\wd0}\z@\or\unhbox\z@\setbox\z@
\lastbox\CD@bH{.5\wd0}{\box0}{.5\wd0}\z@\unpenalty\unpenalty\setbox\z@
\lastbox\or\CD@TG\else\advance\CD@GB-100 \ifnum\CD@GB<\z@\cd@shouldnt B\fi
\setbox\z@\hbox{\kern\CD@mF\copy\CD@MH\kern\CD@mI\CD@uA\CD@XB\advance\CD@uA-%
\CD@NB\penalty\CD@uA\CD@uA\CD@FA\advance\CD@uA-\CD@lA\penalty\CD@uA\unhbox\z@
\global\CD@yA\lastpenalty\unpenalty\global\CD@zA\lastpenalty\unpenalty}\CD@uA
-\CD@yA\CD@vA\CD@zA\CD@fI\fi}\def\CD@TG{\unhbox\z@\setbox\z@\lastbox\CD@uA
\lastpenalty\unpenalty\advance\CD@uA\CD@mA\CD@vA\CD@XB\advance\CD@vA-%
\lastpenalty\unpenalty\dimen1\lastkern\unkern\setbox3\lastbox\dimen0\lastkern
\unkern\setbox0=\hbox to\z@{\unhbox0\setbox0\lastbox\setbox7\lastbox
\unpenalty\CD@eJ\lastkern\unkern\CD@DC\lastkern\unkern\setbox6\lastbox\dimen7%
\CD@tB\advance\dimen7-\wd\CD@uA\ifdim\dimen7<\z@\CD@CI\multiply\dimen7\m@ne
\let\mv\empty\else\CD@BI\def\mv{\raise\ht1}\kern-\dimen7 \fi\ifnum\CD@vA>%
\CD@NB\dimen6\CD@uB\advance\dimen6-\ht\CD@vA\else\dimen6\z@\fi\CD@jJ\CD@mK
\setbox1\null\ht1\dimen6\wd1\dimen7 \dimen7\dimen2 \dimen6\wd1 \CD@KJ\CD@uA
\CD@LH\CD@vA\CD@TC\dimen6\ht1 \CD@KJ\setbox2\null\divide\dimen2\tw@\advance
\dimen2\CD@eJ\CD@eG{\dimen2}\wd2\dimen2 \dimen0.5\dimen7 \advance\dimen0%
\ifPositiveGradient\else-\fi\CD@eJ\CD@dG{\dimen0}\advance\dimen0-\axisheight
\ht2\dimen0 \dimen0\CD@DC\CD@eG{\dimen0}\advance\dimen0\ht2\ht2\dimen0 \dimen
0\ifPositiveGradient-\fi\CD@DC\CD@dG{\dimen0}\advance\dimen0\wd2\wd2\dimen0
\setbox4\null\dimen0 .6\CD@zC\CD@eG{\dimen0}\ht4\dimen0 \dimen0 .2\CD@zC
\CD@dG{\dimen0}\wd4\dimen0 \dimen0\wd2 \ifvoid6\else\dimen1\ht4 \advance
\dimen1\ht2 \CD@CC6+-\raise\dimen1\rlap{\ifPositiveGradient\advance\dimen0-%
\wd6\advance\dimen0-\wd4 \else\advance\dimen0\wd4 \fi\kern\dimen0\box6}\fi
\dimen0\wd2 \ifvoid7\else\dimen1\ht4 \advance\dimen1-\ht2 \CD@CC7-+\lower
\dimen1\rlap{\ifPositiveGradient\advance\dimen0\wd4 \else\advance\dimen0-\wd7%
\advance\dimen0-\wd4 \fi\kern\dimen0\box7}\fi\mv\box0\hss}\ht0\z@\dp0\z@
\CD@bH{\z@}{\box\z@}{\z@}{\axisheight}}\def\CD@CC#1#2#3{\dimen4.5\wd#1 \ifdim
\dimen4>.25\dimen7\dimen4=.25\dimen7\fi\ifdim\dimen4>\CD@zC\dimen4.4\dimen4
\advance\dimen4.6\CD@zC\fi\CD@eG{\dimen4}\dimen5\axisheight\CD@dG{\dimen5}%
\advance\dimen4-\dimen5 \dimen5\dimen4\CD@eG{\dimen5}\advance\dimen0%
\ifPositiveGradient#2\else#3\fi\dimen5 \CD@dG{\dimen4}\advance\dimen1\dimen4 }
\def\CD@eD#1{\expandafter\CD@IK{#1}}\CD@ZA\CD@EK{output is PostScript
dependent}\def\CD@SC{\CD@IK{/bturn {gsave currentpoint currentpoint translate
4 2 roll neg exch atan rotate neg exch neg exch translate } def /eturn {%
currentpoint grestore moveto} def}}\def\CD@gK{\relax\CD@hK\CD@tK{Q}\else
\CD@IK{eturn}\fi} \def\CD@OJ#1{\count@#1\relax\multiply\count@7\advance
\count@16577\divide\count@33154 }\def\CD@fD#1{\expandafter\special{#1}} \def
\CD@xJ#1{\setbox#1=\hbox{\dimen0\dimen#1\CD@dG{\dimen0}\CD@OJ{\dimen0}\setbox
0=\null\ifPositiveGradient\count@-\count@\ht0\dimen0 \else\dp0\dimen0 \fi\box
0 \CD@uA\count@\CD@OJ\CD@LF\CD@fD{pn \the\count@}\CD@fD{pa 0 0}\CD@OJ{\dimen#%
1}\CD@fD{pa \the\count@\space\the\CD@uA}\CD@fD{fp}\kern\dimen#1}}\def\CD@JI{%
\CD@KJ\begingroup\ifdim\dimen7<\dimen6 \dimen2=\dimen6 \dimen6=\dimen7 \dimen
7=\dimen2 \count@\CD@LH\CD@LH\CD@TC\CD@TC\count@\else\dimen2=\dimen7 \fi
\ifdim\dimen6>.01\p@\CD@KI\global\CD@QA\dimen0 \else\global\CD@QA\dimen7 \fi
\endgroup\dimen2\CD@QA\CD@iK\CD@lK{\ifPositiveGradient\else-\fi\dimen6}\CD@iK
\CD@kK{\ifPositiveGradient-\fi\dimen6}\CD@iK\CD@eK{\dimen7}}\def\CD@KI{\CD@hJ
\ifdim\dimen7>1.73\dimen6 \divide\dimen2 4 \multiply\CD@TC2 \else\dimen2=0.%
353553\dimen2 \advance\CD@LH-\CD@TC\multiply\CD@TC4 \fi\dimen0=4\dimen2 \CD@ZG
4\CD@ZG{-2}\CD@ZG2\CD@ZG{-2.5}}\def\CD@AI{\begingroup\count@\dimen0 \dimen2 45%
pt \divide\count@\dimen2 \ifdim\dimen0<\z@\advance\count@\m@ne\fi\ifodd
\count@\advance\count@1\CD@@A\else\CD@y\fi\advance\dimen0-\count@\dimen2
\CD@gE\multiply\dimen0\m@ne\fi\ifnum\count@<0 \multiply\count@-7 \fi\dimen3%
\dimen1 \dimen6\dimen0 \dimen7 3754936sp \ifdim\dimen0<6\p@\def\CD@OG{4000}%
\fi\CD@KJ\dimen2\dimen3\CD@dG{\dimen2}\CD@hJ\multiply\CD@TC-6 \dimen0\dimen2
\CD@ZG1\CD@ZG{0.3}\dimen1\dimen0 \dimen2\dimen3 \dimen0\dimen3 \CD@ZG3\CD@ZG{%
1.5}\CD@ZG{0.3}\divide\count@2 \CD@gE\multiply\dimen1\m@ne\fi\ifodd\count@
\dimen2\dimen1\dimen1\dimen0\dimen0-\dimen2 \fi\divide\count@2 \ifodd\count@
\multiply\dimen0\m@ne\multiply\dimen1\m@ne\fi\global\CD@QA\dimen0\global
\CD@RA\dimen1\endgroup\dimen6\CD@QA\dimen7\CD@RA}\def\CD@OC{255}\let\CD@OG
\CD@OC\def\CD@KJ{\begingroup\ifdim\dimen7<\dimen6 \dimen9\dimen7\dimen7\dimen
6\dimen6\dimen9\CD@@A\else\CD@y\fi\dimen2\z@\dimen3\CD@XH\dimen4\CD@XH\dimen0%
\z@\dimen8=\CD@OG\CD@XH\CD@lC\global\CD@yA\dimen\CD@gE0\else3\fi\global\CD@zA
\dimen\CD@gE3\else0\fi\endgroup\CD@LH\CD@yA\CD@TC\CD@zA}\def\CD@lC{\count@
\dimen6 \divide\count@\dimen7 \advance\dimen6-\count@\dimen7 \dimen9\dimen4
\advance\dimen9\count@\dimen0 \ifdim\dimen9>\dimen8 \CD@@C\else\CD@AC\ifdim
\dimen6>\z@\dimen9\dimen6 \dimen6\dimen7 \dimen7\dimen9 \expandafter
\expandafter\expandafter\CD@lC\fi\fi}\def\CD@@C{\ifdim\dimen0=\z@\ifdim\dimen
9<2\dimen8 \dimen0\dimen8 \fi\else\advance\dimen8-\dimen4 \divide\dimen8%
\dimen0 \ifdim\count@\CD@XH<2\dimen8 \count@\dimen8 \dimen9\dimen4 \advance
\dimen9\count@\dimen0 \CD@AC\fi\fi}\def\CD@AC{\dimen4\dimen0 \dimen0\dimen9
\advance\dimen2\count@\dimen3 \dimen9\dimen2 \dimen2\dimen3 \dimen3\dimen9 }%
\def\CD@ZG#1{\CD@dG{\dimen2}\advance\dimen0 #1\dimen2 }\def\CD@dG#1{\divide#1%
\CD@TC\multiply#1\CD@LH}\def\CD@eG#1{\divide#1\CD@vA\multiply#1\CD@uA}\def
\CD@iC#1{\divide#1\CD@LH\multiply#1\CD@TC}\def\CD@hJ{\dimen6\CD@LH\CD@XH
\multiply\dimen6\CD@LH\dimen7\CD@TC\CD@XH\multiply\dimen7\CD@TC\CD@KJ}\def
\CD@iK#1#2{\begingroup\dimen@#2\relax\loop\ifdim\dimen2<.4\maxdimen\multiply
\dimen2\tw@\multiply\dimen@\tw@\repeat\divide\dimen2\@cclvi\divide\dimen@
\dimen2\relax\multiply\dimen@\@cclvi\expandafter\CD@jK\the\dimen@\endgroup
\let#1\CD@fK}{\catcode`p=12 \catcode`0=12 \catcode`.=12 \catcode`t=12 \gdef
\CD@jK#1pt{\gdef\CD@fK{#1}}}\ifx\errorcontextlines\CD@qK\CD@tJ\let\CD@GH
\relax\else\def\CD@GH{\errorcontextlines\m@ne}\fi\ifnum\inputlineno<0 \let
\CD@CD\empty\let\CD@W\empty\let\CD@mD\relax\let\CD@uI\relax\let\CD@vI\relax
\let\CD@zF\relax\else\def\CD@W{ at line \number\inputlineno}\def\CD@mD{ - first occurred%
}\def\CD@uI{\edef\CD@h{\the\inputlineno}\global\let\CD@jB\CD@h}\def\CD@h{9999%
}\def\CD@vI{\xdef\CD@jB{\the\inputlineno}}\def\CD@jB{\CD@h}\def\CD@zF{\ifnum
\CD@h<\inputlineno\edef\CD@CD{\space at lines \CD@h--\the\inputlineno}\else
\edef\CD@CD{\CD@W}\fi}\fi\let\CD@CD\empty\def\CD@YA#1#2{\CD@GH\errhelp=#2%
\expandafter\errmessage{\CD@tA: #1}}\def\CD@KB#1{\begingroup\expandafter
\message{! \CD@tA: #1\CD@CD}\ifnum\CD@XB>\CD@NB\ifnum\CD@CA>\CD@NB\else\ifnum
\CD@lA>\CD@FA\else\ifnum\CD@LB>\CD@FA\advance\CD@XB-\CD@NB\advance\CD@FA-%
\CD@lA\advance\CD@FA1\relax\expandafter\message{! (error detected at row \the
\CD@XB, column \the\CD@FA, but probably caused elsewhere)}\fi\fi\fi\fi
\endgroup}\def\CD@gB#1{{\expandafter\message{\CD@tA\space Warning: #1\CD@W}}}%
\def\CD@CB#1#2{\CD@gB{#1 \string#2 is obsolete\CD@mD}}\def\CD@AB#1{\CD@CB{%
Dimension}{#1}\CD@DE#1\CD@BB\CD@BB}\def\CD@BB{\CD@OA=}\def\CD@@B#1{\CD@CB{%
Count}{#1}\CD@DE#1\CD@OH\CD@OH}\def\CD@OH{\count@=}\def\HorizontalMapLength{%
\CD@AB\HorizontalMapLength}\def\VerticalMapHeight{\CD@AB\VerticalMapHeight}%
\def\VerticalMapDepth{\CD@AB\VerticalMapDepth}\def\VerticalMapExtraHeight{%
\CD@AB\VerticalMapExtraHeight}\def\VerticalMapExtraDepth{\CD@AB
\VerticalMapExtraDepth}\def\DiagonalLineSegments{\CD@@B\DiagonalLineSegments}%
\ifx\tenln\nullfont\CD@ZA\CD@KH{\CD@eF\space diagonal line and arrow font not
available}\else\let\CD@KH\relax\fi\def\CD@aG#1#2<#3:#4:#5#6{\begingroup\CD@PA
#3\relax\advance\CD@PA-#2\relax\ifdim.1em<\CD@PA\CD@uA#5\relax\CD@vA#6\relax
\ifnum\CD@uA<\CD@vA\count@\CD@vA\advance\count@-\CD@uA\CD@KB{#4 by \the\CD@PA
}\if#1v\let\CD@CH\CD@JK\edef\tmp{\the\CD@uA--\the\CD@vA,\the\CD@FA}\else
\advance\count@\count@\if#1l\advance\count@-\CD@A\else\if#1r\advance\count@
\CD@A\fi\fi\advance\CD@PA\CD@PA\let\CD@CH\CD@ZE\edef\tmp{\the\CD@NB,\the
\CD@uA--\the\CD@vA}\fi\divide\CD@PA\count@\ifdim\CD@CH<\CD@PA\global\CD@CH
\CD@PA\fi\fi\fi\endgroup}\CD@tG\CD@xE\CD@JD\CD@ID\CD@rG\CD@xI{See the message
above.}\CD@rG\CD@lH{Perhaps you've forgotten to end the diagram before
resuming the text, in\CD@uG which case some garbage may be added to the
diagram, but we should be ok now.\CD@uG Alternatively you've left a blank line
in the middle - TeX will now complain\CD@uG that the remaining \CD@S s are
misplaced - so please use comments for layout.}\CD@rG\CD@hD{You have already
closed too many brace pairs or environments; an \CD@HD\CD@uG command was (%
over)due.}\CD@rG\CD@hH{\CD@dC\space and \CD@HD\space commands must match.}%
\def\CD@jH{\ifnum\inputlineno=0 \else\expandafter\CD@iH\fi}\def\CD@iH{\CD@MD
\CD@GD\crcr\CD@YA{missing \CD@HD\space inserted before \CD@kH- type "h"}%
\CD@lH\enddiagram\CD@AG\CD@kH\par}\def\CD@AG#1{\edef\enddiagram{\noexpand
\CD@rD{#1\CD@W}}}\def\CD@rD#1{\CD@YA{\CD@HD\space(anticipated by #1) ignored}%
\CD@xI\let\enddiagram\CD@SG}\def\CD@SG{\CD@YA{misplaced \CD@HD\space ignored}%
\CD@hH}\def\CD@mC{\CD@YA{missing \CD@HD\space inserted.}\CD@hD\CD@AG{closing
group}}\ifx\ifx
\fi\fi

\catcode`\$=3 %% make sure that $ means maths-shift
\def\vboxtoz{\vbox to\z@}%% \z@ is in plain TeX and means 0pt

\def\scriptaxis#1{\@scriptaxis{$\scriptstyle#1$}}%%
\def\ssaxis#1{\@ssaxis{$\scriptscriptstyle#1$}}%%
\def\@scriptaxis#1{\dimen0\axisheight\advance\dimen0-\ss@axisheight\raise
\dimen0\hbox{#1}}\def\@ssaxis#1{\dimen0\axisheight\advance\dimen0-%
\ss@axisheight\raise\dimen0\hbox{#1}}

\ifx\boldmath\CD@qK%%
\let\boldscriptaxis\scriptaxis%%
\def\boldscript#1{\hbox{$\scriptstyle#1$}}%%
\else\def\boldscriptaxis#1{\@scriptaxis{\boldmath$\scriptstyle#1$}}%%
\def\boldscript#1{\hbox{\boldmath$\scriptstyle#1$}}%%
\fi

\def\raisehook#1#2#3{\hbox{\setbox3=\hbox{#1$\scriptscriptstyle#3$}%
\dimen0\ss@axisheight%% \scriptscriptstyle axis height
\dimen1\axisheight\advance\dimen1-\dimen0%% difference in axis heights
\dimen2\ht3\advance\dimen2-\dimen0%
\advance\dimen2-0.021em\advance\dimen1 #2\dimen2%
\raise\dimen1\box3}}%% print the character
\def\shifthook#1#2#3{\setbox1=\hbox{#1$\scriptscriptstyle#3$}\dimen0\wd1%
\divide\dimen0 12\CD@zH{\dimen0}%%  "u"
\dimen1\wd1\advance\dimen1-2\dimen0 \advance\dimen1-2\CD@oI\CD@zH{\dimen1}%
\kern#2\dimen1\box1}%% print

\def\@cmex{\mathchar"03}%%ascii double quote

\def\make@pbk#1{\setbox\tw@\hbox to\z@{#1}\ht\tw@\z@\dp\tw@\z@\box\tw@}\def
\CD@fH#1{\overprint{\hbox to\z@{#1}}}\def\CD@qH{\kern0.11em}\def\CD@pH{\kern0%
.35em}

\def\dblvert{\def\CD@rH{\kern.5\PileSpacing}}\def\CD@rH{}

\def\SEpbk{\make@pbk{\CD@qH\CD@rH\vrule depth 2.87ex height -2.75ex width 0.%
95em \vrule height -0.66ex depth 2.87ex width 0.05em \hss}}

\def\SWpbk{\make@pbk{\hss\vrule height -0.66ex depth 2.87ex width 0.05em
\vrule depth 2.87ex height -2.75ex width 0.95em \CD@qH\CD@rH}}

\def\NEpbk{\make@pbk{\CD@qH\CD@rH\vrule depth -3.81ex height 4.00ex width 0.%
95em \vrule height 4.00ex depth -1.72ex width 0.05em \hss}}

\def\NWpbk{\make@pbk{\hss\vrule height 4.00ex depth -1.72ex width 0.05em
\vrule depth -3.81ex height 4.00ex width 0.95em \CD@qH\CD@rH}}

\def\puncture{{\setbox0\hbox{A}\vrule height.53\ht0 depth-.47\ht0 width.35\ht
0 \kern.12\ht0 \vrule height\ht0 depth-.65\ht0 width.06\ht0 \kern-.06\ht0
\vrule height.35\ht0 depth0pt width.06\ht0 \kern.12\ht0 \vrule height.53\ht0
depth-.47\ht0 width.35\ht0 }}

\def\NEclck{\overprint{\raise2.5ex\rlap{ \CD@rH$\scriptstyle\searrow$}}}%%
\def\NEanti{\overprint{\raise2.5ex\rlap{ \CD@rH$\scriptstyle\nwarrow$}}}%%
\def\NWclck{\overprint{\raise2.5ex\llap{$\scriptstyle\nearrow$ \CD@rH}}}%%
\def\NWanti{\overprint{\raise2.5ex\llap{$\scriptstyle\swarrow$ \CD@rH}}}%%
\def\SEclck{\overprint{\lower1ex\rlap{ \CD@rH$\scriptstyle\swarrow$}}}%%
\def\SEanti{\overprint{\lower1ex\rlap{ \CD@rH$\scriptstyle\nearrow$}}}%%
\def\SWclck{\overprint{\lower1ex\llap{$\scriptstyle\nwarrow$ \CD@rH}}}%%
\def\SWanti{\overprint{\lower1ex\llap{$\scriptstyle\searrow$ \CD@rH}}}

\def\rhvee{\mkern-10mu\greaterthan}%%
\def\lhvee{\lessthan\mkern-10mu}%%
\def\dhvee{\vboxtoz{\vss\hbox{$\vee$}\kern0pt}}%%
\def\uhvee{\vboxtoz{\hbox{$\wedge$}\vss}}%%
\newarrowhead{vee}\rhvee\lhvee\dhvee\uhvee

\def\dhlvee{\vboxtoz{\vss\hbox{$\scriptstyle\vee$}\kern0pt}}%%
\def\uhlvee{\vboxtoz{\hbox{$\scriptstyle\wedge$}\vss}}%%
\newarrowhead{littlevee}{\mkern1mu\scriptaxis\rhvee}{\scriptaxis\lhvee}%
\dhlvee\uhlvee\ifx\boldmath\CD@qK%%
\newarrowhead{boldlittlevee}{\mkern1mu\scriptaxis\rhvee}{\scriptaxis\lhvee}%
\dhlvee\uhlvee\else%%
\def\dhblvee{\vboxtoz{\vss\boldscript\vee\kern0pt}}%%
\def\uhblvee{\vboxtoz{\boldscript\wedge\vss}}%%
\newarrowhead{boldlittlevee}{\mkern1mu\boldscriptaxis\rhvee}{\boldscriptaxis
\lhvee}\dhblvee\uhblvee%%
\fi

\def\rhcvee{\mkern-10mu\succ}%%
\def\lhcvee{\prec\mkern-10mu}%%
\def\dhcvee{\vboxtoz{\vss\hbox{$\curlyvee$}\kern0pt}}%%
\def\uhcvee{\vboxtoz{\hbox{$\curlywedge$}\vss}}%%
\newarrowhead{curlyvee}\rhcvee\lhcvee\dhcvee\uhcvee

\def\rhvvee{\mkern-13mu\gg}%% 24.8.92 changed 10mu to 13mu
\def\lhvvee{\ll\mkern-13mu}%% to make rule go through
\def\dhvvee{\vboxtoz{\vss\hbox{$\vee$}\kern-.6ex\hbox{$\vee$}\kern0pt}}%%
\def\uhvvee{\vboxtoz{\hbox{$\wedge$}\kern-.6ex \hbox{$\wedge$}\vss}}%%
\newarrowhead{doublevee}\rhvvee\lhvvee\dhvvee\uhvvee

\def\rhtriangle{\triangleright\mkern1.2mu}%% 29.1.93
\def\lhtriangle{\triangleleft\mkern.8mu}%%
\def\uhtriangle{\vbox{\kern-.2ex \hbox{$\scriptscriptstyle\bigtriangleup$}%
\kern-.25ex}}%%
\def\dhtriangle{\vbox{\kern-.28ex \hbox{$\scriptscriptstyle\bigtriangledown$}%
\kern-.1ex}}%% 15.1.93 from -.25ex
\def\dhblack{\vbox{\kern-.25ex\nointerlineskip\hbox{$\blacktriangledown$}}}%
\def\uhblack{\vbox{\kern-.25ex\nointerlineskip\hbox{$\blacktriangle$}}}%
\def\dhlblack{\vbox{\kern-.25ex\nointerlineskip\hbox{$\scriptstyle
\blacktriangledown$}}}%% AMS
\def\uhlblack{\vbox{\kern-.25ex\nointerlineskip\hbox{$\scriptstyle
\blacktriangle$}}}%% AMS
\newarrowhead{triangle}\rhtriangle\lhtriangle\dhtriangle\uhtriangle
\newarrowhead{blacktriangle}{\mkern-1mu\blacktriangleright\mkern.4mu}{%
\blacktriangleleft}\dhblack\uhblack\newarrowhead{littleblack}{\mkern-1mu%
\scriptaxis\blacktriangleright}{\scriptaxis\blacktriangleleft\mkern-2mu}%
\dhlblack\uhlblack

\def\rhla{\hbox{\setbox0=\lnchar55\dimen0=\wd0\kern-.6\dimen0\ht0\z@\raise
\axisheight\box0\kern.1\dimen0}}%%
\def\lhla{\hbox{\setbox0=\lnchar33\dimen0=\wd0\kern.05\dimen0\ht0\z@\raise
\axisheight\box0\kern-.5\dimen0}}%%
\def\dhla{\vboxtoz{\vss\rlap{\lnchar77}}}%%
\def\uhla{\vboxtoz{\setbox0=\lnchar66 \wd0\z@\kern-.15\ht0\box0\vss}}%% 1/93
\newarrowhead{LaTeX}\rhla\lhla\dhla\uhla

\def\lhlala{\lhla\kern.3em\lhla}%%
\def\rhlala{\rhla\kern.3em\rhla}%%
\def\uhlala{\hbox{\uhla\raise-.6ex\uhla}}%%
\def\dhlala{\hbox{\dhla\lower-.6ex\dhla}}%%
\newarrowhead{doubleLaTeX}\rhlala\lhlala\dhlala\uhlala

\def\hhO{\scriptaxis\bigcirc\mkern.4mu} \def\hho{{\circ}\mkern1.2mu}%
\newarrowhead{o}\hho\hho\circ\circ%%
\newarrowhead{O}\hhO\hhO{\scriptstyle\bigcirc}{\scriptstyle\bigcirc}%%

\def\rhtimes{\mkern-5mu{\times}\mkern-.8mu}\def\lhtimes{\mkern-.8mu{\times}%
\mkern-5mu}\def\uhtimes{\setbox0=\hbox{$\times$}\ht0\axisheight\dp0-\ht0%
\lower\ht0\box0 }\def\dhtimes{\setbox0=\hbox{$\times$}\ht0\axisheight\box0 }%
\newarrowhead{X}\rhtimes\lhtimes\dhtimes\uhtimes\newarrowhead+++++

\newarrowhead{Y}{\mkern-3mu\Yright}{\Yleft\mkern-3mu}\Ydown\Yup

\newarrowhead{->}\rightarrow\leftarrow\downarrow\uparrow

\newarrowhead{=>}\Rightarrow\Leftarrow{\@cmex7F}{\@cmex7E}

\newarrowhead{harpoon}\rightharpoonup\leftharpoonup\downharpoonleft
\upharpoonleft

\def\twoheaddownarrow{\rlap{$\downarrow$}\raise-.5ex\hbox{$\downarrow$}}%%
\def\twoheaduparrow{\rlap{$\uparrow$}\raise.5ex\hbox{$\uparrow$}}%%
\newarrowhead{->>}\twoheadrightarrow\twoheadleftarrow\twoheaddownarrow
\twoheaduparrow

\def\ltvee{\mkern-1mu{\lessthan}\mkern.4mu}%% \mkern added 15.1.93
\newarrowtail{vee}\greaterthan\ltvee\vee\wedge

\newarrowtail{littlevee}{\scriptaxis\greaterthan}{\mkern-1mu\scriptaxis
\lessthan}{\scriptstyle\vee}{\scriptstyle\wedge}\ifx\boldmath\CD@qK
\newarrowtail{boldlittlevee}{\scriptaxis\greaterthan}{\mkern-1mu\scriptaxis
\lessthan}{\scriptstyle\vee}{\scriptstyle\wedge}\else\newarrowtail{%
boldlittlevee}{\boldscriptaxis\greaterthan}{\mkern-1mu\boldscriptaxis
\lessthan}{\boldscript\vee}{\boldscript\wedge}\fi

\newarrowtail{curlyvee}\succ{\mkern-1mu{\prec}\mkern.4mu}\curlyvee\curlywedge

\def\rttriangle{\mkern1.2mu\triangleright}%% 29.1.93
\newarrowtail{triangle}\rttriangle\lhtriangle\dhtriangle\uhtriangle
\newarrowtail{blacktriangle}\blacktriangleright{\mkern-1mu\blacktriangleleft
\mkern.4mu}\dhblack\uhblack\newarrowtail{littleblack}{\scriptaxis
\blacktriangleright\mkern-2mu}{\mkern-1mu\scriptaxis\blacktriangleleft}%
\dhlblack\uhlblack

\def\rtla{\hbox{\setbox0=\lnchar55\dimen0=\wd0\kern-.5\dimen0\ht0\z@\raise
\axisheight\box0\kern-.2\dimen0}}%%
\def\ltla{\hbox{\setbox0=\lnchar33\dimen0=\wd0\kern-.15\dimen0\ht0\z@\raise
\axisheight\box0\kern-.5\dimen0}}%%
\def\dtla{\vbox{\setbox0=\rlap{\lnchar77}\dimen0=\ht0\kern-.7\dimen0\box0%
\kern-.1\dimen0}}%% 15.1.93 from -.6
\def\utla{\vbox{\setbox0=\rlap{\lnchar66}\dimen0=\ht0\kern-.1\dimen0\box0%
\kern-.6\dimen0}}%%
\newarrowtail{LaTeX}\rtla\ltla\dtla\utla

\def\rtvvee{\gg\mkern-3mu}%%
\def\ltvvee{\mkern-3mu\ll}%%
\def\dtvvee{\vbox{\hbox{$\vee$}\kern-.6ex \hbox{$\vee$}\vss}}%%
\def\utvvee{\vbox{\vss\hbox{$\wedge$}\kern-.6ex \hbox{$\wedge$}\kern\z@}}%%
\newarrowtail{doublevee}\rtvvee\ltvvee\dtvvee\utvvee

\def\ltlala{\ltla\kern.3em\ltla}%%
\def\rtlala{\rtla\kern.3em\rtla}%%
\def\utlala{\hbox{\utla\raise-.6ex\utla}}%%
\def\dtlala{\hbox{\dtla\lower-.6ex\dtla}}%%
\newarrowtail{doubleLaTeX}\rtlala\ltlala\dtlala\utlala

\def\utbar{\vrule height 0.093ex depth0pt width 0.4em}%%
\let\dtbar\utbar%%
\def\rtbar{\mkern1.5mu\vrule height 1.1ex depth.06ex width .04em\mkern1.5mu}%
\let\ltbar\rtbar%%
\newarrowtail{mapsto}\rtbar\ltbar\dtbar\utbar%%
\newarrowtail{|}\rtbar\ltbar\dtbar\utbar%%ascii vertical bar (|)

\def\rthooka{\raisehook{}+\subset\mkern-1mu}%%
\def\lthooka{\mkern-1mu\raisehook{}+\supset}%%
\def\rthookb{\raisehook{}-\subset\mkern-2mu}%%
\def\lthookb{\mkern-1mu\raisehook{}-\supset}%%

\def\dthooka{\shifthook{}+\cap}%%
\def\dthookb{\shifthook{}-\cap}%%
\def\uthooka{\shifthook{}+\cup}%%
\def\uthookb{\shifthook{}-\cup}%%

\newarrowtail{hooka}\rthooka\lthooka\dthooka\uthooka\newarrowtail{hookb}%
\rthookb\lthookb\dthookb\uthookb

\ifx\boldmath\CD@qK\newarrowtail{boldhooka}\rthooka\lthooka\dthooka\uthooka
\newarrowtail{boldhookb}\rthookb\lthookb\dthookb\uthookb\newarrowtail{%
boldhook}\rthooka\lthooka\dthookb\uthooka\else\def\rtbhooka{\raisehook
\boldmath+\subset\mkern-1mu}%%
\def\ltbhooka{\mkern-1mu\raisehook\boldmath+\supset}%%
\def\rtbhookb{\raisehook\boldmath-\subset\mkern-2mu}%%
\def\ltbhookb{\mkern-1mu\raisehook\boldmath-\supset}%%
\def\dtbhooka{\shifthook\boldmath+\cap}%%
\def\dtbhookb{\shifthook\boldmath-\cap}%%
\def\utbhooka{\shifthook\boldmath+\cup}%%
\def\utbhookb{\shifthook\boldmath-\cup}%%
\newarrowtail{boldhooka}\rtbhooka\ltbhooka\dtbhooka\utbhooka\newarrowtail{%
boldhookb}\rtbhookb\ltbhookb\dtbhookb\utbhookb\newarrowtail{boldhook}%
\rtbhooka\ltbhooka\dtbhooka\utbhooka\fi

\def\dtsqhooka{\shifthook{}+\sqcap}%%
\def\ltsqhooka{\mkern-1mu\raisehook{}+\sqsupset}%%
\def\rtsqhooka{\raisehook{}+\sqsubset\mkern-1mu}%%
\def\utsqhooka{\shifthook{}+\sqcup}%%
\newarrowtail{sqhook}\rtsqhooka\ltsqhooka\dtsqhooka\utsqhooka

\newarrowtail{hook}\rthooka\lthookb\dthooka\uthooka\newarrowtail{C}\rthooka
\lthookb\dthooka\uthooka

\newarrowtail{o}\hho\hho\circ\circ%%
\newarrowtail{O}\hhO\hhO{\scriptstyle\bigcirc}{\scriptstyle\bigcirc}%%

\newarrowtail{X}\lhtimes\rhtimes\uhtimes\dhtimes\newarrowtail+++++

\newarrowtail{Y}\Yright\Yleft\Ydown\Yup

\newarrowtail{harpoon}\leftharpoondown\rightharpoondown\upharpoonright
\downharpoonright

\newarrowtail{<=}\Leftarrow\Rightarrow{\@cmex7E}{\@cmex7F}

\newarrowfiller{=}=={\@cmex77}{\@cmex77}%% 16.1.93
\def\vfthree{\mid\!\!\!\mid\!\!\!\mid}%%ascii
\newarrowfiller{3}\equiv\equiv\vfthree\vfthree

\def\vfdashstrut{\vrule width0pt height1.3ex depth0.7ex}%%
\def\vfthedash{\vrule width\CD@LF height0.6ex depth 0pt}%%
\def\hfthedash{\CD@AJ\vrule\horizhtdp width 0.26em}%%
\def\hfdash{\mkern5.5mu\hfthedash\mkern5.5mu}%%
\def\vfdash{\vfdashstrut\vfthedash}%%
\newarrowfiller{dash}\hfdash\hfdash\vfdash\vfdash

\newarrowmiddle+++++

\iffalse%%
\newarrow{To}----{vee}%%
\newarrow{Arr}----{LaTeX}%%
\newarrow{Dotsto}....{vee}%%
\newarrow{Dotsarr}....{LaTeX}%%
\newarrow{Dashto}{}{dash}{}{dash}{vee}%%
\newarrow{Dasharr}{}{dash}{}{dash}{LaTeX}%%
\newarrow{Mapsto}{mapsto}---{vee}%%
\newarrow{Mapsarr}{mapsto}---{LaTeX}%%
\newarrow{IntoA}{hooka}---{vee}%%
\newarrow{IntoB}{hookb}---{vee}%%
\newarrow{Embed}{vee}---{vee}%%
\newarrow{Emarr}{LaTeX}---{LaTeX}%%
\newarrow{Onto}----{doublevee}%%
\newarrow{Dotsonarr}....{doubleLaTeX}%%
\newarrow{Dotsonto}....{doublevee}%%
\newarrow{Dotsonarr}....{doubleLaTeX}%%
\else%%
\newarrow{To}---->%%
\newarrow{Arr}---->%%
\newarrow{Dotsto}....>%%
\newarrow{Dotsarr}....>%%
\newarrow{Dashto}{}{dash}{}{dash}>%%
\newarrow{Dasharr}{}{dash}{}{dash}>%%
\newarrow{Mapsto}{mapsto}--->%%
\newarrow{Mapsarr}{mapsto}--->%%
\newarrow{IntoA}{hooka}--->%%
\newarrow{IntoB}{hookb}--->%%
\newarrow{Embed}>--->%%
\newarrow{Emarr}>--->%%
\newarrow{Onto}----{>>}%%
\newarrow{Dotsonarr}....{>>}%%
\newarrow{Dotsonto}....{>>}%%
\newarrow{Dotsonarr}....{>>}%%
\fi%%

\newarrow{Implies}===={=>}%% minimum cell height 9.5pt
\newarrow{Project}----{triangle}%%
\newarrow{Pto}----{harpoon}%% partial function
\newarrow{Relto}{harpoon}---{harpoon}%% binary relation

\newarrow{Eq}=====%%
\newarrow{Line}-----%%
\newarrow{Dots}.....%%
\newarrow{Dashes}{}{dash}{}{dash}{}%%

\newarrow{SquareInto}{sqhook}--->

\newarrowhead{cmexbra}{\@cmex7B}{\@cmex7C}{\@cmex3B}{\@cmex38}%%
\newarrowtail{cmexbra}{\@cmex7A}{\@cmex7D}{\@cmex39}{\@cmex3A}%%
\newarrowmiddle{cmexbra}{\braceru\bracelu}{\bracerd\braceld}{\vcenter{%
\hbox@maths{\@cmex3D\mkern-2mu}}}%% right
{\vcenter{\hbox@maths{\mkern2mu\@cmex3C}}}%% left
\newarrow{@brace}{cmexbra}-{cmexbra}-{cmexbra}%% braces
\newarrow{@parenth}{cmexbra}---{cmexbra}%% straight parentheses
\def\rightBrace{\d@brace[thick,cmex]}%%ASCII square brackets []
\def\leftBrace{\u@brace[thick,cmex]}%%ASCII square brackets []
\def\upperBrace{\r@brace[thick,cmex]}%%ASCII square brackets []
\def\lowerBrace{\l@brace[thick,cmex]}%%ASCII square brackets []
\def\rightParenth{\d@parenth[thick,cmex]}%%ASCII square brackets []
\def\leftParenth{\u@parenth[thick,cmex]}%%ASCII square brackets []
\def\upperParenth{\r@parenth[thick,cmex]}%%ASCII square brackets []
\def\lowerParenth{\l@parenth[thick,cmex]}%%ASCII square brackets []

\let\hEq\rEq%%
\let\vEq\uEq%%
\def\labelstyle{%%
\ifincommdiag%%
\textstyle%%
\else%%
\scriptstyle%%
\fi}%%
\let\objectstyle\displaystyle

\newdiagramgrid{pentagon}{0.618034,0.618034,1,1,1,1,0.618034,0.618034}{1.%
17557,1.17557,1.902113,1.902113}

\newdiagramgrid{perspective}{0.75,0.75,1.1,1.1,0.9,0.9,0.95,0.95,0.75,0.75}{0%
.75,0.75,1.1,1.1,0.9,0.9}

\diagramstyle[%%ascii open square bracket
dpi=300,%%              office laserwriters are usually 300 dots per inch
vmiddle,nobalance,%%    vertical and horizontal positioning
loose,%%                allow rows and columns to stretch
thin,%%                 line10 arrows; default rule thickness (TeXbook p447)
pilespacing=10pt,%
shortfall=4pt,%%        distance between arrowheads and their targets
]%%ascii close square bracket

\ifx\CD@qK\else\CD@PK\relax\CD@FF\CD@e\fi\fi

\CD@vE\CD@hK\else\fi\else\fi\cdrestoreat
\begin{document}

\address{Independent University of Moscow,
B. Vlasievsky per., 11,
Moscow 121002,
Russia}
\email{alxklg@gmail.com}

\author{Alexey Kalugin}
\title{On categorical approach to Verdier Duality}

\dedicatory{To the memory of my grandmother and my grand grandmother}
\maketitle
\begin{abstract}In present paper we develop categorical formalism of Verdier duality.
\end{abstract}

\tableofcontents

\section{Introduction}

\subsection{Classical picture} To the scheme $\mathrm X$ one can associate $\'{e}$tale topos $\mathrm X_{\'{e}\mathrm t}$ with corresponding derived category $\mathrm D^{\mathrm b}_{\mathrm c}(\mathrm X,\Lambda)$ of constructible sheaves on $\mathrm X.$ Classical approach \cite{SGA4(3)} to six functor formalism states that if we have morphism of schemes $\mathrm f\colon \mathrm X\longrightarrow \mathrm Y$ we have associated morphism of topoi
\begin{equation*}
\mathrm f\colon \mathrm X_{\'{e}\mathrm t} \longrightarrow \mathrm Y_{\'{e} \mathrm t},
\end{equation*}
with corresponding derived functors:
\begin{equation*}
\underline{\mathrm L}^*(\mathrm f^*)\colon \mathrm D^{\mathrm b}_{\mathrm c}(\mathrm Y,\Lambda)\longleftrightarrow \mathrm D^{\mathrm b}_{\mathrm c}(\mathrm X,\Lambda)\colon\underline{\mathrm R}^*(\mathrm f_*)
\end{equation*}
and $!$-operations:

\begin{equation*}
\underline{\mathrm R}^*(\mathrm f_!)\colon \mathrm D^{\mathrm b}_{\mathrm c}(\mathrm Y,\Lambda)\longleftrightarrow \mathrm D^{\mathrm b}_{\mathrm c}(\mathrm X,\Lambda)\colon\mathrm f^!.
\end{equation*}
Also we have tensor product $\otimes$ on category $\mathrm D^{\mathrm b}_{\mathrm c}(\mathrm X,\Lambda)$ and inner hom functor. These functors satisfies natural compatibilities, such as base change property \cite{SGA5}. Starting from such data, Verdier duality functor $\mathbb D$ for category $\mathrm D^{\mathrm b}_{\mathrm c}(\mathrm X,\Lambda)$ can be defined \cite{SGA5}:
 \begin{equation*}
  \mathbb D\colon \mathrm D^{\mathrm b}_{\mathrm c}(\mathrm X,\Lambda)\longrightarrow\mathrm D^{\mathrm b}_{\mathrm c}(\mathrm X,\Lambda)^{\circ}
 \end{equation*}
The same approach works well in setting of constructible sheaves on topological space \cite{KS}, however, there are some situations when this approach fails. For example, when $\mathrm X$ is a smooth scheme we can consider derived category $\mathrm D^{\mathrm b}(\mathrm {Mod}(\EuScript D_{\mathrm X}))$ of right $\EuScript D$-modules on $\mathrm X.$ In general classical six functor formalism for $\EuScript D$-modules does not exists, but Verdier duality functor can be defined \cite{Ber}.
We present axiomatic approach to the problem of construction Verdier duality and six functors formalism.

\subsection{Covariant duality} Object for which we propose construction of Verdier duality is the \textit{ringed $D$-topos} $(\mathrm E,\mathrm A)$ in the sense of \cite{SGA4(2)}($\mathrm D$-topos $(\mathrm E,\mathrm A)$ is the category bifibered in topoi, over category $\mathrm D,$ with a ringed object $\mathrm A).$ Our framework for construction of Verdier duality functor is so called \textit{cross functor} $(\mathrm H_{\star},\mathrm H^{\star},\mathrm H_!,\mathrm H^!)$ in the sense of Voevodsky and Deligne \cite{Voe}. That is quadruple of $2$-functors:
\begin{equation*}
\mathrm H_{\star},\mathrm H_{!}\colon \mathrm D\longrightarrow \mathrm {Cat},\quad \mathrm H^{\star},\mathrm H^{!}\colon \mathrm D^{\circ}\longrightarrow \mathrm {Cat},
\end{equation*}
where $\mathrm H_{\star}$ and $\mathrm H^{\star}$ and $\mathrm H_{!}$ and $\mathrm H^{!}$ are adjoint and isomorphic on objects. Moreover we have natural coherence relations like base change property (Definition \ref{cr.d.5}). By the \textit{Grothendieck cross functor} associated with the ringed $\mathrm D$-topos $(\mathrm E,\mathrm A)$ we understand cross functor $(\mathrm H_{\star},\mathrm H^{\star},\mathrm H_!,\mathrm H^!),$ such that pseudo-functors $\mathrm H_!$ and $\mathrm H_{\star}$ are isomorphic and bifibration, which corresponds to $(\mathrm H_{\star},\mathrm H^{\star})$ is equivalent to the category of modules over ringed $\mathrm D$-topos $(\mathrm E,\mathrm A)$ (Definition \ref{d4.3.1}).

To every cross functor we can associate pair of pseudo-functors $(\mathrm H^{\star!},\mathrm H_{!\star}),$ where $\mathrm H^{\star!}$ is called\textit{ Mackey $\star!$-functor} and $\mathrm H_{!\star}$ is called \textit{Mackey $!\star$-functor}. These pseudo-functors acts from the category of correspondences $\mathrm {cospan}(\mathrm D)$ to the category $\mathrm {Cat}:$
\begin{equation*}
\mathrm H_{!\star},\mathrm H^{\star!}\colon \mathrm {cospan}(\mathrm D)\longrightarrow \mathrm {Cat},
\end{equation*}
such that, when restricting to $\mathrm D$ Mackey pseudo-functor $\mathrm H^{\star!}$ (resp. $\mathrm H_{!\star}$) is isomorphic to $\mathrm H^{\star}$ (resp. $\mathrm H_{\star}$) and when restricting to $\mathrm D^{\circ}$ is isomorphic to $\mathrm H^!$ (resp. $\mathrm H^!$):
\begin{equation*}
\mathrm i^*(\mathrm H^{\star!})\overset{\sim}{\longrightarrow} \mathrm H_!\quad  \mathrm i^*(\mathrm H_{!\star})\overset{\sim}{\longrightarrow} \mathrm H^!\quad \mathrm j^*(\mathrm H^{\star!})\overset{\sim}{\longrightarrow} \mathrm H^{\star}\quad \mathrm j^*(\mathrm H_{!\star})\overset{\sim}{\longrightarrow} \mathrm H_{\star},
\end{equation*}
where $\mathrm i$ and $\mathrm j$ are natural embedding functors:
\begin{equation*}
\mathrm i\colon \mathrm D\hookrightarrow \mathrm {cospan}(\mathrm D)\hookleftarrow \mathrm D^{\circ}\colon \mathrm j.
\end{equation*}
Denote by $\mathrm D(\underline{\Gamma}(\mathrm {H}^{\star}))$ and $\mathrm D^*(\underline{\Gamma}(\mathrm {H}^{!})$ derived categories of sections of corresponding fibrations. Covariant Verdier duality functors $\mathbb V_{\star\mapsto!}$ and $\mathbb V_{!\mapsto\star}$ act between categories $\mathrm D(\underline{\Gamma}(\mathrm {H}^{\star}))$ and $\mathrm D^*(\underline{\Gamma}(\mathrm {H}^{!}):$
\begin{equation}
\mathbb V_{\star\mapsto!}\colon \mathrm D(\underline{\Gamma}(\mathrm H^{\star}))\longleftrightarrow\mathrm D(\underline{\Gamma}(\mathrm H^{!}))\colon \mathbb V_{!\mapsto\star},
\end{equation}
We define \textit{covariant Verdier duality functors }as composition of Kan extension to the category $\mathrm D(\underline{\Gamma}(\mathrm {H}^{\star!})$ (resp. $\mathrm D(\underline{\Gamma}(\mathrm {H}_{!\star})$) and restriction functor (Definition \ref{d4.3.3}).  \par\medskip

We mostly interested in subcategories $\mathrm D_{\mathrm {cocart}}(\underline{\Gamma}(\mathrm {H}^{\star}))$ and $\mathrm D_{\mathrm {cart}}(\underline{\Gamma}(\mathrm {H}^{!}))$ consisting of objects with cocartesian (resp. cartesian) cohomology and generally Verdier duality functors do not preserve these subcategories. To fix this we introduce additional pair of pseudo-functors $(\mathrm H^{\star\star},\mathrm H_{!!}),$ (Definition \ref{d4.3.4} and Definition \ref{d4.3.5}):
\begin{equation*}
\mathrm H_{!!},\mathrm H^{\star\star}\colon \mathrm {cospan}(\mathrm D)\longrightarrow \mathrm {Cat},
\end{equation*}
with properties:
\begin{equation*}
\mathrm i^*(\mathrm H^{\star\star})\overset{\sim}{\longrightarrow}  \mathrm H^{\star\circ}\quad   \mathrm i^*(\mathrm H_{!!})\overset{\sim}{\longrightarrow}  \mathrm H^{!!}\quad \mathrm j^*(\mathrm H^{\star\star})\overset{\sim}{\longrightarrow}  \mathrm H^{\star}\quad \mathrm j^*(\mathrm H_{!!})\overset{\sim}{\longrightarrow} \mathrm H^{!\circ},
\end{equation*}
where by $\circ$ we denote the opposite pseudo-functor. These pseudo-functors are called \textit{Mackey $\star\star$-functor} and \textit{Mackey $!!$-functor}. Analogically to the definition of pair $(\mathrm H^{\star!},\mathrm H_{!\star})$, these pseudo-functors can be canonically associated with the Grothendieck cross functor. We define functors $\Xi_{\star}$ and $\Xi_{!}$ analogically to Verdier duality functors. Functor $\Xi_{\star}$ is defined as compositions of Kan extension functor to category $\mathrm D(\underline{\Gamma}(\mathrm H_{\star\star}))$ and restriction functor:
\begin{equation*}
\Xi_{\star}\colon \mathrm D(\underline{\Gamma}(\mathrm {H}^{\star}))\longrightarrow\mathrm D_{\mathrm {cocart}}(\underline{\Gamma}(\mathrm {H}^{\star})),
\end{equation*}
And functor $\Xi_{!}$ is defined as compositions of Kan extension functor to category $\mathrm D(\underline{\Gamma}(\mathrm H^{!!}))$ and restriction functor:
\begin{equation*}
\Xi_{!}\colon \mathrm D(\underline{\Gamma}(\mathrm {H}^{!}))\longrightarrow\mathrm D_{\mathrm {cart}}(\underline{\Gamma}(\mathrm {H}^{!})),
\end{equation*}
Functors $\Xi_{\star}$ and $\Xi_{!}$ are adjoint to inclusion functors $\Xi^{!}\colon \mathrm D_{\mathrm {cart}}(\underline{\Gamma}(\mathrm {H}^{!}))\rightarrow \mathrm D(\underline{\Gamma}(\mathrm {H}^{!}))$ and $\Xi^{\star}\colon \mathrm D_{\mathrm {cocart}}(\underline{\Gamma}(\mathrm {H}^{!}))\rightarrow \mathrm D(\underline{\Gamma}(\mathrm {H}^{\star})).$ Then we define Verdier duality functors:
\begin{equation}
\mathbb V_{\star\mapsto!}^{\mathrm {cart}}\colon \mathrm D_{\mathrm {cocart}}(\underline{\Gamma}(\mathrm H^{\star}))\longleftrightarrow\mathrm D_{{\mathrm {cart}}}(\underline{\Gamma}(\mathrm H^{!}))\colon \mathbb V_{!\mapsto\star}^{\mathrm {cocart}},
\end{equation}
As compositions of Verdier duality functors and functors $\Xi_!$ and $\Xi_{\star}.$
\subsection{Contravariant duality} In order to define contravariant Verdier duality functors we introduce notion of \textit{Grothendieck cross functor with $!$-dualizing object} as the cross functor $(\mathrm H_{\star},\mathrm H^{\star},\mathrm H_!,\mathrm H^!)$ with distinguished object $\mathrm w_{\mathrm D}$ in $\mathrm D(\underline{\Gamma}(\mathrm {H}^{!})$ (Definition \ref{d4.3.8}). With every Grothendieck cross functor with $!$-dualizing object we can associate \textit{duality functors} $\mathbf D_{!\mapsto \star}$ and $\mathbf D_{\star\mapsto!}$ which are functors acting between categories $\mathrm D(\underline{\Gamma}(\mathrm {H}^{!})$ and $\mathrm D(\underline{\Gamma}(\mathrm {H}^{\star})$ (Definition \ref{d4.3.10}):
\begin{equation*}
\mathbf D_{\star\rightarrow !}\colon \mathrm D(\underline{\Gamma}(\mathrm {H}^{\star})\longleftrightarrow  \mathrm D(\underline{\Gamma}(\mathrm {H}^{!})  \colon \mathbf D_{!\mapsto\star}
\end{equation*}
For Grothendieck cross functor $(\mathrm H_{\star},\mathrm H^{\star},\mathrm H_!,\mathrm H^!)$ with $!$-dualizing object $\mathrm w_{\mathrm D}$ we can define \textit{contravariant Verdier duality functors} $\mathbb D_{(\mathrm E,\mathrm A)}$ and $\mathbb D_{\mathrm H^{!}}$ as composition of above duality functors with covariant Verdier duality functors (Definition \ref{d4.3.10}):
\begin{equation}
\mathbb D_{(\mathrm E,\mathrm A)}\colon \mathrm D_{\mathrm {cocart}}(\underline{\Gamma}((\mathrm E,\mathrm A)))\longrightarrow\mathrm D_{\mathrm {cocart}}(\underline{\Gamma}((\mathrm E,\mathrm A)))^{\circ}
\end{equation}
\begin{equation}
\mathbb D_{\mathrm H^{!}}\colon \mathrm D_{\mathrm {cart}}(\underline{\Gamma}(\mathrm H^{!}))\longrightarrow\mathrm D_{\mathrm {cart}}(\underline{\Gamma}(\mathrm H^{!}))^{\circ}
\end{equation}
From classical point of view (where we consider category of sheaves on locally compact space or scheme in $\'{e}$tale topology) Verdier duality functor $\mathbb V_{\star\mapsto !}$ plays role of functor, which takes sheaf $\EuScript F$ to corresponding cosheaf of compact support sections \cite{Lu}. Functor $\mathbf D_{!\mapsto\star}$ is the linear duality functor, which takes cosheaf to corresponding linear dual sheaf and functor $\mathbb D_{(\mathrm E,\mathrm A)}$ corresponds to original Verdier duality functor \cite{Ved} \cite{SGA5}.

\subsection{Mackey functors} Our definition of Mackey pseudo-functors can be considered as $2$-categorical analog of the classical notion of Mackey functor. By a \textit{Mackey functor} \cite{Lin} over category $\mathrm D$ we understand functor $\mathrm M$ from the category $\mathrm {cospan}(\mathrm D)$ to the category of abelian groups $\mathrm {Ab}$ (more generally with values in the category of spectra $\EuScript S\mathrm p$). It is well known \cite{Lin} that one can consider Mackey functor $\mathrm M$ as a pair of functors $(\mathrm i^*\mathrm M,\mathrm j^*\mathrm M)$, which satisfies base-change property. It is interesting problem try to construct functor from colimit of functor $\mathrm i^*\mathrm M$ to the limit of functor $\mathrm j^*\mathrm M:$
\begin{equation}
{}^1\mathbb V_{!\mapsto \star}^{\mathrm{full}}\colon \clim_{\mathrm D^{\circ}}\mathrm i^*\mathrm M\longrightarrow \lim_{{\mathrm D}}\mathrm j^*\mathrm M
\end{equation}
Examples of such morphisms appear in different contexts (See \cite{GM},\cite{LH} and \cite{HKR}) and perhaps the simplest one is given by \textit{Norm map} between homology and cohomology of finite group $\mathrm G$ with coefficient in representation $\mathrm L:$
\begin{equation*}
 \mathrm {Nm}\colon \mathrm C_{\hdotc}(\mathrm G,\mathrm L)\longrightarrow \mathrm C^{\hdot}(\mathrm G,\mathrm L)
\end{equation*}
Covariant Verdier duality functors can be considered as categorical analog of morphism ${}^1\mathbb V_{!\mapsto \star}^{\mathrm{full}}.$ By correspondence between higher Picard groupoids and connective spectra we will explain how our formalism can be applied to construct such morphisms (Subsection \ref{decat}).

\subsection{Six operations} Let $\mathrm f\colon (\mathrm E,\mathrm A)\longrightarrow (\mathrm E',\mathrm A')$ be a morphism of $\mathrm D$-topoi, we have corresponding functors:
\begin{equation*}
\mathrm f_{\star}\colon \mathrm D_{\mathrm {cocart}}(\underline{\Gamma}(\mathrm E,\mathrm A))\longleftrightarrow\mathrm D_{\mathrm {cocart}}(\underline{\Gamma}(\mathrm E',\mathrm A'))\colon \mathrm f^{\star},
\end{equation*}
where functor $\mathrm f_{\star}$ is defined as composition of derived pushforward functor $\underline{\mathrm R}(\mathrm f_*)$ and functor $\Xi_{\star}.$ Note that functors $\mathrm f_{\star}$ and $\mathrm f^{\star}$ are adjoint in the standard way. We can also define \textit{$!$-operations}:
\begin{equation*}
\mathrm f_{!}\colon \mathrm D_{\mathrm {cocart}}(\underline{\Gamma}(\mathrm E,\mathrm A))\longleftrightarrow\mathrm D_{\mathrm {cocart}}(\underline{\Gamma}(\mathrm E',\mathrm A'))\colon \mathrm f^{!},
\end{equation*}
by the rule:
\begin{equation*}
\mathrm f_{!}:=\mathbb D_{(\mathrm E,\mathrm A)}\circ\mathrm f_{\star}\circ\mathbb D_{(\mathrm E,\mathrm A)},\qquad \mathrm f^{!}:=\mathbb D_{(\mathrm E,\mathrm A)}\circ\mathrm f^{\star}\circ\mathbb D_{(\mathrm E,\mathrm A)}.
\end{equation*}
We also have tensor product $\otimes^{\star},$ defined componentwise and inner hom functor $\underline{\mathrm {Hom}}^{\star},$ which make $\mathrm D_{\mathrm {cocart}}(\underline{\Gamma}(\mathrm E,\mathrm A))$ into closed monoidal category. When $\mathrm D$-topos $(\mathrm E,\mathrm A)$ is associated with Artin stack $\EuScript X$, our construction gives six functors formalism for $\EuScript X$ \cite{LO}. I hope to elaborate on this elsewhere.

\subsection{Applications} Apart from applications, which were already mentioned we have following. Consider diagram\footnote{By diagram we understand functor where all morphisms $\mathrm X_{\mathrm i}\longrightarrow \mathrm X_{\mathrm j}$ are proper, such diagrams are called pseudo-schemes in \cite{Gai1}, we can also consider diagrams of more generally type, but constructions become more involved see Remark \ref{r4.4.1}} of topological spaces (schemes) $\mathrm X_{\EuScript J},$ where $\EuScript J$ is some category. We have natural $\mathrm {Top}^{\circ}$-topos $\mathbf {Sh}_{\mathrm {top}},$ whose fiber over object $\mathrm X\in \mathrm {Top}$ is given by category of sheaves on $\mathrm X,$ with corresponding pseudo-functors $(\mathrm H_{\star\mathrm{top}},\mathrm H^{\star}_{\mathrm {top}}).$ We can restrict this bifibration to category $\EuScript J,$ to obtain $\EuScript J$-topos $\mathbf {Sh}_{\mathrm {top}}^{\EuScript J}$. In this situation triangulated category $\mathrm D_{\mathrm {cocart}}(\underline{\Gamma}(\mathrm {H}_{\EuScript J}^{\star}))$ is the category of so called \textit{admissible $\star$-sheaves} on diagram $\mathrm X_{\EuScript J}.$ This category has natural $\mathrm t$-structure, whose heart is the category of sheaves on the colimit of diagram $\mathrm X_{\EuScript J}.$ Triangulated category $\mathrm D_{\mathrm {cart}}(\underline{\Gamma}(\mathrm {H}_{\EuScript J}^{!}))$ is the category of \textit{admissible $!$-sheaves} in the sense of Beilinson and Drinfeld \cite{BD}. By Proposition \ref{pts} category $\mathrm D_{\mathrm {cart}}(\underline{\Gamma}(\mathrm {H}_{\EuScript J}^{!}))$ can be equipped with perverse $\mathrm t$-structure (see also Remark \eqref{rps}).   And category $\mathrm D(\underline{\Gamma}(\mathrm {H}_{!\star}^{\EuScript J}))$ is what we call derived category of \textit{$!\star$-sheaves} on diagram $\mathrm X_{\EuScript J}$. Category $\mathrm D_{\mathrm {cart}}(\underline{\Gamma}(\mathrm {H}_{\EuScript J}^{!}))$ has unital tensor structure $\otimes^!,$ defined by componentwise $!$-product. With every morphism $\mathrm f\colon \mathrm X_{\EuScript J}\longrightarrow \mathrm Y_{\EuScript J}$ of diagrams we associate functors:
\begin{equation*}
\mathrm f_{!}\colon  \mathrm D_{\mathrm {cart}}(\underline{\Gamma}(\mathrm {H}_{\mathrm X_{\EuScript J}}^{!}))\longleftrightarrow \mathrm D_{\mathrm {cart}}(\underline{\Gamma}(\mathrm {H}_{\mathrm Y_{\EuScript J}}^{!}))\colon \mathrm f^{!},
\end{equation*}
defined analogically to the case of $\star$-sheaves and we can also associate $\star$-operations for $!$-sheaves:
\begin{equation*}
\mathrm f_{\star}\colon  \mathrm D_{\mathrm {cart}}(\underline{\Gamma}(\mathrm {H}_{\mathrm X_{\EuScript J}}^{!}))\longleftrightarrow \mathrm D_{\mathrm {cart}}(\underline{\Gamma}(\mathrm {H}_{\mathrm Y_{\EuScript J}}^{!}))\colon \mathrm f^{\star},
\end{equation*}
In addition to above operations one can also define Kan extension functors. Let $\mathrm j\colon \EuScript I\longrightarrow \EuScript J$ be a functor, then we have a pairs of adjoint functors:
\begin{equation*}
\mathrm j_{!}\colon  \mathrm D_{\mathrm {cart}}(\underline{\Gamma}(\mathrm {H}_{\mathrm X_{\EuScript I}}^{!}))\longleftrightarrow \mathrm D_{\mathrm {cart}}(\underline{\Gamma}(\mathrm {H}_{\mathrm X_{\EuScript J}}^{!}))\colon \mathrm j^{*},
\end{equation*}
where $\mathrm j^*$ is the restriction functor and $\mathrm j_!$ is defined as composition of left Kan extension functor and functor $\Xi_!.$ When diagram $\mathrm X_{\EuScript J}$ is associated with Ran prestack our constructions give Koszul duality for factorizable sheaves (chiral algebras, $\mathrm n$-algebras). Perverse $\mathrm t$-structure on $!$-sheaves is important in construction of Hopf algebras from factorizable sheaves and quantization of Lie bialgebras. When diagram $\mathrm X_{\EuScript J}$ is associated with de Rham prestack our constructions give six operations for $\EuScript D$-modules. I hope to elaborate on these examples elsewhere.

\subsection{Content} In the Section $2$ we recall basic facts about topoi and $\mathrm D$-topoi, mostly following \cite{SGA4(2)}, \cite{SGA4(2)} and \cite{Beh}. \par\medskip

Section $3$ is a main part of this paper. First we recall definition of a cross functor, following \cite{Voe} and some definitions of morphisms in monoidal categories, which were introduced in \cite{May}. Then we give definition of Grothendieck cross functor as well as definitions of Mackey pseudo-functors and Verdier duality functors. We also study properties of these functors. In Subsection \ref{tst} we construct interesting $\mathrm t$-structures on categories  $\mathrm D_{\mathrm {cocart}}(\underline{\Gamma}(\mathrm E,\mathrm A))$ and  $\mathrm D_{\mathrm {cart}}(\underline{\Gamma}(\mathrm {H}_{\mathrm X_{\EuScript J}}^{!})).$ In Subsection \ref{sixtop} we construct six operations formalism for ringed $\mathrm D$-topos $(\mathrm E,\mathrm A)$ with Grothendieck cross functor.  In Subsection \ref{decat} we explain how our definition of Mackey pseudo-functors is related to original definition of Mackey functors and how categorical Verdier duality can be useful to construct analogs of Norm morphism \cite{GM} \cite{LH}.
In Section $4$ we construct Grothendieck cross functor for diagrams of topological spaces (schemes) and define Verdier duality functors. We also construct six functors formalism for $!$-sheaves on diagrams of topological spaces.\par\medskip

\subsection{Acknowledgments} We wish to thank Boris Feigin, Michael Finkelberg, Dennis Gaitsgory and Dmitry Kaledin for fruitful discussions. Special thanks are due to Nikita Markarian for countless illuminating discussions on this subject.

\subsection{Preliminaries} For a category $\mathrm C$ we denote by $\mathrm C^{\circ}$ opposite category, for a morphism $\mathrm f\colon \mathrm x \rightarrow \mathrm y$ in $\mathrm C$ we denote by $\mathrm f^{\circ}$ corresponding morphism in opposite category $\mathrm C^{\circ}.$ By $[\mathrm n]$ we denote finite ordinal category $[\mathrm n]:=\{0\rightarrow 1\rightarrow\cdots\rightarrow\mathrm n\}.$ By $\Delta$ we denote simplex category. It is category whose objects are nonempty finite ordinals $[\mathrm n]$ and order preserving maps between them. We have subcategory $\Delta_+\subset \Delta,$ with same objects and morphisms given by order preserving maps, which send maximal element to maximal. For every $\mathrm n\in \mathbb N$ we denote by $\Delta_{\leq \mathrm n}$ (resp. $\Delta_{+\leq \mathrm n}$) full subcategory of $\Delta$ (resp. $\Delta_+$) on the objects $[0],[1],\ldots[\mathrm n].$ For a category $\mathrm C$ with fiber products we denote by $\mathrm {span}(\mathrm C)$ category of correspondences in $\mathrm C.$ Objects of this category are the same as in $\mathrm C$ and morphisms $\mathrm f\colon \mathrm x\rightarrow \mathrm y$ given by equivalence classes of "roofs" $\mathrm x\leftarrow\mathrm z\rightarrow\mathrm y,$ where $\mathrm z$ is an object of $\mathrm C,$ such morphisms are denoted by triples $(\mathrm f_{\mathrm l},\mathrm z,\mathrm f_{\mathrm r}),$ two roofs $\mathrm x\leftarrow \mathrm z\rightarrow \mathrm y,$ and $\mathrm x\leftarrow \mathrm z'\rightarrow \mathrm y,$ are said to be equivalent if we have an isomorphism $\mathrm z \cong \mathrm z',$ which is compatible with maps of roofs. We have embedding of categories $\mathrm C \hookrightarrow \mathrm {span}(\mathrm C)\hookleftarrow \mathrm C^{\circ}.$ For the dual category $\mathrm C^{\circ}$ we will use a notation $\mathrm {cospan}(\mathrm C):=\mathrm {span}(\mathrm C^{\circ}).$ Let $\mathrm E\longrightarrow \mathrm B$ be a fibration (resp. cofibration) in the sense of ~\cite{SGA1}, between categories $\mathrm E$ and $\mathrm B.$ For every morphism $\mathrm m\colon \mathrm i \rightarrow \mathrm j$ in $\mathrm B$ we have a functor $\mathrm m^*\colon  \mathrm E_{\mathrm j} \longrightarrow \mathrm E_{\mathrm i}$ ($\mathrm m_*\colon  \mathrm E_{\mathrm i} \longrightarrow \mathrm E_{\mathrm j}$), where $\mathrm E_{\mathrm i}$ is a fiber over $\mathrm i.$ For a $\mathrm B$-functor $\varphi\colon \mathrm E \longrightarrow \mathrm E'$ between two fibrations (resp. cofibrations) we denote by $\varphi_{\mathrm i}\colon \mathrm E_{\mathrm i} \longrightarrow \mathrm E'_{\mathrm i}$ restriction of a functor $\varphi$ to a fiber over $\mathrm i.$ Category of $\mathrm B$-functors will be denoted by $\underline{\mathrm{Hom}}_{\mathrm B}(\mathrm E,\mathrm E').$ Category of sections of fibration (resp. cofibration) $\mathrm E$ will be denoted by $\underline{\Gamma}(\mathrm E):=\underline{\mathrm{Hom}}_{\mathrm B}(\mathrm B,\mathrm E),$ we also have a subcategory $\underline{\Gamma}^{\mathrm {cart}}(\mathrm E),$ (resp. $\underline{\Gamma}^{\mathrm {cocart}}(\mathrm E))$ of cartesian (resp. cocartesian) sections. Fibration (resp. cofibration) is called abelian if every fiber is an abelian category and corresponding functors between fibers are right exact. If $\mathrm B^{\circ}$ is a monoidal category and $\mathrm E$ is a symmetric monoidal category, such that the underlying tensor product functor in $\mathrm E$ preserves colimits in each variable, then category of sections $\underline{\Gamma}(\mathrm E)$ can be endowed with a monoidal structure via Day convolution product, denoted by $\boxtimes.$ For fibration (resp. cofibration) $\mathrm E\longrightarrow\mathrm B$ we denote by $\mathrm E^{\circ}\longrightarrow\mathrm B$ cartesian (resp. cocartesian) dual fibration (resp. cofibration).

Let $\mathrm B\longrightarrow \mathrm D$ be an abelian fibration over $\mathrm D.$ We have an abelian category $\underline{\Gamma}(\mathrm B)$ of sections of the fibration $\mathrm B\longrightarrow \mathrm D.$ We can associate with it triangulated category $\mathrm D^*(\underline{\Gamma}(\mathrm B))$ where $*=+,\emptyset.$ Denote by $\mathrm D_{\mathrm{cart}}^*(\underline{\Gamma}(\mathrm B)),$ where $*=+,\emptyset,$ triangulated subcategory consisting of complexes with cartesian cohomology. If category $\mathrm B$ was a cofibered over $\mathrm D$ we denote by $\mathrm D_{\mathrm{cocart}}^*(\underline{\Gamma}(\mathrm B)),$ where $*=+,\emptyset,$ triangulated subcategory consisting of complexes with cocartesian cohomology.
\section{Topoi}
\subsection{Notations} First of all we want to recollect some facts about topoi, following \cite{SGA4(1)}. Let $\mathrm T$ and $\mathrm T'$ be topoi. Then a morphism of topoi $\varphi\colon \mathrm T\longrightarrow \mathrm T'$ is given by a pair of functors $\varphi=(\varphi_*,\varphi^*),$ where $\varphi_*\colon \mathrm T\longrightarrow \mathrm T'$ is right adjoint to functor $\varphi^*\colon \mathrm T'\longrightarrow \mathrm T.$ Functor $\varphi^*$ must preserves finite limits. Morphism $\varphi\colon \mathrm T\longrightarrow \mathrm T'$ of topoi is called an embedding if $\varphi_*$ is fully-faithful. Embedding $\varphi$ is called closed (resp. open) if essential image of $\varphi_*$ is closed (resp. open) subtopos of $\mathrm T'.$ Ringed topos $(\mathrm T,\EuScript O_{\mathrm T})$ is a pair, where $\mathrm T$ is a topos and $\EuScript O_{\mathrm T}$ is a ring object in $\mathrm T.$ Morphism between ringed topoi $(\mathrm T,\EuScript O_{\mathrm T})$ and $(\mathrm T',\EuScript O_{\mathrm T'})$ is a pair $(\varphi,\theta)$, where $\varphi$ is a morphism of topoi and $\theta\colon \EuScript O_{\mathrm T'}\longrightarrow \varphi_*\EuScript O_{\mathrm T}$ is a morphism of rings. With a ringed topos $(\mathrm T,\EuScript O_{\mathrm T})$ we can associate an abelian category of modules over a topos ${\mathrm {Mod}}(\mathrm T,\EuScript O_{\mathrm T}).$ This category will be complete and cocomplete. Also filtered colimits are exact in ${\mathrm {Mod}}(\mathrm T,\EuScript O_{\mathrm T}).$ Morphism of ringed topoi $(\varphi,\theta)$ is called a closed embedding if $\varphi$ is a closed embedding of topoi and morphism $\theta\colon \EuScript O_{\mathrm T'}\longrightarrow \varphi_*\EuScript O_{\mathrm T}$ is surjective. In this case functor $\varphi_*\colon {\mathrm {Mod}}(\mathrm T,\EuScript O_{\mathrm T})\longrightarrow {\mathrm{Mod}}(\mathrm T',\EuScript O_{\mathrm T'})$ is exact.

A stratification $\EuScript S$ of a topos $\mathrm T$ is a finite number of locally closed subtopoi $\mathrm i_{\mathrm s}\colon\mathrm T_{\mathrm s}\rightarrow \mathrm T,$ where $\mathrm s\in \EuScript  S,$ called strata, such that $\mathrm T$ is a disjoint union of strata and closure of each stratum is an union of strata. Pair $(\mathrm T,\EuScript S)$ will be called a stratified topos. An object $\EuScript K\in \mathrm T$ is called a locally constant sheaf if it's locally isomorphic to a constant object associated with a finite set. An object $\EuScript K$ in a stratified topos $(\mathrm T,\EuScript  S)$ called constructible with respect to a stratification $\EuScript S$ if each restriction $\mathrm i_{\mathrm s}^*\EuScript K$ is a locally constant sheaf. An object $\EuScript K\in {\mathrm{Mod}}(\mathrm T,\EuScript O_{\mathrm T})$ is a locally constant if it's locally constant as a sheaf of sets. We will denote a weak Serre subcategory of modules of locally constant sheaves  by ${\mathrm {Mod}}_{\mathrm {lcc}}(\mathrm T,\EuScript O_{\mathrm T}).$ We also have a weak Serre subcategory ${\mathrm {Mod}}_{\mathrm c}(\mathrm T,\EuScript O_{\mathrm T})$ of modules of constructible sheaves.

\subsection{$\mathrm D$-topoi} We are going to recollect some facts about fibered topoi, partially following ~\cite{SGA4(2)}. Let $\mathrm D$ be a small category.

\begin{Def}\label{d3.2.1} Category $\mathrm E$ is called a \underline{$\mathrm D$-topos} if we have a bifibration $\mathrm E\longrightarrow\mathrm D$, such that for every $\mathrm i\in \mathrm D$ corresponding fiber $\mathrm{E}_{\mathrm i}$ is a topos and for every morphism $\mathrm m\colon \mathrm i \longrightarrow \mathrm j$ we have $\mathrm m^*=\mathrm f_*$ $\mathrm m_*=\mathrm f^*,$ where $\mathrm f=(\mathrm f_*,\mathrm f^*)\colon \mathrm{E}_{\mathrm j} \longrightarrow \mathrm{E}_{\mathrm i}$ is a morphism of topoi.
\end{Def}
\begin{Lemma}\label{l3.2.1} Let $\mathrm A$ be a bifibration over $\mathrm J,$ suppose that we have functor $\mathrm f\colon \mathrm I \longrightarrow \mathrm J.$ We can define pullback of bifibration by the rule $\mathrm I\times_{\mathrm J}\mathrm A.$ We have forgetful functor:
\begin{equation}\label{e3.2.1}
\mathrm f^*\colon \underline{\Gamma}(\mathrm A)\longrightarrow \underline{\Gamma}(\mathrm I\times_{\mathrm J}\mathrm A)
\end{equation}
If each fiber of bifibration $\mathrm A$ is complete and cocomplete then functor $\mathrm f^*$ posses a right adjoint $\mathrm f_*$ (resp. a left adjoint $\mathrm f_!$). Functor $\mathrm f_*$ is called \underline{right Kan extension} and functor $\mathrm f_!$ is called \underline{left Kan extension}.
\end{Lemma}
\begin{proof} See \cite[Expos\'{e} $\mathrm V^{\mathrm{bis}}$, Lemma~1.2.10.1]{SGA4(2)}.

\end{proof}
Let $\mathrm f\colon \mathrm D'\longrightarrow \mathrm D$ be a functor, we can define category $\mathrm D'\times_{\mathrm D} \mathrm E$ which will be a $\mathrm D'$-topos. Following lemma will be very useful to us:
\begin{Cor}\label{l3.2.1} Let $\mathrm E$ be a $\mathrm D$-topos, a restriction functor $\mathrm f^*$ has a right adjoint $\mathrm f_*$ (resp. a left adjoint $\mathrm f_!$).
\end{Cor}

Let $\mathrm e_{\mathrm i}\colon \mathrm i\longrightarrow \mathrm D,$ be an inclusion of an object $\mathrm i$ as a category with a unique object. Then we a have functor $\mathrm e_{\mathrm i}^*\colon \underline{\Gamma}(\mathrm E)\longrightarrow \mathrm E_{\mathrm i},$ called evaluation at point $\mathrm i.$ This functor has a left adjoint $\mathrm e_{\mathrm i!}$ (resp. a right adjoint $\mathrm e_{\mathrm i*}).$

\begin{Prop}\label{p3.2.1} Let $\mathrm E$ be a $\mathrm D$-topos, then  category $\underline{\Gamma}(\mathrm E)$ and a subcategory of cocartesian sections $\underline{\Gamma}^{\mathrm {cocart}}(\mathrm E)$ are topoi. We have a morphism of topoi
\begin{equation}\label{e3.2.2}
\Xi\colon  \underline{\Gamma}^{}(\mathrm E)\longrightarrow \underline{\Gamma}^{\mathrm {cocart}}(\mathrm E),
\end{equation}
where $\Xi^*$ is an inclusion functor.
\end{Prop}
\begin{proof} See \cite[Expos\'{e} $\mathrm V^{\mathrm{bis}}$, Proposition~1.2.12]{SGA4(2)} for $\underline{\Gamma}(\mathrm E).$ The same argument can be applied to proof that $\underline{\Gamma}^{\mathrm {cocart}}(\mathrm E)$ is a topos. Existence of a functor $\Xi_*$ follows from an adjoint functor theorem.

\end{proof}

\begin{Def}\label{d3.2.2} \underline{Morphism $\Phi$ of $\mathrm D$-topoi $\mathrm E$ and $\mathrm E'$} is a couple $(\Phi_*,\Phi^*)$ of adjoint $\mathrm D$-functors ($\Phi_*$ is a right adjoint to $\Phi^*$), such that for an every $\mathrm i\in \mathrm D$ we have a morphism $(\Phi_{*\mathrm i},\Phi^*_{\mathrm i})$ of topoi $\mathrm E_{\mathrm i}$ and $\mathrm E'_{\mathrm i}.$
\end{Def}

\begin{Prop}\label{p3.2.2} Morphism of $\mathrm D$-topoi $\Phi=(\Phi_*,\Phi^*)\colon \mathrm E\longrightarrow \mathrm {E'},$ induces a morphism $(\Gamma(\Phi_*),\Gamma(\Phi^*))\colon \underline{\Gamma}(\mathrm E)\longrightarrow \underline{\Gamma}(\mathrm E')$ of associated total topoi.

\end{Prop}
\begin{proof} See \cite[Expos\'{e} $\mathrm V^{\mathrm{bis}}$, Proposition~1.2.15 ]{SGA4(2)}.

\end{proof}
\begin{Ex}\label{ex3.2.1} Let $(\Phi_*,\Phi^*)$ be a morphism of $\mathrm D$-topoi. We call it open (resp. closed) if for every $\mathrm i\in \mathrm D$ a morphism of topoi $(\Phi_{*\mathrm i},\Phi^*_{\mathrm i})\colon \mathrm E_{\mathrm i} \longrightarrow \mathrm E_{\mathrm i}'$ is open (resp. closed). In this case morphism $(\Gamma(\Phi_*),\Gamma(\Phi^*))\colon \underline{\Gamma}(\mathrm E)\longrightarrow \underline{\Gamma}(\mathrm E')$ will also be open (resp. closed).

\end{Ex}

\begin{Def}\label{d3.2.3} \underline{Ringed $\mathrm D$-topos} is a couple $ (\mathrm E,\mathrm A),$ where $\mathrm E$ is a $\mathrm D$-topos and $\mathrm A$ is a ring object in monoidal category $\underline{\Gamma}(\mathrm E).$

\end{Def}
Ring object $\mathrm A$ is defined as ring $\mathrm A_{\mathrm i}$ in $\mathrm E_{\mathrm i}$ for every $\mathrm i$ and for every morphism $\mathrm m\colon \mathrm i\rightarrow \mathrm j$ we have morphism of rings $\mathrm A_{\mathrm j}\rightarrow \mathrm m_*(\mathrm A_{\mathrm i}).$

\begin{Def}\label{d3.2.4} Let $(\mathrm E,\mathrm A)$ and $(\mathrm E',\mathrm A')$ be ringed $\mathrm D$-topoi. Pair $(\Phi,\theta),$ where $\Phi\colon \mathrm E\longrightarrow \mathrm E'$ is a morphism of $\mathrm D$-topoi and $\theta\colon \mathrm A'\longrightarrow \Gamma(\Phi_*)(\mathrm A)$ is homomorphism of rings called \underline{Morphism  of ringed topoi $(\mathrm E,\mathrm A)$ and $(\mathrm E',\mathrm A')$}
\end{Def}

\begin{Prop}\label{p3.2.3} Morphism of ringed $\mathrm D$-topoi $(\Phi,\theta)$ induces a morphism of total ringed topoi $(\Gamma(\Phi),\theta)\colon (\underline{\Gamma}(\mathrm E),\mathrm A)\longrightarrow (\underline{\Gamma}(\mathrm E'),\mathrm A').$
\end{Prop}
\begin{proof} See \cite[Expos\'{e} $\mathrm V^{\mathrm{bis}}$, Remark~1.3.3 ]{SGA4(2)}.

\end{proof}
With a ringed $\mathrm D$-topos $(\mathrm E,\mathrm A)$ we can associate category $\underline{\mathrm {Mod}}(\mathrm E,\mathrm A),$ bifibered in abelian categories over a category $\mathrm D,$ whose fiber over $\mathrm i\in \mathrm D$ given by a category of modules $\mathrm {Mod}(\mathrm E_{\mathrm i},\mathrm A_{\mathrm i})$ over ringed topos $(\mathrm E_{\mathrm i},\mathrm A_{\mathrm i}).$ We denote abelian category of modules over total topos $\Gamma(\mathrm E)$ by $\mathrm {Mod}(\underline{\Gamma}(\mathrm E),\mathrm A).$ This category is identified with a category $\underline{\mathrm {Hom}}_{\mathrm D}(\underline{\mathrm {Mod}}(\mathrm E,\mathrm A))$ of sections of an abelian bifibration $\underline{\mathrm {Mod}}(\mathrm E,\mathrm A).$

\begin{Ex}\label{ex3.2.2} Let $(\mathrm S,\EuScript O_{\mathrm S})$ be ringed topos. We can associate with it ringed $\mathrm D$-topos $(\mathrm S\times \mathrm D,\EuScript O_{\mathrm S})$ called constant ringed $\mathrm D$-topos. By definition ${\mathrm {Mod}}(\underline{\Gamma}(\mathrm S\times \mathrm D),\EuScript O_{\mathrm S})$ is equivalent to a category of functors from $\mathrm D$ to $\underline{\mathrm {Mod}}(\mathrm S\times \mathrm D,\EuScript O_{\mathrm S}).$
\end{Ex}

Morphism $(\Phi,\theta)$ of ringed $\mathrm D$-topoi $(\mathrm E,\mathrm A)$ and $(\mathrm E',\mathrm A')$ induces additive functors $\varphi_*\colon  \underline{\mathrm {Mod}}(\mathrm E,\mathrm A)\longrightarrow \underline{\mathrm {Mod}}(\mathrm E',\mathrm A')$ and $\varphi^*\colon  \underline{\mathrm {Mod}}(\mathrm E',\mathrm A')\longrightarrow \underline{\mathrm {Mod}}(\mathrm E,\mathrm A),$ called direct and inverse image functors. When functor $\varphi^*$ is exact we call morphism $(\Phi,\theta)$ flat.
\begin{Ex}\label{ex3.2.3} $\mathrm{ (i)}$ For a ringed $\mathrm D$-topos $(\mathrm E,\mathrm A)$ an augmentation is a map of $\mathrm D$-topoi $\theta \colon \mathrm E \longrightarrow \mathrm S\times \mathrm D.$ Let $(\mathrm S\times \mathrm D,\EuScript O_{\mathrm S})$ be a constant ringed $\mathrm D$-topos, we have exact functor $\varepsilon^*\colon {\mathrm{Mod}}(\mathrm S,\EuScript O_{\mathrm S})\longrightarrow {\mathrm{Mod}}(\underline{\Gamma}(\mathrm S\times \mathrm D),\EuScript O_{\mathrm S}),$ which is defined by attaching to an object of ${\mathrm{Mod}}(\mathrm S,\EuScript O_{\mathrm S})$ a constant functor. Functor $\varepsilon^*$ has a right adjoint denoted by $\varepsilon_*.$ Functor $\varepsilon_*$ associates to an element of ${\mathrm{Mod}}(\underline{\Gamma}(\mathrm S\times \mathrm D),\EuScript O_{\mathrm S})$ it's limit over category $\mathrm D.$ We have functors:
\begin{equation}\label{e3.2.3}
\bar{\theta}^*:=\Gamma(\theta^*)\circ\varepsilon^*\colon \mathrm{Mod}(\mathrm S,\EuScript O_{\mathrm S})\longrightarrow \mathrm{Mod}(\underline{\Gamma}(\mathrm E),\mathrm A)
\end{equation}
and
\begin{equation}\label{e3.2.4}
\bar{\theta}_*:=\varepsilon_*\circ\Gamma(\theta_*)\colon \mathrm{Mod}(\underline{\Gamma}(\mathrm E),\mathrm A)\longrightarrow \mathrm{Mod}(\mathrm S,\EuScript O_{\mathrm S})
\end{equation}
Image of functor $\bar{\theta}^*$ lies in subcategory $\underline{\Gamma}^{\mathrm {cocart}}(\mathrm D,\mathrm {Mod}(\mathrm E,\mathrm A)).$

\end{Ex}

\begin{Prop}\label{p3.2.4} Let $\mathrm f\colon \mathrm D'\longrightarrow \mathrm D$ be a functor and $(\mathrm E,\mathrm A)$ be a ringed $\mathrm D$-topos. Restriction functor
\begin{equation}\label{e3.2.6}
\mathrm f^*\colon \mathrm {Mod}(\underline{\Gamma}(\mathrm E),\mathrm A)\longrightarrow \mathrm {Mod}(\underline{\Gamma}(\mathrm E\times_{\mathrm D} \mathrm D'),\mathrm A\circ\mathrm f)
\end{equation}
has a right and left adjoint functors denoted by $\mathrm f_*$ and by $\mathrm f_!,$ respectively.
\end{Prop}
\begin{proof} See \cite[Expos\'{e} $\mathrm V^{\mathrm{bis}}$, Proposition~1.3.7]{SGA4(2)}.

\end{proof}

\begin{remark} In the same spirit one can consider the same construction for bifibrations, where fibers are complete and cocomplete.

\end{remark}

For the ringed topos $(\mathrm T,\EuScript O_{\mathrm T})$ with stratification $\EuScript S$ we denote by $\mathrm D_{\mathrm c}(\mathrm T,\EuScript O_{\mathrm T})$ triangulated subcategory of $\mathrm D(\mathrm T,\mathrm A),$ consisting of complexes $\EuScript K$ such that $\mathrm H^{\mathrm n}(\EuScript K^{\hdot})\in \mathrm {Mod}_{\mathrm c}(\mathrm T,\mathrm A)$ for all $\mathrm n.$ This category has natural $\mathrm t$-structure, whose heart is equivalent to $\mathrm {Mod}_{\mathrm c}(\mathrm T,\mathrm A).$ We have natural triangulated comparison functor $\mathrm D(\mathrm {Mod}_{\mathrm c}(\mathrm T,\mathrm A))\longrightarrow \mathrm D_{\mathrm c}(\mathrm T,\mathrm A).$

\begin{Def}\label{d3.4.1} Let $(\mathrm E,\mathrm A)$ be ringed the $\mathrm D$-topos, \underline{a stratification $\EuScript S$} on $(\mathrm E,\mathrm A)$ is a stratification $\EuScript S_{\mathrm i}$ on $\mathrm E_{\mathrm i}$ for every $\mathrm i\in \mathrm D,$ such that for every $\mathrm m\colon \mathrm i\rightarrow \mathrm j\in \mathrm D$ functors
\begin{equation}\label{e3.4.1}
\mathrm f^*\colon {\mathrm{Mod}}(\mathrm E_{\mathrm i},\mathrm A) \longleftrightarrow{\mathrm{Mod}}(\mathrm E_{\mathrm j},\mathrm A)\colon \mathrm f_*,
\end{equation}
where $\mathrm f_*:=\mathrm m^*$ and $\mathrm f^*:=\mathrm m_*$ preserves constructible subcategories.
\end{Def}
Object $\EuScript K\in \mathrm{Mod}(\underline{\Gamma}(\mathrm E),\mathrm A)$ is called constructible with respect to $\EuScript S$ if for every $\mathrm i\in \mathrm D$ object $\EuScript K_{\mathrm i}$ is constructible. For a ringed $\mathrm D$-topos $(\mathrm E,\mathrm A)$ we have Serre subcategory ${\mathrm {Mod}}_{\mathrm c}(\mathrm E,\mathrm A)$ of constructible objects of topos ${\mathrm{Mod}}(\underline{\Gamma}(\mathrm E),\mathrm A).$
By $\mathrm D_{\mathrm c}(\Gamma(\mathrm E),\mathrm A),$ we denote triangulated subcategory of $\mathrm D(\Gamma(\mathrm E),\mathrm A),$ consisting of objects $\EuScript K,$ such that for every $\mathrm i\in \mathrm D$ complex $\EuScript K_{\mathrm i}$ has constructible cohomology. This category has obvious $\mathrm t$-structure, whose heart is equivalent to $\mathrm{Mod}_{\mathrm c}(\underline{\Gamma}(\mathrm E),\mathrm A).$ We also have triangulated subcategories $\mathrm D^{}_{\mathrm {cocart},\mathrm c}(\underline{\Gamma}(\mathrm E),\mathrm A)$ and $\mathrm D^{}_{\mathrm {cart},\mathrm c}(\underline{\Gamma}(\mathrm E),\mathrm A).$

\section{Cross functors and Verdier duality}
\subsection{Notations} Let $\mathrm D$ and $\mathrm C$ be a pair of $2$-categories, we can consider the $2$-category of $2$-functors (pseudo-functors) between $\mathrm D$ and $\mathrm C.$ When $\mathrm C$ is the $2$-category of small categories $\mathrm{Cat},$ and $\mathrm D$ is an ordinary category we have an equivalence between the $2$-category of $2$-functors $\underline{\mathrm {Hom}}(\mathrm D^{\circ},\mathrm {Cat})$ and the $2$-category of fibrations over $\mathrm D.$ This equivalence is given by the Grothendieck construction ~\cite{SGA1}. We always denote fibration and corresponding pseudo-functor by the same symbol. Given a pseudo-functor $\Psi\colon \mathrm D^{\circ}\longrightarrow \mathrm {Cat}$ we can consider it is colimit over category $\mathrm D^{\circ}.$ We have equivalence of categories:
\begin{equation}\label{87}
\clim_{\mathrm D^{\circ}}\Psi\overset{\sim}{\longrightarrow} \Psi[\mathrm S^{-1}],
\end{equation}
where $\Psi[\mathrm S^{-1}]$ is the localization of  total category of a fibration $\Psi$, with respect to the class $\mathrm S$ of cocartesian morphisms \cite[Expos\'{e} $\mathrm {VI}$, Section~6]{SGA4(2)}. We can also consider limit of pseudo-functor $\Psi,$ over category $\mathrm D^{\circ}.$ We have following equivalence:
\begin{equation}\label{87}
\lim_{\mathrm D^{\circ}}\Psi\overset{\sim}{\longrightarrow} \underline{\Gamma}_{\mathrm {cart}}(\Psi),
\end{equation}
where $\underline{\Gamma}_{\mathrm {cart}}(\Psi)$ is the category of cartesian sections of corresponding fibration. If fibers of fibration are cocomplete and for every morphism $\mathrm m\colon \mathrm i\longrightarrow \mathrm j\in \mathrm D$ functor $\mathrm m^*\colon  \Psi_{\mathrm j} \longrightarrow \Psi_{\mathrm j}$ commutes with colimits we have following canonical morphism:
\begin{equation}
\mathrm {can}_{!},\mathrm{can}_*\colon \clim_{\mathrm D^{\circ} }\Psi\longrightarrow\lim_{\mathrm D^{\circ} }\Psi
\end{equation}
Note, that if we have a bifibration $\Phi:=(\Psi,\Theta),$ where $\Psi$ is a fibration and $\Theta$ is cofibration over $\mathrm D,$ such that for every $\mathrm i\in \mathrm D$ fibers of $\Phi$ are cocomplete, then we have following equivalence:
\begin{equation}
 \clim_{{\mathrm D^{\circ} }}\Psi\overset{\sim}{\longrightarrow} \lim_{\mathrm D } \Theta
\end{equation}
With every $2$-category $\mathrm C$ one can associate $2$-category $\mathrm C^{1-\circ}$ with reversed $1$-morphism, $2$-category $\mathrm C^{2-\circ}$ with reversed $2$-morphism and $2$-category $\mathrm C^{12-\circ}$ with reversed $1$-morphisms and $2$-morphisms.

\subsection{Cross functors} Here we are going to recollect some facts about cross functors, following Deligne and Voevodsky \cite{Voe}. Let $\mathrm D$ be a category with the class $\mathrm B$ of commutative squares in $\mathrm D$ and $\mathrm C$ is the $2$-category.

\begin{Def} Pair $(\mathrm H_{\star},\mathrm H_!),$ where $\mathrm H_{\star}$ and $\mathrm H_!$ are $2$-functors from $\mathrm D$ to $\mathrm C,$ such that $\mathrm H_{\star}(\mathrm X)=\mathrm H_!(\mathrm x):=\mathrm H(\mathrm x)$ for every $\mathrm x\in \mathrm D,$ is called \underline{lower $\mathrm e$-functor} if for each square $\mathrm Q$ in $\mathrm B$

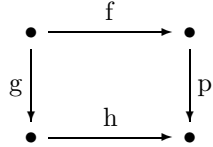
\begin{diagram}[height=2em]
\bullet &   \rTo^{\mathrm f}   &                    \bullet            \\
\dTo^{\mathrm g} & &  \dTo_ {\mathrm p} \\
\bullet &   \rTo^{\mathrm h}   &                    \bullet \\
\end{diagram}
we have $2$-morphism:
\begin{equation}\label{76}
\mathrm e_{\mathrm Q}\colon \mathrm p_!\mathrm f_*\longrightarrow\mathrm h_*\mathrm g_!
\end{equation}
satisfying following axioms:

(i) Compatibility with vertical and horizontal compositions. For vertical it means that following solid is commutative:
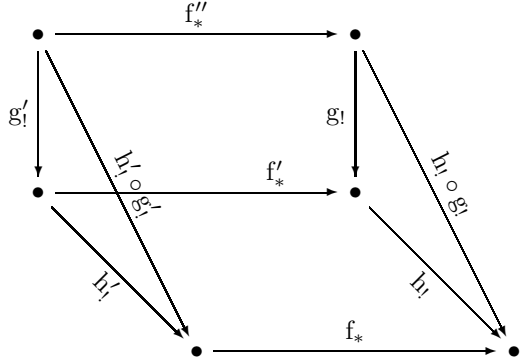
\begin{diagram}
\bullet &         & \rTo^{\mathrm f''_*} &                   & \bullet            & & \\
\dTo^{\mathrm g'_!} &\rdTo(2,4)^{\mathrm h_!'\circ\mathrm g_!'} &  &  &  \dTo^{\mathrm g_!} &\rdTo(2,4)^{\mathrm h_!\circ\mathrm g_!} & \\
\bullet &  &  & \rTo^{\mathrm f'_*} & \bullet & &  \\
& \rdTo_{\mathrm h'_!} & & & & \rdTo_{\mathrm h_!} & \\
        &         & \bullet              &                   & \rTo^{\mathrm f_*} & & \bullet \\
\end{diagram}

(ii) If both vertical or both horizontal morphisms are identities then
the exchange morphism is an isomorphism.

\end{Def}

\begin{Def} Pair $(\mathrm H^{\star},\mathrm H^!),$ where $\mathrm H^{\star}$ and $\mathrm H^!$ are $2$-functors from $\mathrm D^{\circ}$ to $\mathrm C,$ is called \underline{upper $\mathrm e$-functor}, if $(\mathrm H^{\star},\mathrm H^!),$ is a lower $\mathrm e$-functor with values in $\mathrm C^{1-\circ}.$ Square $\mathrm Q$ gives $2$-morphism:
\begin{equation}\label{76}
\mathrm e_{\mathrm Q}\colon \mathrm f^*\mathrm p^!\longrightarrow\mathrm g^!\mathrm h^*
\end{equation}
\end{Def}

\begin{Def} Pair $(\mathrm H_{!},\mathrm H^{\star}),$ where $\mathrm H_{!}$ is functor from $\mathrm D$ to $\mathrm C$ and $\mathrm H^{\star}$ is $2$-functors from $\mathrm D^{\circ}$ to $\mathrm C,$ such that $\mathrm H_{!}(\mathrm X)=\mathrm H^{\star}(\mathrm x):=\mathrm H(\mathrm x)$ for every $\mathrm x\in \mathrm D,$ is called \underline{$\mathrm e^{\star}$-contradirectional functors} if for every $\mathrm Q$ we have a morphism:
\begin{equation}\label{76}
\mathrm e_{\mathrm Q}\colon \mathrm p_!\mathrm h^{*}\longrightarrow\mathrm g_!\mathrm f^*
\end{equation}
satisfying same axioms as previous functors.
\end{Def}

\begin{Def} Pair $(\mathrm H_{\star},\mathrm H^{!}),$ where $\mathrm H_{\star}$ is functor from $\mathrm D$ to $\mathrm C$ and $\mathrm H^{!}$ is $2$-functors from $\mathrm D^{\circ}$ to $\mathrm C,$ such that $\mathrm H_{\star}(\mathrm x)=\mathrm H^{!}(\mathrm x):=\mathrm H(\mathrm x)$ for every $\mathrm x\in \mathrm D,$ is called \underline{$\mathrm e^{!}$-contradirectional functors} if for every $\mathrm Q$ we have a morphism:
\begin{equation}\label{76}
\mathrm e_{\mathrm Q}\colon \mathrm g_*\mathrm f^{!}\longrightarrow\mathrm h^!\mathrm p_*
\end{equation}
satisfying same axioms as $\mathrm e^{\star}$-contradirectional functors.
\end{Def}
Now we can give main definition from \cite{Voe}:
\begin{Def}\label{cr.d.5} A \underline{cross-functor} from $\mathrm D$ to $\mathrm C,$ relative to class $\mathrm B,$ is the following collection of data:\par\medskip

(i) an upper $\mathrm e$-functor $(\mathrm H^{\star},\mathrm H^!),$ and a lower $\mathrm e$-functor $(\mathrm H_{\star},\mathrm H_!),$ from $\mathrm D$ to $\mathrm C,$ which are equal on objects, $\mathrm H_{\star}(\mathrm x)=\mathrm H^{\star}(\mathrm x)=\mathrm H_!(\mathrm x)=\mathrm H^{!}(\mathrm x),$ for very $\mathrm x\in \mathrm D.$ For a morphism $\mathrm f\colon \mathrm x\longrightarrow \mathrm y$ in $\mathrm D$ we will write $\mathrm f_*$ for $\mathrm H_{\star}(\mathrm f),$ $\mathrm f^*$ for $\mathrm H^{\star}(\mathrm f),$ $\mathrm f_!$ for $\mathrm H_{!}(\mathrm f),$ and $\mathrm f^!$ for $\mathrm H^!(\mathrm f),$\par\bigskip

(ii) For every $\mathrm f\colon \mathrm x\rightarrow \mathrm y$ we have standard adjunction between $\mathrm f_*$ and $\mathrm f^*$ and $\mathrm f_!$ and $\mathrm f^!,$ i.e. $\mathrm f_*$ is right adjoint to $\mathrm f^*$ and $\mathrm f^!$ is right adjoint to $\mathrm f_!.$\par\bigskip

Following axioms should hold\par\medskip
(a) Compatibility of adjunction with composition \par\bigskip

(b) Given a square in $\mathrm B$

\begin{diagram}[height=2em]
\bullet &   \rTo^{\mathrm f}   &                    \bullet            \\
\dTo^{\mathrm g} & &  \dTo_ {\mathrm p} \\
\bullet &   \rTo^{\mathrm h}   &                    \bullet \\
\end{diagram}
The corresponding exchange morphisms:
\begin{equation}\label{65}
 \mathrm p_!\circ\mathrm f_*\longrightarrow \mathrm h_*\circ\mathrm g_!,\quad  \mathrm f^*\circ\mathrm p^!\longrightarrow \mathrm g^!\circ\mathrm h^*
\end{equation}
transmute by into morphism:
\begin{equation}\label{56}
\mathrm p^*\circ\mathrm h_!\longrightarrow \mathrm g_!\circ\mathrm f^*,\quad  \mathrm g_!\circ\mathrm f^*\longrightarrow \mathrm h^*\circ\mathrm p_!
\end{equation}
\end{Def}

\subsection{Monoidal categories}

Here we are going to recall some facts about morphism of monoidal categories following \cite{May}.
Let $\mathrm M$ and $\mathrm M'$ be a pair of closed symmetric monoidal categories and let $\mathrm f^*\colon \mathrm M'\longrightarrow\mathrm M$ be a strong monoidal functor with right adjoint $\mathrm f_*\colon \mathrm M\longrightarrow \mathrm M'.$ We denote counit and unit of adjunction by $\epsilon\colon \mathrm f^*\circ\mathrm f_*\EuScript F \longrightarrow \EuScript F$ and $\eta\colon \EuScript F \longrightarrow\mathrm f_*\circ\mathrm f^*\EuScript F.$ We have following evident
\begin{Prop}\label{p4.2.1}
Adjuncts of natural morphisms
\begin{equation}\label{e4.2.1}
\mathrm f^*\circ(\mathrm f_*\EuScript F\otimes \mathrm f_*\EuScript G) \cong \mathrm f^*\circ\mathrm f_*\EuScript F\otimes\mathrm f^*\circ\mathrm f_*\EuScript G \xrightarrow{\epsilon\otimes\epsilon}\EuScript F\otimes \EuScript G\qquad \EuScript U_{\mathrm M'}\longrightarrow \mathrm f_*\EuScript U_{\mathrm M'}
\end{equation}
make $\mathrm f_*$ into lax monoidal functor.
\end{Prop}

We have natural morphism:
\begin{equation}\label{e4.2.2}
\mathrm f^*\circ\underline{\mathrm{Hom}}(\EuScript F,\EuScript G)\otimes \mathrm f^*\EuScript K\cong \mathrm f^*\circ(\underline{\mathrm{Hom}}(\EuScript F,\EuScript G)\otimes\EuScript K)\xrightarrow{\mathrm f^*\mathrm{ev}} \mathrm f^*\EuScript G
\end{equation}
with corresponding adjoint morphism
\begin{equation}\label{e4.2.3}
\alpha\colon\mathrm f^*\circ\underline{\mathrm{Hom}}(\EuScript F,\EuScript G)\longrightarrow \underline{\mathrm{Hom}}(\mathrm f^*\EuScript F,\mathrm f^*\EuScript G)
\end{equation}
We also have obvious
\begin{Prop}\label{p4.2.2} Adjunct of morphism
\begin{equation}\label{e4.2.4}
\mathrm f^*\circ\underline{\mathrm{Hom}}(\EuScript F,\mathrm f_*\EuScript G)\xrightarrow{\alpha}\underline{\mathrm{Hom}}(\mathrm f^*\EuScript F,\mathrm f^*\circ\mathrm f_*\EuScript G)\xrightarrow{\mathrm{Hom}(\mathrm{id},\epsilon)}\underline{ \mathrm{Hom}}(\mathrm f^*\EuScript F,\EuScript G)
\end{equation}
is natural isomorphism:
\begin{equation}\label{e4.2.5}
\underline{ \mathrm{Hom}}(\EuScript F,\mathrm f_*\EuScript G)\cong \mathrm f_*\underline{ \mathrm{Hom}}(\mathrm f^*\EuScript F,\EuScript G)
\end{equation}
\end{Prop}

We can also define another morphism:
\begin{equation}\label{e4.2.6}
\beta\colon \mathrm f_*\circ\underline{\mathrm{Hom}}(\EuScript F,\EuScript G)\xrightarrow{\mathrm f_*\mathrm{Hom}(\epsilon,\mathrm{id})}\mathrm f_*\circ\underline{\mathrm{Hom}}(\mathrm f^*\circ\mathrm f_*\EuScript F,\EuScript G)\cong\underline{\mathrm{Hom}}(\mathrm f_*\EuScript F,\mathrm f_*\EuScript G)
\end{equation}
\begin{Def}\label{d4.2.1} Let $\mathrm M$ and $\mathrm M'$ pair of closed monoidal categories with adjoint functors $(\mathrm f_*,\mathrm f^*).$ If natural morphism
\begin{equation}\label{e4.2.7}
\pi\colon \EuScript F\otimes \mathrm f_*\EuScript G \longrightarrow\mathrm f_*\circ\mathrm f^*\EuScript F\otimes \mathrm f_*\EuScript G\longrightarrow \mathrm f_*\circ(\mathrm f^*\EuScript F\otimes \EuScript G)
\end{equation}
is an isomorphism we say that \underline{projection formula} holds.
\end{Def}
Assume that in addition to adjoint functors $(\mathrm f_*,\mathrm f^*)$ between closed monoidal categories $\mathrm M$ and $\mathrm M'$ we have another pair of adjoint functors $\mathrm f_!\colon \mathrm M \longleftrightarrow\mathrm M'\colon \mathrm f^!,$ where $\mathrm f^!$ is right adjoint to $\mathrm f_!.$ We denote counit and unit of adjunction by $\sigma\colon \mathrm f_!\circ\mathrm f^!\EuScript F \longrightarrow \EuScript F$ and $\zeta\colon \EuScript F\longrightarrow\mathrm f^!\circ\mathrm f_!\EuScript F.$
We have evident
\begin{Prop}\label{p4.2.3} Suppose that we have isomorphism $\mathrm f^*\cong \mathrm f^!,$ then adjunct of natural morphisms
\begin{equation}\label{e4.2.8}
\EuScript F\otimes \EuScript G\xrightarrow{\zeta\otimes\zeta} \mathrm f^*\circ\mathrm f_!\EuScript F\otimes \mathrm f^*\circ\mathrm f_!\EuScript G\cong \mathrm f^*\circ(\mathrm f_!\EuScript F\otimes \mathrm f_!\EuScript G)   \qquad \mathrm f_!\EuScript U_{\mathrm M}\longrightarrow\EuScript U_{\mathrm M'}
\end{equation}
make $\mathrm f_!$ into op-lax monoidal functor.

\end{Prop}

Consider following natural morphisms:
\begin{align}\label{a4.2.1(1)}
&\gamma\colon \mathrm f_*\circ\underline{\mathrm{Hom}}(\EuScript F,\mathrm f^!\EuScript G)\longrightarrow\underline{\mathrm{Hom}}(\mathrm f_*\EuScript F,\EuScript G)\\\label{a4.2.1(2)}
&\delta\colon \underline{\mathrm{Hom}}(\mathrm f^*\EuScript F,\mathrm f^!\EuScript G)\longrightarrow\mathrm f^!\circ\underline{\mathrm{Hom}}(\EuScript F,\EuScript G)\\\label{a4.2.1(3)}
&\hat{\pi}\colon \EuScript F\otimes \mathrm f_!\EuScript G \longrightarrow \mathrm f_!\circ(\mathrm f^*\EuScript F\otimes \EuScript G)
\end{align}
\begin{Prop}\label{p4.2.4}\label{p2.4} Suppose that we are given one of the maps  $\hat{\pi}, \gamma$ or $\delta$ then it determines all others. Moreover if one of these maps is natural isomorphism then so the others too.
\end{Prop}
\begin{proof} See \cite[Proposition~2.4]{May}.

\end{proof}
We can also consider maps \eqref{a4.2.1(1)}, \eqref{a4.2.1(2)} and \eqref{a4.2.1(3)} with reverse directions:
\begin{align}\label{a4.2.2(1)}
&\bar{\gamma}\colon \underline{\mathrm{Hom}}(\mathrm f_*\EuScript F,\EuScript G)\longrightarrow\mathrm f_*\circ\underline{\mathrm{Hom}}(\EuScript F,\mathrm f^!\EuScript G)\\\label{a4.2.2(2)}
&\bar{\delta}\colon \mathrm f^!\circ\underline{\mathrm{Hom}}(\EuScript F,\EuScript G)\longrightarrow\underline{\mathrm{Hom}}(\mathrm f^*\EuScript F,\mathrm f^!\EuScript G)\\\label{a4.2.2(3)}
&\bar{\pi}\colon \mathrm f_!\circ(\mathrm f^*\EuScript F\otimes \EuScript G)\longrightarrow\EuScript F\otimes \mathrm f_!\EuScript G
\end{align}
\begin{Prop}\label{p4.2.5} Suppose that we are given one of the maps  $\bar{\pi}, \bar{\gamma}$ or $\bar{\delta}$ then it determines all others. Moreover if one of these maps is natural isomorphism then so the others too.
\end{Prop}
\begin{proof} See \cite[Proposition~2.8]{May}.

\end{proof}
\begin{Prop}\label{p4.2.6} Suppose that we have natural isomorphism $\mathrm f_!\cong \mathrm f_*.$ Then we can take $\hat{\pi}$ to be $\pi$ and $\gamma$ is a natural morphism
\begin{equation}\label{e4.2.9}
\mathrm f_*\circ\underline{ \mathrm {Hom}}(\EuScript F,\mathrm f^!\EuScript G)\xrightarrow{\beta}\underline{ \mathrm {Hom}}(\mathrm f_*\EuScript F,\mathrm f_*\circ\mathrm f^!\EuScript G)\xrightarrow{\mathrm{Hom}(\mathrm{id},\sigma)}\underline{ \mathrm {Hom}}(\mathrm f_*\EuScript F,\EuScript G)
\end{equation}
Morphism $\delta$ is adjunct of following natural morphism:
\begin{equation}\label{e4.2.10}
\mathrm f_*\circ\underline{ \mathrm {Hom}}(\mathrm f^*\EuScript F,\mathrm f^!\EuScript G)\cong\underline{ \mathrm {Hom}}(\EuScript F,\mathrm f_*\mathrm f^!\EuScript G)\xrightarrow{\mathrm{Hom}(\mathrm{id},\sigma)}\underline{ \mathrm {Hom}}(\EuScript F,\EuScript G)
\end{equation}
\end{Prop}
\begin{proof} See \cite[Proposition~2.9]{May}.

\end{proof}
\begin{Def}\label{d4.2.2} \underline{Grothendieck context} is a pair of closed symmetric monoidal categories $\mathrm M$ and $\mathrm N$ and a triple of functors $(\mathrm f^*,\mathrm f_*,\mathrm f^!),$ $\mathrm f_*\colon \mathrm M \longrightarrow \mathrm N\colon \mathrm f^*,\mathrm f^!,$ where functor $\mathrm f^*$ is left adjoint to $\mathrm f_*$ and $\mathrm f^!$ is right adjoint to $\mathrm f_*.$ We also assume that $\mathrm f^*$ is a strong monoidal functor and projection formula holds.

\end{Def}

\begin{Def}\label{d4.2.4} For an object $\EuScript K\in \mathrm M$ we define object $\mathbf D_{\EuScript K}\EuScript F:=\underline{\mathrm {Hom}}(\EuScript E,\EuScript K),$ the \underline{$\EuScript K$-twisted dual of $\EuScript F.$} We call $\EuScript F$ a \underline{$\EuScript K$-reflexive} if we have isomorphism $\EuScript F \cong \mathbf D_{\EuScript K}\mathbf D_{\EuScript K}\EuScript F.$

\end{Def}

\begin{remark}\label{r4.2.1} Note then if $\EuScript K$ is dualizable then $\mathbf D_{\EuScript K}\EuScript F=\mathbf D_{\EuScript U}\EuScript F\otimes \EuScript K.$

\end{remark}
Let $\mathrm f\colon \mathrm M\longrightarrow \mathrm N$ be a Grothendieck context, if we set $\mathrm f^!\EuScript G=\EuScript K$ for an object $\EuScript G\in \mathrm N,$ then from isomorphisms $\delta$ and $\gamma$ we get:
\begin{equation}\label{e4.2.13}
\mathrm f_*\mathbf D_{\EuScript K}\EuScript F\cong \mathbf D_{\EuScript G}\mathrm f_!\EuScript F\quad \mathbf D_{\EuScript K}\mathrm f^*\EuScript Y\cong \mathrm f^!\mathbf D_{\EuScript G}\EuScript Y
\end{equation}

\begin{Def}\label{d4.2.5} A \underline{dualizing object} for a full subcategory $\mathrm M_0$ of $\mathrm M$ is an object $\EuScript K\in \mathrm M_0$ such that if $\EuScript F\in \mathrm M_0,$ then $\mathbf D_{\EuScript K}\EuScript F$ is in $\mathrm M_0$ and $\EuScript F$ is $\EuScript K$-reflexive. Thus $\mathbf D_{\EuScript K}$ defines an auto-duality of the category $\mathrm M_0:$
\begin{equation}\label{e4.2.14}
 \mathbf D_{\EuScript K}\colon \mathrm M_0 \overset{\sim}{\longrightarrow} \mathrm M_0
\end{equation}
\end{Def}

\subsection{Grothendieck context} We assume that in category $\mathrm D$ colimits over diagram $\mathrm x\leftarrow \mathrm z \rightarrow \mathrm y$ exits for all $\mathrm x,$ $\mathrm y$ and $\mathrm z$ in $\mathrm D.$ For the class $\mathrm B$ of commutative squares in $\mathrm D$ we take class of pushouts in category $\mathrm D.$ We assume further that category $\mathrm D^{\circ}$ satisfies following generalized filtered condition:
\begin{Assum} For every $\mathrm x$ and $\mathrm y$ in $\mathrm D^{\circ}$ we have object $\mathrm z\in \mathrm D^{\circ}$ and arrows $\mathrm x\rightarrow \mathrm z\leftarrow \mathrm y,$ for every pair of parallel morphisms $\mathrm x\rightrightarrows \mathrm y$ we have an object $\mathrm w\in \mathrm D^{\circ},$ such that
$\mathrm x\rightarrow \mathrm w\leftarrow \mathrm y,$ and $\sigma\circ\mathrm j\circ\mathrm u=\mathrm j\circ\mathrm u,$ where $\sigma$ is an automorphism of $\mathrm w.$
\end{Assum}
Let $(\mathrm E,\mathrm A)$ be ringed $\mathrm D$-topos we give following:

\begin{Def}\label{d4.3.1} A \underline{Grothendieck cross functor} for the $\mathrm D$-topos $(\mathrm E,\mathrm A)$ is a cross functor $(\mathrm H_{\star},\mathrm H^{\star},\mathrm H_!,\mathrm H^!),$ such that bifibration which corresponds to $2$-functors $(\mathrm H_{\star},\mathrm H^{\star})$ is isomorphic to bifibration $\underline{\mathrm {Mod}}(\mathrm E,\mathrm A)\longrightarrow \mathrm D$ and moreover we have isomorphism of $!$-lower $\mathrm e$-functor and $\star$-lower $\mathrm e$-functor:
\begin{equation}\label{65}
\mathrm H_!\overset{\sim}{\longrightarrow} \mathrm H_{\star}
\end{equation}

\end{Def}

\begin{remark}\label{r4.4.1}

We specified to the setting of Grothendieck cross functors for several reasons. In order to make our definitions reasonable in general case we need to work in the setting of categories with enhancement. However, this is not the most important issue. In general, contravariant Verdier duality can not be expressed in terms of Verdier duality functors, which act between fibers, as it remarked in \cite{Gai1}. Also definition of dualizing object becomes much more involved \cite{May}.

\end{remark}
We have triangulated categories of sections $\mathrm D(\underline{\Gamma}({\mathrm H^{\star}}))$ and $\mathrm D(\underline{\Gamma}({\mathrm H^!})),$ with full triangulated subcategories $\mathrm D_{\mathrm {cocart}}(\underline{\Gamma}({\mathrm H^{\star}}))$ and $\mathrm D_{\mathrm {cart}}(\underline{\Gamma}({\mathrm H^{!}})),$ corresponding inclusion functor:
\begin{equation}
 \Xi^{\star} \colon\mathrm D_{\mathrm {cocart}}(\underline{\Gamma}({\mathrm H^{\star}}))\longrightarrow \mathrm D(\underline{\Gamma}({\mathrm H^{\star}})),
\end{equation}
and corresponding inclusion functor in the $!$-case:
\begin{equation}
 \Xi^{!} \colon \mathrm D_{\mathrm {cart}}(\underline{\Gamma}({\mathrm H^{!}}))\longrightarrow \mathrm D(\underline{\Gamma}({\mathrm H^{!}}))
\end{equation}
\begin{Def}\label{d4.3.2} Let $(\mathrm H^{\star!},\mathrm H_{!\star})$ be a pair of $2$-functor from $2$-category $\mathrm {cospan}(\mathrm D)$ to $\mathrm {Cat},$ such that we have isomorphisms:
\begin{equation}\label{98}
\mathrm i^*(\mathrm H^{\star!})\overset{\sim}{\longrightarrow} \mathrm H_!,\quad  \mathrm i^*(\mathrm H_{!\star})\overset{\sim}{\longrightarrow} \mathrm H^!,\quad \mathrm j^*(\mathrm H^{\star!})\overset{\sim}{\longrightarrow} \mathrm H^{\star},\quad \mathrm j^*(\mathrm H_{!\star})\overset{\sim}{\longrightarrow} \mathrm H_*.
\end{equation}
Pseudo functor $\mathrm H^{\star!}$ is called \underline{Mackey $\star!$-functor} associated with Grothendieck cross functor and pseudo functor $\mathrm H_{\star!}$ is called \underline{Mackey $!\star$-functor}.
\end{Def}

\begin{remark} One can develop our theory for cross functors with values in more general categories, see Subsection \ref{decat}.

\end{remark}

\begin{Prop}\label{p4.3.1} For every Grothendieck cross functor $(\mathrm H_{\star},\mathrm H^{\star},\mathrm H_!,\mathrm H^!),$ associated with $\mathrm D$-topos $(\mathrm E,\mathrm A)$ there exists unique Mackey pseudo-functors $(\mathrm H^{\star!},\mathrm H_{!\star}).$

\end{Prop}

\begin{proof} See proofs of Proposition \ref{p5.4.1}.

\end{proof}

\begin{Def}\label{d4.3.3} Let $(\mathrm H_{\star},\mathrm H^{\star},\mathrm H_!,\mathrm H^!)$ be a Grothendieck cross functor associated with ringed $\mathrm D$-topos $(\mathrm E,\mathrm A)$. We define pair of adjoint triangulated functors by the rule $\mathbb V_{\star\mapsto!}:=\mathrm i^*\circ\underline{\mathrm R}(\mathrm j_*)$ and $\mathbb V_{!\mapsto\star}:=\mathrm j^*\circ \underline{{\mathrm {L}}}(\mathrm i_!):$
\begin{equation}\label{e4.3.1}
\mathbb V_{\star\mapsto!}\colon \mathrm D(\underline{\Gamma}(\mathrm H^{\star}))\longleftrightarrow\mathrm D(\underline{\Gamma}(\mathrm H^{!}))\colon \mathbb V_{!\mapsto\star}.
\end{equation}
These functors are called \underline{covariant Verdier duality}.
\end{Def}
\begin{remark} Verdier duality functors are well defined as derived functors. Indeed functor $\mathbb V_{!\mapsto\star}$ is well defined: filtered colimits in modules over topos are exact. Functor $\mathbb V_{\star\mapsto !}$ is well defined, since for every morphism $\mathrm m\colon \mathrm i\rightarrow \mathrm j$ functor $\mathrm f_*:=\mathrm m^*$ takes injective objects to injective and thus $\mathrm f^!$ takes injective to injective.

\end{remark}
Sometimes we denote morphisms in $\mathrm {cospan}(\mathrm D)$ by $\mathrm h:=(\mathrm h_{\mathrm l},\mathrm k,\mathrm h_{\mathrm r})\colon \mathrm i\rightarrow \mathrm j.$ For morphism $\mathrm h$ corresponding functor between fibers will be denoted by $\mathrm h^!_{*}$ in the case of $!\star$-Mackey functor and by $\mathrm h^{*}_{!}$ in the case of $\star!$-Mackey functor.
\begin{Prop}\label{p4.3.2}  Let $(\mathrm H_{\star},\mathrm H^{\star},\mathrm H_!,\mathrm H^!)$ be a Grothendieck cross functor, such that for every $\mathrm m\colon \mathrm i\rightarrow \mathrm j$ in $\mathrm D,$ morphism $(\mathrm m_*,\mathrm m^*)$ is closed, then covariant Verdier duality functor $\mathbb V_{\star\mapsto!}$ (resp. $\mathbb V_{!\mapsto\star}$) takes values in subcategories with cartesian (resp. cocartesian) cohomology:
\begin{equation}\label{e4.3.2}
\mathbb V_{\star\mapsto!}\colon \mathrm D_{\mathrm {cocart}}(\underline{\Gamma}(\mathrm {H}^{\star}))\longleftrightarrow\mathrm D_{\mathrm {cart}}(\underline{\Gamma}(\mathrm{H}^{!}))\colon \mathbb V_{!\mapsto\star}.
\end{equation}
\end{Prop}
\begin{proof} Let us consider case of functor $\mathbb V_{\star\mapsto!},$ case of functor $\mathbb V_{!\mapsto\star}$ can be treated analogically. Let $\mathrm m^{\circ}\colon \mathrm i\rightarrow \mathrm j$ be a morphism in $\mathrm D^{\circ}$ and $\EuScript K\in\mathrm D_{\mathrm {cocart}}(\underline{\Gamma}(\mathrm {E}),\mathrm A),$ then by definition we have the morphism:
\begin{equation}\label{e4.3.3}
\hlim_{{\mathrm p\colon\mathrm i\rightarrow\mathrm w}}\underline{\mathrm R}(\mathrm p^!_*)\EuScript K_{\mathrm w}\longrightarrow \underline{\mathrm R}^*(\mathrm f^!)\circ\hlim_{{\mathrm l\colon\mathrm j\rightarrow\mathrm w'}}\underline{\mathrm R}(\mathrm l^!_*)\EuScript K_{\mathrm w'}
\end{equation}
Which is induced by morphism of corresponding abelian functors. Thus it is enough to show that diagram, underlying first limit is contained in diagram, underlying the second one. Let $\mathrm h=(\mathrm h_{\mathrm l},\mathrm v,\mathrm h_{\mathrm r})\colon\mathrm i\rightarrow\mathrm x$ be any morphism in $\mathrm{cospan}(\mathrm D),$ then we also have morphism $\mathrm h'=(\mathrm h_{\mathrm l}',\mathrm v,\mathrm h_{\mathrm r})\colon \mathrm j\rightarrow\mathrm x,$ with property $\mathrm h_{\mathrm r}'=\mathrm m^{\circ}\circ\mathrm h_{\mathrm l}.$ Then we have adjunction morphism $\mathrm h_{\mathrm l*}\circ\mathrm h_{\mathrm r}\longrightarrow\mathrm m^{\circ*}\circ\mathrm m^{\circ}_*\circ\mathrm h_{\mathrm l*}\circ\mathrm h_{\mathrm r}^*,$ which is isomorphism, since $\mathrm m^{\circ}_*=\mathrm f_*$ is fully faithful.
\end{proof}

\begin{Def}\label{d4.3.4} Let $\mathrm H^{\star\star}$ be a pseudo functor from $\mathrm {cospan}(\mathrm D)$ to $\mathrm {Cat}$ such that:
 \begin{equation}\label{98}
 \mathrm j^*(\mathrm E^{\star\star})\overset{\sim}{\longrightarrow} \mathrm H^{\star\circ},\quad \mathrm i^*(\mathrm E^{\star\star})\overset{\sim}{\longrightarrow} \mathrm H^{\star}
\end{equation}
We call such pseudo functor \underline{Mackey $\star\star$-functor} associated with Grothendieck cross functor $(\mathrm H_{\star},\mathrm H^{\star},\mathrm H_!,\mathrm H^!).$
\end{Def}
Assume, that for every morphism $\mathrm m\colon \mathrm i\rightarrow \mathrm j$ in category the $\mathrm D,$ functor $\mathrm m_*=\mathrm f^*\colon {\mathrm {Mod}}(\mathrm E_{\mathrm i},\mathrm A_{\mathrm i })\longrightarrow {\mathrm {Mod}}(\mathrm E_{\mathrm j},\mathrm A_{\mathrm j })$ commutes with limits and flat. We define functor
\begin{equation}\label{e4.3.4}
\Xi_{\star}\colon \mathrm D(\underline{\Gamma}(\mathrm {H}^{\star}))\longrightarrow \mathrm D(\underline{\Gamma}(\mathrm {H}^{\star})),
\end{equation}
by the rule $\Xi_{\star}=\mathrm j^*\circ\underline{\mathrm {R}}(\mathrm j_*)$
\begin{Def}\label{d4.3.5} Let $\mathrm H^{!!}$ be a pseudo-functor from category $\mathrm {cospan}(\mathrm D)$ such that:
\begin{equation}\label{nn}
\mathrm i^*(\mathrm H^{!!})\overset{\sim}{\longrightarrow} \mathrm H^{!\circ},\quad \mathrm j(\mathrm H^{!!})\overset{\sim}{\longrightarrow} \mathrm E^{!}.
\end{equation}
We call such pseudo functor \underline{Mackey $!!$-functor} associated with Grothendieck cross functor $(\mathrm H_{\star},\mathrm H^{\star},\mathrm H_!,\mathrm H^!).$
\end{Def}
Assume, that for every morphism $\mathrm m\colon \mathrm i\rightarrow \mathrm j$ in the category $\mathrm D^{\circ},$ functor $\mathrm m^*=\mathrm f^!\colon {\mathrm {Mod}}(\mathrm E_{\mathrm i},\mathrm A_{\mathrm i })\longrightarrow {\mathrm {Mod}}(\mathrm E_{\mathrm j},\mathrm A_{\mathrm j })$ commutes with colimits. We define functor
\begin{equation}\label{e4.3.5}
\Xi_!\colon \mathrm D(\underline{\Gamma}(\mathrm {H}^{!}))\longrightarrow \mathrm D(\underline{\Gamma}(\mathrm {H}^{!})),
\end{equation}
by the rule $\Xi_!=\mathrm i^*\circ\underline{\mathrm {L}}(\mathrm i_!).$
\begin{remark}\label{r4.3.1} We need to mention what we understand by functor $\Xi_!.$ Usually functor $\mathrm f^!$ is only left exact and since then functor $\mathrm g_*\circ\mathrm f^!$ is also left exact and fibrations $\mathrm H^{!!}$ and $\mathrm H^{!\circ}$ are not abelian. Using some enhancement (For example see \cite[Subsection~7.4.]{BD2}, \cite[Subsection~1.6.]{DG}), \cite{Lu1}) we can define derived version of this fibrations and $\mathrm i^*\circ\underline{\mathrm {L}}(\mathrm i_!)$ are just composition of homotopy Kan extension and restriction functor. In the same way one can remove flatness restriction in the definition of functor $\Xi_{\star}.$
 In some situations it is also useful to consider functors from opposite fibrations:
\begin{equation}\label{}
\Psi_{\star}\colon \mathrm D(\underline{\Gamma}(\mathrm {H}^{\star\circ}))\longrightarrow \mathrm D_{\mathrm {cocart}}(\underline{\Gamma}(\mathrm {H}^{\star}))
\end{equation}
and
\begin{equation}\label{}
\Psi_!\colon \mathrm D(\underline{\Gamma}(\mathrm {H}^{!\circ}))\longrightarrow \mathrm D_{\mathrm {cart}}(\underline{\Gamma}(\mathrm {H}^{!})).
\end{equation}
\end{remark}
\begin{Prop}\label{p4.3.3} For every Grothendieck cross functor $(\mathrm H_{\star},\mathrm H^{\star},\mathrm H_!,\mathrm H^!)$ over $\mathrm D$-topos $\mathrm (\mathrm E,\mathrm A)$ there exist a unique pair of Mackey pseudo-functors $(\mathrm H^{\star\star},\mathrm H^{!!}).$
\end{Prop}
\begin{proof} See proof of Proposition \ref{p4.3.1}.

\end{proof}
\begin{Prop}\label{p4.3.4} Let $(\mathrm H_{\star},\mathrm H^{\star},\mathrm H_!,\mathrm H^!)$ be a Grothendieck cross functor over ringed $\mathrm D$-topos $(\mathrm E,\mathrm A).$ Functor $\Xi_{\star}$ takes values in subcategory $\mathrm D_{\mathrm{cocart}}(\underline{\Gamma}(\mathrm {H}^{\star})).$ Moreover if for every $\mathrm i\in \mathrm D$ limits over coslice category $\mathrm i/\mathrm {cospan}(\mathrm D)$ are exact, then triangulated subcategory:
\begin{equation}\label{}
\Xi^{\star}\colon\mathrm D_{\mathrm{cocart}}(\underline{\Gamma}(\mathrm {H}^{\star}))\longrightarrow\mathrm D(\underline{\Gamma}(\mathrm {H}^{\star}))
\end{equation}
is a right admissible with a right adjoint functor $\Xi_{\star}.$ Triangulated subcategory:
\begin{equation}\label{}
\Xi^!\colon\mathrm D_{\mathrm{cart}}(\underline{\Gamma}(\mathrm {H}^{!}))\longrightarrow\mathrm D(\underline{\Gamma}(\mathrm {H}^{!}))
\end{equation}
is left admissible with left adjoint functor $\Xi_!.$
\end{Prop}
\begin{proof} Analogical to Proposition \ref{p4.3.2}. Adjunction follows from assumption on category $\mathrm D^{\circ}.$
\end{proof}
\begin{remark}\label{r4.3.2} Note that if ring object $\mathrm A$ is cocartesian we have isomorphism of functors $\underline{\mathrm R}^0(\Xi_{\star})\cong \Xi_*.$
\end{remark}
\begin{remark} Denote by $\EuScript S\mathrm p$ category of spectra, by definition it is stabilization of category of pointed topological spaces $\mathrm {Top}_*.$ That is category $\EuScript S\mathrm p$ is the category of sections of bifibration over $\mathbb N,$ whose fiber is given by category $\mathrm {Top}_*$ and functors between fibers are loop space object $\Omega$ and suspension functor $\Sigma.$ Category $\EuScript S\mathrm p$ has full subcategory $\Omega\EuScript S\mathrm p,$ of $\Omega$-spectra that is subcategory of $\Omega$-cartesian sections. Functors $\Xi_{\star}$ and $\Xi_!$ are analogous to functor which takes category of spectra $\EuScript S\mathrm p$ to to the category of $\Omega$-spectra $\Omega\EuScript S\mathrm p$ \cite{Bousf}:
\begin{equation}\label{}
\mathrm Q\colon \EuScript S\mathrm p\longrightarrow\Omega\EuScript S\mathrm p,\quad \mathrm Q(\mathrm X)^{\mathrm n}:=\hclim_{\mathrm i\in \mathbb N} \Omega^{\mathrm i}\mathrm X^{\mathrm n+\mathrm i}.
\end{equation}

\end{remark}
\begin{Def}\label{d4.3.6} Let $(\mathrm H_{\star},\mathrm H^{\star},\mathrm H_!,\mathrm H^!)$ be a Grothendieck cross functor associated with ringed $\mathrm D$-topos $(\mathrm E,\mathrm A).$ We define \underline{covariant Verdier duality functors}, which respects categories
$\mathrm D_{\mathrm {cocart}}(\underline{\Gamma}(\mathrm {H}^{\star}))$ and $\mathrm D_{\mathrm {cart}}(\underline{\Gamma}(\mathrm{H}^{!})):$
\begin{equation}\label{e4.3.6}
\mathbb V_{\star\mapsto!}^{\mathrm{cart}}\colon \mathrm D_{\mathrm {cocart}}(\underline{\Gamma}(\mathrm {H}^{\star}))\longleftrightarrow\mathrm D_{\mathrm {cart}}(\underline{\Gamma}(\mathrm{H}^{!}))\colon \mathbb V_{!\mapsto\star}^{\mathrm{cocart}}
\end{equation}
by the rules:
\begin{equation}\label{}
\mathbb V_{\star\mapsto!}^{\mathrm{cart}}:=\Xi_{!}\circ \mathbb V_{\star\mapsto!},\qquad\qquad\quad \mathbb V_{!\mapsto\star}^{\mathrm{cocart}}:=\Xi_{\star}\circ \mathbb V_{!\mapsto\star}.
\end{equation}
\end{Def}

Grothendieck cross functor $(\mathrm H_{\star},\mathrm H^{\star},\mathrm H_!,\mathrm H^!)$ is called constructible if we have stratification $\EuScript S_{\mathrm i}$ on $\mathrm E_{\mathrm i}$ for every $\mathrm i,$ such that $(\mathrm E,\mathrm A)$ is a stratified $\mathrm D$ topoi and functors $\Xi_{\star}$ and $\Xi_!$ and covariant Verder duality functors respect constructible subcategories:
\begin{equation}\label{e4.3.7}
\mathbb V_{\star\mapsto!}^{\mathrm{cart}}\colon \mathrm D_{\mathrm{cocart},\mathrm {c}}(\underline{\Gamma}(\mathrm {H}^{\star}))\longleftrightarrow\mathrm D_{\mathrm{cart},\mathrm {c}}(\underline{\Gamma}(\mathrm{H}^{!}))\colon \mathbb V_{!\mapsto\star}^{\mathrm{cocart}}
\end{equation}

\begin{Def}\label{d4.3.8} Let $(\mathrm H_{\star},\mathrm H^{\star},\mathrm H_!,\mathrm H^!)$ be a Grothendieck cross functor\footnote{What we really need here is pseudo-functors with values in category of monoidal categories}, such that for every $\mathrm m\colon\mathrm i\rightarrow \mathrm j\in \mathrm D$ morphism $\mathrm f=(\mathrm f_*,\mathrm f^*,\mathrm f^!)\colon\mathrm D(\mathrm E_{\mathrm i},\mathrm A_{\mathrm i})\longrightarrow\mathrm D(\mathrm E_{\mathrm j},\mathrm A_{\mathrm j}),$ is a Grothendieck context and we have an object $\mathrm w_{\mathrm D}\in \mathrm D_{\mathrm {cart}}(\underline{\Gamma}(\mathrm {E}^{!}),\mathrm A),$ then we call $(\mathrm H_{\star},\mathrm H^{\star},\mathrm H_!,\mathrm H^!,\mathrm w_{\mathrm D})$ Grothendieck cross functor with \underline{dualizing object $\mathrm w_{\mathrm D}$}.
\end{Def}
Grothendieck cross functor with dualizing object $(\mathrm H_{\star},\mathrm H^{\star},\mathrm H_!,\mathrm H^!,\mathrm w_{\mathrm D})$ is called constructible if $(\mathrm H_{\star},\mathrm H^{\star},\mathrm H_!,\mathrm H^!)$ is a constructible Grothendieck cross functor and objects $\mathrm w_{\mathrm i}$ are dualizing objects for the category $\mathrm D_{\mathrm c}(\underline{\Gamma}(\mathrm {E}_{\mathrm i}),\mathrm A),$ where $\mathrm i\in \mathrm D.$
\begin{Def}\label{d4.3.10} Let $(\mathrm H_{\star},\mathrm H^{\star},\mathrm H_!,\mathrm H^!,\mathrm w_{\mathrm D})$ be a Grothendieck cross functor with dualizing object, then we have triangulated functors called \underline{duality}
\begin{equation}\label{e4.3.8}
\mathbf D_{!\mapsto \star}\colon \mathrm D(\underline{\Gamma}(\mathrm {H}^{!}))\longleftrightarrow\mathrm D(\underline{\Gamma}(\mathrm{H}^{\star}))^{\circ}\colon \mathbf D_{\star\mapsto!},
\end{equation}
defined for every $\mathrm i\in \mathrm D$ by duality functors $\mathbb D_{\mathrm w_{\mathrm i}}.$
\end{Def}
\begin{Prop}\label{p4.3.5} Duality functor $\mathbf D_{\star\mapsto!}$ respects subcategories with cocartesian and cartesian cohomology:
\begin{equation}\label{e4.3.9}
\mathbf D_{\star\mapsto !}\colon \mathrm D_{\mathrm{cocart}}(\underline{\Gamma}(\mathrm {H}^{\star}))\longrightarrow\mathrm D_{\mathrm{cart}}(\underline{\Gamma}(\mathrm{H}^{!}))^{\circ},
\end{equation}
moreover if Grothendieck cross functor was constructible then both duality functors \eqref{e4.3.8} respects subcategories with cocartesian and cartesian cohomology:
\begin{equation}\label{e4.3.10}
\mathbf D_{\star\mapsto!}\colon \mathrm D_{\mathrm{cocart},\mathrm c}(\underline{\Gamma}(\mathrm {H}^{\star}))\longleftrightarrow\mathrm D_{\mathrm{cart},\mathrm c}(\underline{\Gamma}(\mathrm{H}^{!}))^{\circ}\colon \mathbf D_{!\mapsto\star},
\end{equation}
and induce mutual inverse equivalence.
\end{Prop}
\begin{proof} Statement obviously follows from isomorphisms \eqref{e4.2.13}.

\end{proof}

\begin{Cor}\label{c4.3.1} Let $(\mathrm H_{\star},\mathrm H^{\star},\mathrm H_!,\mathrm H^!,\mathrm w_{\mathrm D})$ be a constructible Grothendieck cross functor with dualizing object. Then following diagram commute:
\begin{equation}\label{e4.3.11}
\begin{diagram}[height=3em,width=4em]
\mathrm D_{\mathrm c}(\underline{\Gamma}(\mathrm {H}^{!})) & &   \rTo^{\Xi_!}   & &                   \mathrm D_{\mathrm{cart},{\mathrm c}}(\underline{\Gamma}(\mathrm {H}^{!}))           \\
\dTo^{\mathbf D_{!\mapsto\star}} & & & &  \dTo_ {\mathbf D_{!\mapsto\star}} \\
\mathrm D_{{\mathrm c}}(\underline{\Gamma}(\mathrm {H}^{\star}))^{\circ} & &   \rTo^{\Xi_{\star}}   & &                 \mathrm D_{\mathrm {cocart},\mathrm c}(\underline{\Gamma}(\mathrm {H}^{\star}))^{\circ} \\
\end{diagram}
\end{equation}
\end{Cor}
\begin{proof} Follows from the fact that inner hom functor takes colimits to limits and isomorphisms \eqref{e4.2.13}.

\end{proof}
\begin{Prop}\label{p4.3.6} Let $(\mathrm H_{\star},\mathrm H^{\star},\mathrm H_!,\mathrm H^!,\mathrm w_{\mathrm D})$ be a Grothendieck cross functor with dualizing object. Then following diagram commute:
\begin{equation}\label{e4.3.12}
\begin{diagram}[height=3em,width=4em]
\mathrm D_{\mathrm c}(\underline{\Gamma}(\mathrm {H}^{!})) & &   \rTo^{\mathbb V_{!\mapsto\star}}   & &                    \mathrm D_{\mathrm c}(\underline{\Gamma}(\mathrm {H}^{\star}))           \\
\dTo^{\mathbf D_{!\mapsto\star}} & & & &  \dTo_ {\mathbf D_{\star\mapsto!}} \\
\mathrm D_{\mathrm c}(\underline{\Gamma}(\mathrm {H}^{\star}))^{\circ} & &   \rTo^{\mathbb V_{\star\mapsto!}}   & &                 \mathrm D_{\mathrm c}(\underline{\Gamma}(\mathrm {H}^{!}))^{\circ} \\
\end{diagram}
\end{equation}
\end{Prop}
\begin{proof} Follows from fact that inner hom functor takes colimits to limits and isomorphisms \eqref{e4.2.13}.

\end{proof}
\begin{remark}\label{r4.3.2} Suppose that we have constructible Grothendieck cross functor over $\mathrm D.$ Using Corollary \ref{c4.3.1} we can define functor $\mathrm {\Xi}_{\star}$ by the rule $\mathrm {\Xi}_{\star}:=\mathbf D_{\mathrm D^{\circ}}\circ \mathrm {\Xi}_!\circ \mathbf D_{\mathrm D}.$

\end{remark}

\begin{Def}\label{d4.3.11} Let $(\mathrm H_{\star},\mathrm H^{\star},\mathrm H_!,\mathrm H^!,\mathrm w_{\mathrm D})$ be a constructible Grothendieck cross functor associated with ringed $\mathrm D$-topos $(\mathrm E,\mathrm A).$ Then we have triangulated functors
\begin{equation}\label{e4.3.13}
\mathbb D_{(\mathrm E,\mathrm A)}\colon\mathrm D_{\mathrm {cocart},\mathrm c}(\underline{\Gamma}(\mathrm {E}),\mathrm A) \longrightarrow   \mathrm D_{\mathrm {cocart},\mathrm c}(\underline{\Gamma}(\mathrm {E}),\mathrm A)^{\circ}
\end{equation}
defined as $\mathbb D_{(\mathrm E,\mathrm A)}:=\mathbf D_{!\mapsto\star}\circ\mathbb V_{\star\mapsto!}^{\mathrm{cocart}},$ and triangulated functor
\begin{equation}\label{e4.3.14}
\mathbb D_{\mathrm H^!}\colon\mathrm D_{\mathrm {cart},\mathrm c}(\underline{\Gamma}(\mathrm {H}^{!})) \longrightarrow   \mathrm D_{\mathrm {cart},\mathrm c}(\underline{\Gamma}(\mathrm {H}^{!}))^{\circ}
\end{equation}
defined as $\mathbb D_{\mathrm H^!}:=\mathbf D_{\star\mapsto!}\circ\mathbb V_{!\mapsto\star}^{\mathrm{cart}}.$ Functors $\mathbb D_{(\mathrm E,\mathrm A)}$ and $\mathbb D_{\mathrm H^!}$ are called \underline{Verdier duality} for Grothendieck cross functor with dualizing object $(\mathrm H_{\star},\mathrm H^{\star},\mathrm H_!,\mathrm H^!,\mathrm w_{\mathrm D}).$
\end{Def}
\begin{remark} It is also useful to consider so called \underline{full Verdier duality functor}:
\begin{equation}\label{}
\mathbb V_{!\mapsto\star}^{\mathrm{full}}\colon \clim_{{\mathrm D^{\circ}}}\mathrm H^!\longrightarrow \lim_{{\mathrm D}}\mathrm H^{\star}
\end{equation}
This functor is defined as composition of canonical functor $\mathrm {can}$ and covariant Verdier duality $\mathbb V_{!\mapsto \star}^{\mathrm{cocart}}.$

\end{remark}

\subsection{$\mathrm t$-structure}\label{tst} Let $(\mathrm E,\mathrm A)$ be the ringed $\mathrm D$-topos with corresponding Grothendieck cross functor $(\mathrm H_{\star},\mathrm H^{\star},\mathrm H_!,\mathrm H^!).$ Suppose that for every $\mathrm i\in \mathrm D$ we have a $\mathrm t$-structure $\mathrm D^{\leq 0}_{\mathrm i}$ on category $\mathrm D(\mathrm E_{\mathrm i},\mathrm A_{\mathrm i}),$ with corresponding truncation functors denoted by $\sigma^{\mathrm i}_{\leq0},$ such that for every $\mathrm m\colon \mathrm i\rightarrow \mathrm j\in \mathrm D$ functor $\mathrm f^!=\mathrm m^{\circ*}$ is left $\mathrm t$-exact and functor $\mathrm f^*:=\mathrm m_*$ is right $\mathrm t$-exact.
\begin{Lemma}\label{tl1} Category $\mathrm D(\underline{\Gamma}(\mathrm H_!))$ is naturally $\mathrm t$-category with corresponding $\mathrm t$-structure:
\begin{equation}\label{}
\mathrm D^{\leq 0}(\underline{\Gamma}(\mathrm H_!)):=\{\EuScript K\in \mathrm (\underline{\Gamma}(\mathrm H_!))\,|\, \EuScript K_{\mathrm i}\in \mathrm D^{\leq 0}_{\mathrm i}\}
\end{equation}
\end{Lemma}
\begin{proof} It easy to see that $\mathrm D^{\leq 0}(\underline{\Gamma}(\mathrm H_!))[1]\subset \mathrm D^{\leq 0}(\underline{\Gamma}(\mathrm H_!)).$ To prove that above subcategory defines $\mathrm t$-structure we need to introduce truncation functors. We define truncation functor $\sigma_{\leq 0}$ by componentwise truncation functors $\sigma^{\mathrm i}_{\leq0}.$ This truncation functor is well defined, since for every $\mathrm m\colon \mathrm i \rightarrow \mathrm j$ functor $\mathrm f_!=\mathrm m^*$ is right $\mathrm t$-exact.

\end{proof}
Suppose that for every $\mathrm i\in \mathrm D$ category $\mathrm D^{\leq 0}_{\mathrm i}$ (resp. $\mathrm D^{\geq 0}_{\mathrm i}$) contains homotopy limits (colimits) and homotopy limits and colimits are exact.
\begin{Def} We introduce $\mathrm t$-structures on categories $\mathrm D_{\mathrm {cocart}}(\underline{\Gamma}(\mathrm E),\mathrm A)$ and $\mathrm D_{\mathrm {cart}}(\underline{\Gamma}(\mathrm H^!))$ by following rules:
\begin{equation}\label{}
\mathrm D^{\leq 0}_{\mathrm {cocart}}(\underline{\Gamma}(\mathrm E),\mathrm A):=\{\EuScript K\in \mathrm D_{\mathrm {cocart}}(\underline{\Gamma}(\mathrm E),\mathrm A)\,|\, \EuScript K_{\mathrm i}\in \mathrm D^{\leqq 0}_{\mathrm i}\}
\end{equation}
\begin{equation}\label{}
\mathrm D^{\geq 0}_{\mathrm {cart}}(\underline{\Gamma}(\mathrm H^!)):=\{\EuScript K\in \mathrm D_{\mathrm {cart}}(\underline{\Gamma}(\mathrm H^!))\,|\, \EuScript K_{\mathrm i}\in \mathrm D^{\geq 0}_{\mathrm i}\}
\end{equation}
with corresponding truncation functors:
\begin{equation}\label{}
\tau_{\leq 0}\colon \mathrm D_{\mathrm {cocart}}(\underline{\Gamma}(\mathrm E),\mathrm A)\longrightarrow\mathrm D^{\leq 0}_{\mathrm {cocart}}(\underline{\Gamma}(\mathrm E),\mathrm A),\quad \tau_{\leq 0}:=\Xi_{\star}\circ\sigma_{\leq 0}\circ\Xi^{\star}
\end{equation}
and
\begin{equation}\label{}
\tau_{\geq 0}\colon \mathrm D_{\mathrm {cart}}(\underline{\Gamma}(\mathrm H^!))\longrightarrow\mathrm D^{\geq 0}_{\mathrm {cart}}(\underline{\Gamma}(\mathrm H^!)),\quad \tau_{\geq 0}:=\Xi_{!}\circ\sigma_{\geq 0}\circ\Xi^{!}.
\end{equation}
\end{Def}
\begin{Prop}\label{pts} Categories $\mathrm D_{\mathrm {cocart}}(\underline{\Gamma}(\mathrm E),\mathrm A)$ and $\mathrm D_{\mathrm {cart}}(\underline{\Gamma}(\mathrm H^!))$ with above data are indeed $\mathrm t$-categories.

\end{Prop}
\begin{proof} Let us stick to the case of category $\mathrm D_{\mathrm {cocart}}(\underline{\Gamma}(\mathrm E),\mathrm A),$ another case can be proved analogically. It easy to see  $\mathrm D^{\leq 0}_{\mathrm {cocart}}(\underline{\Gamma}(\mathrm E),\mathrm A)[1]\subset \mathrm D^{\leq 0}_{\mathrm {cocart}}(\underline{\Gamma}(\mathrm E),\mathrm A).$ The rest follows from adjunction property of functor $\Xi_{\star}$ and Lemma \ref{tl1} and exactness of limits.

\end{proof}

\begin{remark} Examples of above $\mathrm t$-structures are given by $\mathrm t$-structure on $\star$-crystals and perverse $\mathrm t$-structure on $!$-sheaves (Remark \ref{rps}).

\end{remark}

\subsection{Grothendieck operations}\label{sixtop} Let $(\mathrm E,\mathrm A)$ be the the ringed $\mathrm D$-topos with associated Grothendieck cross functor. Category $\mathrm D_{\mathrm {cocart}}(\mathrm E,\mathrm A)$ is the closed monoidal category with underlying tensor product:
\begin{equation}\label{}
\otimes^{\star} \colon \mathrm D_{\mathrm {cocart}}(\mathrm E,\mathrm A)\times \mathrm D_{\mathrm {cocart}}(\mathrm E,\mathrm A) \longrightarrow \mathrm D_{\mathrm {cocart}}(\mathrm E,\mathrm A),
\end{equation}
defined by the rule:
\begin{equation}\label{}
(\EuScript K\otimes^{\star} \EuScript F)_{\mathrm i}:= \EuScript K_{\mathrm i}\otimes^{\mathrm L} \EuScript F_{\mathrm i},\qquad \mathrm i\in\mathrm D
\end{equation}
and inner hom functor:
\begin{equation}\label{}
\underline{\mathrm {Hom}}^{\star}\colon \mathrm D_{\mathrm {cocart}}(\mathrm E,\mathrm A)^{\circ}\times \mathrm D_{\mathrm {cocart}}(\mathrm E,\mathrm A)\longrightarrow \mathrm D_{\mathrm {cocart}}(\mathrm E,\mathrm A),
\end{equation}
defined as composition:
\begin{equation}\label{innerh}
\underline{\mathrm {Hom}}^{\star}:=\Xi_{\star}\circ\underline{\mathrm {Hom}}^{\mathrm {naive}},\quad \underline{\mathrm {Hom}}^{\mathrm {naive}}(\EuScript E,\EuScript G)_{\mathrm i}:=\underline{\mathrm {Hom}}_{\mathrm D(\mathrm E_{\mathrm i},\mathrm A_{\mathrm i})}(\EuScript E_{\mathrm i},\EuScript G_{\mathrm i})
\end{equation}
Let $\mathrm f\colon (\underline{\Gamma}(\mathrm E),\mathrm A) \longrightarrow (\underline{\Gamma}(\mathrm E'),\mathrm A)$ be a morphism of topoi, then we have morphism between categories $ \mathrm D_{\mathrm {cocart}}(\mathrm E,\mathrm A)$ and $ \mathrm D_{\mathrm {cocart}}(\mathrm E',\mathrm A')$ given by pair of adjoint triangulated functors:
\begin{equation}\label{}
\mathrm f_{\star} \colon\mathrm D_{\mathrm {cocart}}(\mathrm E,\mathrm A)  \longleftrightarrow \mathrm D_{\mathrm {cocart}}(\mathrm E',\mathrm A')\colon\mathrm f^{\star},
\end{equation}
where functor $\mathrm f_{\star}$ is defined as composition $\mathrm f_{\star}=\Xi_{\star}\circ\underline{\mathrm R}^*(\mathrm f_{*})$ and $\mathrm f^{\star}=\underline{\mathrm L}^*(\mathrm f^*).$ We can also define $!$-operations:
\begin{equation}\label{}
\mathrm f_! \colon\mathrm D_{\mathrm {cocart}}(\mathrm E,\mathrm A)  \longleftrightarrow \mathrm D_{\mathrm {cocart}}(\mathrm E',\mathrm A')\colon\mathrm f^!
\end{equation}
By the rule:
\begin{equation}
\mathrm f_{!}:=\mathbb D_{(\mathrm E',\mathrm A')}\circ\mathrm f_{\star}\circ\mathbb D_{(\mathrm E,\mathrm A)},\qquad \mathrm f^{!}:=\mathbb D_{(\mathrm E,\mathrm A)}\circ\mathrm f^{\star}\circ\mathbb D_{(\mathrm E',\mathrm A')}.
\end{equation}

\subsection{Digression: norm maps}\label{decat} 

By \textit{Picard groupoid} we understand symmetric monoidal category $\EuScript P$, with underlying tensor product $+_{\EuScript P}\colon \EuScript P\times \EuScript P\longrightarrow \EuScript P,$ such that $\EuScript P$ is the groupoids and for every object $\mathrm a\in \EuScript P$ functor $\mathrm x\mapsto \mathrm a+_{\EuScript P}\mathrm x$ defines autoequivalence of $\EuScript P.$ Unit object will be denote by $0_{\EuScript P},$ it follows from axioms that for every object $\mathrm x\in \EuScript P,$ there exists unique inverse object $-\mathrm x.$ We say that Picard groupoid $\EuScript P$ is \textit{strictly commutative}, if $\EuScript P$ is strictly commutative as monoidal category. We defined \textit{homotopy groups of Picard groupoid} as follows: $\pi_0(\EuScript P)$ is group of isomorphism classes of objects in $\EuScript P$ and $\pi_1:=\mathrm {Aut}_{\EuScript P}(0).$

By $\mathrm {Pic}$ we denote $2$-category of Picard groupoids, that is the $2$-category with objects Picard groupoid and $1$-morphisms are given by tensor functors and $2$-morphisms are given by natural transformations of monoidal functors. We have $2$-subcategory of strict Picard groupoids, denoted by $\mathrm {Pic}_{\mathrm{stric}}$ and $2$-subcategory of \textit{discrete Picard groupoids} denoted by $\mathrm {Pic}_{\mathrm {disc}}.$ Corresponding homotopy category of Picard groupoids will be denote by $\mathrm {Pic}^{\flat}.$ We have subcategories of strict and discreet Picard groupoids: $\mathrm {Pic}^{\flat}_{\mathrm{stric}}$ and $\mathrm {Pic}^{\flat}_{\mathrm {disc}}$ (these categories are defined as categories with objects given by strict Picard groupoids and morphisms are equivalence classes of monoidal functors).

Denote by $\mathrm D^{[-1,0]}(\mathrm {Ab})$ derived category of length $2$ complexes. We have functor:
\begin{equation}\label{}
\mathrm {ch}\colon \mathrm {D}^{[-1,0]}(\mathrm {Ab})\longrightarrow\mathrm {Pic}^{\flat}_{\mathrm{stric}},
\end{equation}
which takes complex $\mathrm A^{-1}\rightarrow \mathrm A^0$ of abelian groups and associates with it strict Picard groupoid $\EuScript P$ generated by elements of $\mathrm A^0,$ $\mathrm {Ob}(\EuScript P):=\mathrm A^0,$ with tensor structure $+_{\EuScript P}$ given by addition operation in $\mathrm A^0$ and unit object given by zero in $\mathrm A^{0}.$ For every $\mathrm c,\mathrm b\in \EuScript G$ set of morphisms defines as follows:
\begin{equation}\label{}
\mathrm {Hom}_{\EuScript P}(\mathrm c,\mathrm b):=\{\mathrm f\in \mathrm A^{-1}\,|\,\mathrm d(\mathrm f)=\mathrm b-\mathrm c\}.
\end{equation}
We have following classical result:

\begin{Lemma}\label{dlp1} Functor $\mathrm {ch}$ defines equivalence of categories.

\end{Lemma}

\begin{proof} See \cite[Expos\'{e} $\mathrm {XVIII}$]{SGA4(3)}.

\end{proof}
\begin{Ex} Under above equivalence category abelian groups $\mathrm {Ab}$ i.e. subcategory of complexes which live in zero degree corresponds to subcategory of discrete Picard groupoids $\mathrm {Pic}_{\mathrm {disc}}^{\flat}.$
\end{Ex}
\begin{remark} In \cite[Expos\'{e} $\mathrm {XVIII}$]{SGA4(3)} Deligne actually proves equivalence between $2$-category $\mathrm {Pic}$ and $2$-category $\mathrm C^{[-1,0]}(\mathrm {Ab}),$

\end{remark}
 Let $\mathrm M$ be a \textit{Mackey functor} \cite{Lin}\footnote{What we call the Mackey functor in \cite{Lin} is called the P-functor} from category $\mathrm D$ to the category $\mathrm {Ab}$ of abelian groups:
\begin{equation}\label{}
\mathrm M\colon \mathrm {span}(\mathrm D)\longrightarrow\mathrm {Ab}
\end{equation}
Category of all Mackey functors from the category $\mathrm D$ to the category $\mathrm {Ab}$ will be denoted by $\mathrm {mk}_{\mathrm D}(\mathrm {Ab}):=\underline{\mathrm {Hom}}(\mathrm {span}(\mathrm D),\mathrm {Ab}).$

 Denote by $\mathrm {mk}^{\star!}_{\mathrm D^{\circ}}$ category, whose objects are Mackey $\star!$-pseudo-functors and morphism are equivalence classes of pseudo-natural transformations.
\begin{Prop} Functor $\mathrm {ch}$ induces functor:
\begin{equation}\label{e4.4.2}
\mathrm {ch}\colon\mathrm {mk}_{\mathrm D}(\mathrm {Ab})\longrightarrow\mathrm  {mk}^{\star!}_{\mathrm D^{\circ}},
\end{equation}
which takes Mackey functor $\mathrm M$ and associates to it Mackey $\star!$-pseudo-functors $\mathrm H^{\star!}$ over $\mathrm D^{\circ},$ with values in subcategory of discrete Picard groupoids (pr\'{e}champ de Picard).
\end{Prop}
\begin{proof} See \cite[Expos\'{e} $\mathrm {XVIII}$]{SGA4(3)}.

\end{proof}
Denote by $\mathrm {Rig}_{\mathrm {cocart}}$ category with objects cocartesian rigid symmetric monoidal categories and equivalence classes of monoidal functor between them. We have functor:
\begin{equation}\label{}
 \mathrm {grp}\colon\mathrm {Rig}_{\mathrm {cocart}}\longrightarrow  \mathrm {Pic}^{\flat}
\end{equation}
which takes underlying groupoid $\mathrm C^{\times}$ of rigid symmetric monoidal category $\mathrm C.$ We have obvious:
\begin{Prop}
Functor $\mathrm {grp}$ admits a left adjoint:
\begin{equation}\label{}
\mathrm {comp}\colon \mathrm {Pic} \longleftrightarrow \mathrm {Rig}_{\mathrm {cocart}}
\end{equation}
called completion functor.
\end{Prop}

Assume that category $\mathrm D$ is finite and let $\mathrm H^{\star!}$ be a Mackey functor, with values in discreet Picard groupoids. By $\mathrm H^{\star!}_+$ we denote corresponding completion of Mackey functor. Since fibers of pseudo-functor $\mathrm H^{\star!}_+$ have finite colimits we can apply categorical Verdier duality construction:
\begin{equation}\label{e4.4.4}
\mathbb V_{!\mapsto\star}^{\mathrm {full}}\colon \clim_{{\mathrm D^{\circ}}}\mathrm j^*\mathrm H^{\star!}_+\longrightarrow \lim_{{\mathrm D}}\mathrm i^*\mathrm H^{\star!}_+
\end{equation}
We have obvious:
\begin{Lemma} Categories $\clim_{{\mathrm D^{\circ}}}\mathrm j^*\mathrm H^{\star!}_+$ and $ \lim_{{\mathrm D}}\mathrm i^*\mathrm H^{\star!}_+$ are cocartesian rigid symmetric categories.

\end{Lemma}
\begin{proof} Tensor product can be defined fiberwise, since for every morphism $\mathrm m$ in $\mathrm D^{\circ}$ functor $\mathrm m_!$ is the Picard functor.

\end{proof}
By previous Lemma Verdier duality functor $\mathbb V_{!\mapsto\star}^{\mathrm {full}}$ induces morphism:
\begin{equation}\label{e4.4.5}
 {}^1\mathbb V_{!\mapsto\star}^{\mathrm {full}}\colon \clim_{{\mathrm D^{\circ}}} \mathrm j^*\mathrm M\longrightarrow\lim_{{\mathrm D}} \mathrm i^*\mathrm M,
\end{equation}
which we call \textit{Verdier duality map}. We have following baby example: 
\begin{Ex}\label{ex.4.4.1} Let $\mathrm G$ be a finite group and $\mathrm X=\mathrm K(\mathrm G,\mathrm 1).$ We denote by $\Pi_{1}(\mathrm X)$ fundamental groupoid of $\mathrm X.$ We have well known correspondence between category $\mathrm {Rep}(\mathrm G)$ of representations of $\mathrm G$ and category of local systems of abelian groups on $\mathrm X:$
\begin{equation}\label{e4.4.6}
\mathrm{LC}(\mathrm K(\mathrm G,\mathrm 1),\mathrm{Ab})\overset{\sim}{\longrightarrow} \mathrm {Rep}(\mathrm G)
\end{equation}
Let $\mathrm M$ be a representation of $\mathrm F$ and $\EuScript L_{\mathrm M}\in \mathrm{LC}(\mathrm K(\mathrm G,\mathrm 1))$ be the corresponding local system, then we can associate with it Mackey functor $\mathrm i_*(\EuScript L_{\mathrm M}).$ Categorical Verdier duality construction supplies us with the morphism:
\begin{equation}\label{e4.4.7}
{}^1\mathbb V_{!\mapsto\star}^{\mathrm {full}}\colon \clim_{{\Pi_1(\mathrm X)}} \EuScript L_{\mathrm M}\longrightarrow\lim_{{\Pi_1(\mathrm X)}} \EuScript L_{\mathrm M}
\end{equation}
which is can be rewritten as:
\begin{equation}\label{e4.4.8}
{}^1\mathbb V_{!\mapsto\star}^{\mathrm {full}}\colon \mathrm M_{\mathrm G}\longrightarrow\mathrm M^{\mathrm G},\quad \mathrm v\mapsto \sum_{\mathrm g\in \mathrm G}\mathrm g\mathrm v
\end{equation}
where space $\mathrm M_{\mathrm G}$ is the space of coinvariants and $\mathrm M^{\mathrm G}$ is the space of invariants of $\mathrm G.$ Note, that in this example, Verdier duality $\mathbb V_{\mathrm D,\mathrm{full}}^1$ is actually isomorphic to canonical functor $\mathrm {can}$, since our base category $\mathrm G$ is groupoid.
\end{Ex}
It is natural try to extend above construction to spectral Mackey functors. To do it we need to switch our setting to the world of $\infty$-groupoids.
\begin{remark} This remark does not contain rigorous statements or proofs and should be viewed as a sketch. We have stable category of spectra $\EuScript S\mathrm p,$ which is monoidal category with smash product $\wedge$ and unit given by sphere spectrum $\mathbb S.$ \textit{Homotopy groups} of spectra $\mathrm X$ are defined as follows:
\begin{equation}\label{}
\pi_{\mathrm i}(\mathrm X):=\clim_{\mathrm i\in \mathbb N} \pi_{\mathrm i+\mathrm j}(\mathrm X_{\mathrm j}).
\end{equation}
Denote by $\EuScript S\mathrm p^{\geq 0}$ subcategory of \textit{connective spectra}. It is objects are given by spectra $\mathrm X\in \EuScript S\mathrm p,$ such that $\pi_{\mathrm i}(\mathrm X)=0$, for all $\mathrm i<0.$ Subcategory $\EuScript S\mathrm p^{\geq 0}$ with corresponding truncation functor:
\begin{equation}\label{}
\tau_{\geq 0}\colon  \EuScript S\mathrm p\longrightarrow \EuScript S\mathrm p^{{\geq 0}},
\end{equation}
defines $\mathrm t$-structure on $\EuScript S\mathrm p,$ with heart given by category of abelian groups $\mathrm {Ab}.$ We say that spectra $\mathrm X\in \EuScript S\mathrm p^{\geq 0}$ has a \textit{stable homotopy type $\mathrm n$} if all homotopy groups vanishes:
\begin{equation}\label{}
\pi_{\mathrm i}(\mathrm X)=0\qquad \mathrm i> \mathrm n.
\end{equation}

Recall following correspondence which is due to Grothendieck (see \cite[p.114 ]{PS} and \cite{Dr}).
\begin{Const}\label{const1} Functor $\Pi_{\infty}$ takes topological space $\mathrm X$ to corresponding fundamental $\infty$-groupoid $\Pi_{\infty}(\mathrm X)$ and functor $\mathrm B$ takes $\infty$-groupoid $\EuScript G$ to classifying space $\mathrm B\EuScript G.$ This functors induces functors between categories of connective spectra and category Picard $\infty$-groupoids $\infty-\mathrm {Pic}.$ Moreover these functors take stable $\mathrm n$-types to Picard $\mathrm n$-groupoids and vice versa. These functors are adjoint and induce equivalence of homotopy categories:
\begin{equation}\label{87}
\Pi_{\infty}\colon\EuScript S\mathrm p^{{\geq 0}}\longleftrightarrow\infty\text{-}\mathrm {Pic}\colon \mathrm B
\end{equation}
\end{Const}
\begin{remark} This correspondence is a stable analog of \textit{Grothendieck homotopy hypothesis} \cite{PS}
\end{remark}

Denote by $\mathrm {C}^{\leq 0}(\mathrm {Ab})$ category of cochain complexes of abelian groups, which live in non positive degree. And by $\mathrm {Ab}^{\Delta}:=\underline{\mathrm {Hom}}(\Delta,\mathrm {Ab})$ we denoted category of simplicial abelian groups. We can consider $\mathrm {C}^{\leq 0}(\mathrm {Ab})$ as subcategory of $\EuScript S\mathrm p^{\geq 0}$ via following constructions:
\begin{Prop} Normalization functor:
\begin{equation}\label{65}
\mathrm {Norm}\colon\mathrm {Ab}^{\Delta}\longrightarrow \mathrm {C}^{\leq 0}(\mathrm {Ab}
\end{equation}
induces equivalence of homotopy categories.
\end{Prop}
This construction is called \textit{Dold-Kan equivalence}. For any simplicial abelian group $\mathrm A$ one can associate another abelian simplicial group $\mathrm {BA},$ called \textit{classifying group of $\mathrm A.$} Thus, to a simplicial abelian group $\mathrm A$ we can associate spectra by following rule:
\begin{equation}\label{76}
\mathrm {BA}_{\mathrm n}:=\Omega\underbrace{\mathrm B\circ \cdots\circ\mathrm {B}}_{\mathrm n}\mathrm {A}
\end{equation}
\begin{Cor} Category  $\mathrm {C}^{\leq 0}(\mathrm {Ab})$ of cochain complexes of abelian groups, which live in non positive degree can be considered as full subcategory of $\EuScript S\mathrm p^{\geq 0}.$

\end{Cor}
\begin{remark} Note that under above inclusion functor, truncation functors on complexes $\tau_{\mathrm i}$ correspond to truncation functors on spectra.

\end{remark}
\begin{Cor} Under stable homotopy hypothesis category $\mathrm {C}^{\leq 0}(\mathrm {Ab})$ corresponds to subcategory of strict Picard $\infty$-groupoids:
\begin{equation}\label{8}
\mathrm {C}^{\leq 0}(\mathrm {Ab})\longleftrightarrow\infty-\mathrm {Pic}_{\mathrm{strict}}
\end{equation}
\end{Cor}
\begin{remark} Lemma \ref{dlp1} can be considered as special case of this equivalence. Subcategory $\mathrm C^{[-1,0]}(\mathrm {Ab})$ of complexes $\mathrm A^{-1}\rightarrow \mathrm A^0$ can be considered as subcategory of stable $1$-types and therefore corresponds to subcategory of strict Picard groupoids $\mathrm {Pic}_{\mathrm{strict}}^{\flat}.$
\end{remark}
\begin{Ex} Abelian group $\mathrm A$ correspond to \textit{Eilenberg-Maclane spectrum} $\mathrm {HA}$ and corresponding Picard groupoids is discrete Picard groupoid.
\end{Ex}
Let $\mathrm M$ be a spectral Mackey functor, over category $\mathrm D,$ which takes values in connective spectra $\EuScript S\mathrm p^{\geq 0}.$ Then existence of morphism
\begin{equation}\label{}
{}^1\mathbb V_{!\mapsto \star}^{\mathrm {full}}\colon\hclim_{{\mathrm D^{\circ}}}\mathrm i^*\mathrm M\longrightarrow \hlim_{{\mathrm D}}\mathrm j^*\mathrm M
\end{equation}
is equivalent to existence of morphism:
\begin{equation}\label{}
{}^1\mathbb V_{!\mapsto \star}^{\mathrm {full}}\colon\hclim_{{\mathrm D^{\circ}}}\mathrm i^*\mathrm M\longrightarrow \tau_{\geq 0}\mathrm j^*\hlim_{{\mathrm D}}\mathrm M
\end{equation}
due to adjunction of $\tau_{\geq 0}.$ Thus we can apply Construction \ref{const1} and then use categorical Verdier duality construction to obtain desired morphism. We have following homotopy baby example:
\begin{Ex} Let $\mathrm G$ be a finite group and $\mathrm M$ is the representation of $\mathrm G.$ We can associate with it spectral Mackey functor on Kan complex of $\mathrm G.$ Thus categorical Verdier duality supplies us with the morphism:
\begin{equation}\label{}
{}^1\mathbb V_{!\mapsto \star}^{\mathrm {full}}\colon\hclim_{{\mathrm G}}\EuScript L_{\mathrm M}\longrightarrow \hlim_{{\mathrm G}}\EuScript L_{\mathrm M},
\end{equation}
which is given by norm morphism from previous example.
\end{Ex}

\begin{remark} We can also try to extend this construction to the setting of general spectra. Category of spectra $\EuScript S\mathrm p$ is equivalent to stabilization of category of connective spectra $\EuScript S\mathrm p^{\geq 0}.$ That means that we have $\infty$-bifibration $\mathrm {Stab}(\EuScript S\mathrm p^{\geq 0})$ over $\mathbb N,$ with fibers given by category $\EuScript S\mathrm p^{\geq 0}$ and functors between fibers are given by loop space object $\Omega$ and suspension functor $\Sigma.$ Under this equivalence spectrum $\mathrm E\in \EuScript S\mathrm p$ goes to the sequence $\{\mathrm E_{\mathrm i}\}_{\mathrm i\in \mathbb N}:$
\begin{equation}\label{}
\mathrm f\colon \Sigma\mathrm E_{\mathrm i} \longrightarrow\mathrm E_{\mathrm i+1},\qquad \mathrm E_{\mathrm i}:=\Sigma\circ\tau_{\geq-1}\mathrm E.
\end{equation}
Functor in inverse direction is given by attaching to the sequence $\{\mathrm E_{\mathrm i}\}_{\mathrm i\in \mathbb N}$ following homotopy colimit:
\begin{equation}\label{}
\{\mathrm E_{\mathrm i}\}_{\mathrm i\in \mathbb N}\longmapsto \hclim_{\mathrm i\in \mathbb N} \Sigma^{\mathrm i} \mathrm E_{\mathrm i}
\end{equation}
Therefore if we want to construct Verdier duality morphism between Mackey functors with values in $\EuScript S\mathrm p$ it is enough to construct family of morphisms between connective spectra, compatible with functor $\Omega.$ In order to be compatible, we impose condition that $\mathrm D$ is finite category.
\end{remark}

\end{remark}

\section{Sheaves on Diagrams}

\subsection{Notations} We denote by $\mathrm {Top}$ category with objects locally compact, locally completely paracompact topological spaces of finite cohomological dimension, morphisms given by continues maps of topological spaces. It is a monoidal category with underlying cartesian tensor product. With an object $\mathrm X\in \mathrm {Top}$ we associate topos $\mathrm {Sh}(\mathrm X)$ of sheaves of sets on $\mathrm X.$ Let $\underline{\mathrm A}\in \mathrm {Sh}(\mathrm X)$ be a constant ring object, associated with ring $\mathrm A$ (we further assume that $\mathrm A$ is the field). Then we denote by $\mathrm {Sh}_{\mathrm A}(\mathrm X)$ abelian category of sheaves of $\underline{\mathrm A}$-modules. Let $\mathrm f\colon\mathrm X \longrightarrow \mathrm Y$ be a morphism in $\mathrm {Top},$ we have a morphism of topoi, which is given by pair of adjoint functors:
$\mathrm f_*\colon \mathrm {Sh}(\mathrm X)\longleftrightarrow \mathrm {Sh}(\mathrm Y)\colon\mathrm f^*.$ When we restrict these functors to subcategories of sheaves of $\underline{\mathrm A}$-modules functor $\mathrm f^*$ is exact and functor $\mathrm f_*$ is left exact and moreover functor $\mathrm f_*$ commutes with colimits. We have associated derived functors:
\begin{equation}\label{}
\underline{\mathrm {R}}(\mathrm f_*)\colon \mathrm {D}(\mathrm X,\mathrm A)\longleftrightarrow \mathrm {D}(\mathrm Y,\mathrm A)\colon \underline{\mathrm {L}}(\mathrm f^*),
\end{equation}
where $\mathrm D(\mathrm X,\mathrm A)$ is derived category of category $\mathrm {Sh}_{\mathrm A}(\mathrm X).$ We also have pair of adjoint $!$-functors
\begin{equation}\label{}
\underline{\mathrm R}(\mathrm f_!)\colon \mathrm {D}(\mathrm X,\mathrm A)\longrightarrow \mathrm {D}(\mathrm Y,\mathrm A)\colon\mathrm f^!
\end{equation}
Functor $\mathrm f^!$ commutes with filtered homotopy colimits.
We have unital tensor product $\otimes:$
\begin{equation}\label{}
\otimes\colon  \mathrm {D}(\mathrm X,\mathrm A)\times  \mathrm {D}(\mathrm X,\mathrm A) \longrightarrow  \mathrm {D}(\mathrm X,\mathrm A)
\end{equation}
and inner hom functor $\underline{\mathrm {R}\mathrm {Hom}}:$
\begin{equation}\label{}
\underline{\mathrm {R}\mathrm {Hom}}\colon  \mathrm {D}(\mathrm X,\mathrm A)^{\circ}\times  \mathrm {D}(\mathrm X,\mathrm A) \longrightarrow  \mathrm {D}(\mathrm X,\mathrm A).
\end{equation}
These operations make $\mathrm {D}(\mathrm X,\mathrm A)$ into closed monoidal category. Functor $\mathrm f^*$ is monoidal and moreover for closed embedding $\mathrm i\colon \mathrm Z\hookrightarrow \mathrm X$ functor $\mathrm i_*$ is monoidal. For every $\mathrm X\in \mathrm {Top}$ we have distinguished morphism $\mathrm a_{\mathrm X}\colon \mathrm X\longrightarrow \mathrm {pt},$ where $\mathrm{pt}$ is a point. We denote by $\mathrm w_{\mathrm X}:=\mathrm a_{\mathrm X}^!\mathrm A$ dualizing complex in $\mathrm D(\mathrm X,\mathrm A).$ We define Verdier duality functor:
 \begin{equation}\label{e5.1.1}
 \mathbb D\colon \mathrm D(\mathrm X,\mathrm A)\longrightarrow \mathrm D(\mathrm X,\mathrm A)^{\circ},\qquad \EuScript K\longmapsto \mathrm {\underline{RHom}}(\EuScript K, \mathrm w_{\mathrm X}).
 \end{equation}
We have biduality morphism $\mathrm {id}_{\mathrm D(\mathrm X,\mathrm A)}\longrightarrow \mathbb D\circ \mathbb D.$ These functors satisfy natural compatibility listed in ~\cite{KS} and \cite{Spa}.

 For the closed embedding $\mathrm i\colon\mathrm Z\hookrightarrow \mathrm X$ of topological spaces we have isomorphism of functors $\mathrm i_*\cong \mathrm i_!$ and triangulated functor $\mathrm i^!$ is isomorphic to right derived functor of functor which takes sheaf to subsheaf which sections have support in $\mathrm Z.$ Functor $\mathrm i^!$ commutes with homotopy colimits. Category $\mathrm D(\mathrm X,\mathrm A)$ carries another unital monoidal structure denoted by $\otimes^!$ and defined by the rule: $\EuScript K\otimes^!\EuScript G:=\Delta^!\EuScript K\boxtimes \EuScript G,$ where $\Delta\colon \mathrm X\hookrightarrow \mathrm X\times \mathrm X.$ This tensor product is called $!$-tensor product, unit is given by $\mathrm w_{\mathrm X}.$ Verdier duality functor takes tensor product $\otimes^!$ to ordinary tensor product $\otimes.$ Note that for a proper morphism $\mathrm f$ functor $\mathrm f^!$ is monoidal with respect to $!$-tensor product and for closed embedding $\mathrm i$ functor $\mathrm i_*$ is also monoidal with respect to $\otimes^!.$

Let $(\mathrm X,\mathrm S)$ be a pair, where $\mathrm X\in \mathrm {Top}$ and $\mathrm S$ is the stratification on $\mathrm X.$ We denote by $\mathrm D_{\mathrm c}(\mathrm X,\mathrm S),$ full triangulated subcategory of $\mathrm D(\mathrm X,\mathrm A),$ whose object are complexes with $\mathrm  S$-constructible cohomology, objects of latter category are called cohomological $\mathrm  S$-constructible sheaves. Category $\mathrm D_{\mathrm c}(\mathrm X,\mathrm  S)$ is also closed monoidal. Let $\mathrm f\colon (\mathrm X,\mathrm  S) \longrightarrow (\mathrm Y,\mathrm  S')$ be a stratified map in $\mathrm {Top},$ then functors $(\underline{\mathrm L}(\mathrm f^*),\underline{\mathrm R}(\mathrm f_*))$ and $(\underline{\mathrm R}(\mathrm f_!),\mathrm f^!)$ preserve constructible subcategory. Verdier duality functor $\mathbb D$ also preserves this subcategory and moreover defines antiequivalence.

\subsection{$\star$-sheaves} We denote by $\textbf{Sh}_{\mathrm {top}}$ category whose objects are pairs $(\mathrm X,\EuScript K),$ where $\mathrm X\in \mathrm {Top}$ is a topological space and $\EuScript K$ is a sheaf of sets on $\mathrm X.$ Morphisms between $(\mathrm X_1,\EuScript K_1)$ and $(\mathrm X_2,\EuScript K_2)$ are given by pairs $(\mathrm h,\mathrm H),$ where $\mathrm h\colon \mathrm X_2\rightarrow \mathrm X_1$ is map of topological spaces and $\mathrm H\colon \EuScript K_2\rightarrow \mathrm h^*\EuScript K_1$ is a morphism of sheaves. This category is monoidal with underlying tensor product $(\mathrm X_1,\EuScript K_1)\otimes (\mathrm X_2,\EuScript K_2):=(\mathrm X_1\times\mathrm X_2,\EuScript K_1\boxtimes \EuScript K_2).$ We have $\mathrm {Top}^{\circ}$-topos $\textbf{Sh}_{\mathrm {top}},$ whose fiber over topological space $\mathrm X$ is the category $\mathrm {Sh}(\mathrm X).$ For the diagram $\mathrm X_{\EuScript J}\colon \EuScript J\longrightarrow \mathrm{Top},$ whose morphisms are closed embeddings $\mathrm X_{\mathrm i}\longrightarrow \mathrm X_{\mathrm j}$ for every $\mathrm m\colon\mathrm i\rightarrow \mathrm j\in \EuScript J$ we associate category $\textbf{Sh}_{\mathrm {top}}^{\EuScript J}:=\EuScript J^{\circ}\times_{\mathrm {Top}^{\circ}} \textbf{Sh}_{\mathrm {top}},$ which will be a $\EuScript J^{\circ}$-topos. We assume that $\mathrm X_{\EuScript J}$ sends fiber products in $\EuScript J$ to fiber products in $\mathrm {Top}.$ We have corresponding total topos $\underline{\Gamma}(\textbf{Sh}_{\mathrm {top}}^{\EuScript J}).$ Let us fix constant ring object $\mathrm A$ in $\textbf{Sh}_{\mathrm {top}}^{\EuScript J}$, thus we have ringed $\EuScript J^{\circ}$-topos $(\textbf{Sh}_{\mathrm {top}}^{\EuScript J},\mathrm A).$

\begin{Def}\label{d5.2.1}
\underline{Category of $\star$-sheaves} $\mathrm {Sh}^{\star}(\mathrm X_{\EuScript J})$ on the diagram $\mathrm X_{\EuScript J}$ is the category of modules ${\mathrm{Mod}}(\underline{\Gamma}(\textbf{Sh}_{\mathrm {top}}^{\EuScript J}),\mathrm A).$
\end{Def}
Category $\mathrm {Sh}^{\star}(\mathrm X_{\EuScript J})$ is an abelian category and we have corresponding derived category $\mathrm D(\underline{\Gamma}(\mathbf {Sh}_{\mathrm {top}}^{\EuScript J},\mathrm A)).$ We have obvious:
\begin{Lemma}
Abelian bifibration $\underline{\mathrm{Mod}}((\textbf{Sh}_{\mathrm {top}}^{\EuScript J}),\mathrm A)\longrightarrow \EuScript J^{\circ}$ defines a pair of pseudo-functors $(\mathrm H_{\star}^{{\EuScript J}},\mathrm H^{\star}_{{\EuScript J}}),$\footnote{When it wont lead to misunderstanding we suppress $\mathrm X_{\EuScript J}$ from notation} with the structure of contradirectional $\mathrm e^{\star}$-functor.
\end{Lemma}
\begin{proof} Evident.

\end{proof}
We denote by $\mathrm {Sh}^{\star}_{\mathrm {adm}}(\mathrm X_{\EuScript J})$ subcategory of $\mathrm {Sh}^{\star}(\mathrm X_{\EuScript J})$ cocartesian sections of bifibration $\underline{\mathrm{Mod}}((\textbf{Sh}_{\mathrm {top}}^{\EuScript J}),\mathrm A).$
\begin{Def}\label{d5.2.2}
We denote by $\mathrm D_{\mathrm {cocart}}(\underline{\Gamma}({\mathrm H}^{\star}_{\EuScript J}))$ full triangulated subcategory of $\mathrm D(\underline{\Gamma}(\mathrm H^{\star}_{\EuScript J})),$ whose cohomology are in category $\mathrm {Sh}^{\star}_{\mathrm {adm}}(\mathrm X_{\EuScript J}).$ Objects of the category $\mathrm D_{\mathrm {cocart}}(\underline{\Gamma}({\mathrm H}^{\star}_{\EuScript J}))$ are called \underline{admissible $\star$-sheaves} on diagram $\mathrm X_{\EuScript J}.$
\end{Def}
Note, that heart of $\mathrm t$-category $\mathrm D_{\mathrm {cocart}}(\underline{\Gamma}({\mathrm H}^{\star}_{\EuScript J}))$ (with respect to the standard $\mathrm t$-structure) is equivalent to $\mathrm {Sh}^{\star}_{\mathrm {adm}}(\mathrm X_{\EuScript J}).$ Suppose that for every $\mathrm i\in \EuScript J$ we have stratification $\EuScript S_{\mathrm i}$ on $\mathrm X_{\mathrm i}$ and for every morphism
 $\mathrm m\colon\mathrm i\rightarrow \mathrm j\in \EuScript J$ we have stratified map $\mathrm X_{\mathrm i}\longrightarrow \mathrm X_{\mathrm j}.$ Then we have subcategories of cohomological constructible $\star$-sheaves $\mathrm D_{\mathrm c}(\underline{\Gamma}({\mathrm H}^{\star}_{{\EuScript J}}))$ and $\mathrm D_{\mathrm {cocart},\mathrm c}(\underline{\Gamma}({\mathrm H}^{\star}_{{\EuScript J}})).$

\subsection{$!$-sheaves}
Let us consider following category. It is objects are pairs $(\mathrm X,\EuScript K)$ where $\mathrm X\in \mathrm {Top}$ is the topological space and $\EuScript K$ is the sheaf of $\mathrm A$-modules on $\mathrm X.$ Morphisms are given by following data. A morphism from $(\mathrm X_1,\EuScript K_1)$ to $(\mathrm X_2,\EuScript K_2)$ is pair $(\mathrm f,\mathrm F),$ where $\mathrm f\colon \mathrm X_1\rightarrow \mathrm X_2$ is a morphism of topological spaces and $\mathrm F\colon \mathrm f_!\EuScript K_1\rightarrow \EuScript K_2$ is a morphism of sheaves. We denote this category by $\mathrm H_{!}^{\mathrm {top}}.$ This category will be monoidal with tensor product $(\mathrm X_1,\EuScript K_1)\otimes (\mathrm X_2,\EuScript K_2):=(\mathrm X_1\times\mathrm X_2,\EuScript K_1\boxtimes \EuScript K_2).$ We have cofibered category $\mathrm H_{!}^{\mathrm {top}}$ over $\mathrm{Top}.$ Fiber over $\mathrm X$ is given by category $\mathrm {Sh}_{\mathrm A}(\mathrm X).$ Let $\mathrm X_{\EuScript J}$ be a diagram of topological spaces, we associate with it cofibered category $\mathrm H_{!}^{{\EuScript J}}:= \EuScript J^{}\times_{\mathrm{Top}} \mathrm H_{!}^{\mathrm {top}}$ over $\EuScript J.$
\begin{Def}\label{d5.3.1}
\underline{Category of $!$-sheaves} $\mathrm {Sh}^!(\mathrm X_{\EuScript J})$ on diagram $\mathrm X_{\EuScript J}$ is the total category of cofibration  $\mathrm H_{!}^{{\EuScript J}}\longrightarrow \EuScript J.$
\end{Def}
Category $\mathrm {Sh}^!(\mathrm X_{\EuScript J})$ is an abelian and we have corresponding derived category $\mathrm D(\underline{\Gamma}(\mathrm H_{!}^{ {\EuScript J}})).$ We have a bifibration over $\EuScript J$ denoted by $(\mathrm H^{\EuScript J}_!,\mathrm H_{\EuScript J}^!):$
\begin{Lemma} We have a pair of pseudo-functors $(\mathrm H^{\EuScript J}_!,\mathrm H_{\EuScript J}^!),$ which defines the contradirectional $\mathrm e^{!}$-functor.

\end{Lemma}
\begin{proof} Evident.

\end{proof}
We have subcategory $\mathrm {Sh}^!_{\mathrm {adm}}(\mathrm X_{\EuScript J})$ of cartesian sections of fibration $\mathrm H_{\EuScript J}^!\longrightarrow\EuScript J.$
\begin{Def}\label{d5.3.2}
We denote by $\mathrm D_{\mathrm {cart}}(\underline{\Gamma}({\mathrm H}^{!}_{{\EuScript J}}))$ full triangulated subcategory of $\mathrm D(\underline{\Gamma}(\mathrm H^{!}_{{\EuScript J}})),$ whose cohomology are in category $\mathrm {Sh}^{!}_{\mathrm {adm}}(\mathrm X_{\EuScript J}).$ Objects of the category $\mathrm D_{\mathrm {cart}}(\underline{\Gamma}({\mathrm H}^{!}_{{\EuScript J}}))$ are called \underline{admissible $!$-sheaves} on diagram $\mathrm X_{\EuScript J}.$
\end{Def}
Like in the case of $\star$-sheaves we can also define subcategories of cohomological constructible $!$-sheaves $\mathrm D_{\mathrm c}(\underline{\Gamma}({\mathrm H}^{!}_{{\EuScript J}}))$ and $\mathrm D_{\mathrm {cart},\mathrm c}(\underline{\Gamma}({\mathrm H}^{!}_{{\EuScript J}})).$
\begin{remark}\label{rps} Unlike to the case of admissible $\star$-sheaves, category $\mathrm D_{\mathrm {cart},\mathrm c}(\underline{\Gamma}({\mathrm H}^{!}_{{\EuScript J}}))$ does not carry trivial $\mathrm t$-structure, however if diagram $\mathrm X_{\EuScript J}$ is equipped with perversity function, then by Proposition \ref{pts} category of admissible $!$-sheaves can be equipped with perverse $\mathrm t$-structure (which can also be defined via Verdier duality functor $\mathbb V_{!\mapsto \star}$), with corresponding heart given by so called category of perverse $!$-sheaves $\mathrm {Perv}^!(\mathrm X_{\EuScript J},\mathrm p).$ For example one can consider diagram $\mathbb A^{\mathbb R}_{\EuScript S^{\emptyset}}$ (Example \ref{rrs}, see also Remark \ref{rc}). Categories of $!$-sheaves and $\star$-sheaves are canonically equivalent and carries two natural perverse $\mathrm t$-structures $\mathrm p_{\mathrm {min}}$ and $\mathrm p_{\mathrm {min}}.$ Note that category of factorizable sheaves, which live in the heart $\mathrm {Perv}^!(\mathbb A^{\mathbb R}_{\EuScript S^{\emptyset}},\mathrm p_{\mathrm {min}})$ is equivalent to the category of $\mathbb N$-graded associative algebras.

\end{remark}
\subsection{$!\star$-sheaves}

Let us define following category, objects of this category are pairs $(\mathrm X,\EuScript K),$ where $\mathrm X\in\mathrm {Top}$ is the topological space and $\EuScript K$ is the sheaf of $\mathrm A$-modules on $\mathrm X.$ Morphism between objects $(\mathrm X_1,\EuScript K_1)$ and $(\mathrm X_2,\EuScript K_2)$ is defined as pair $(\varphi,\mathrm f),$ where map $\mathrm f$ is a morphism in $\mathrm{span}(\mathrm{Top})$ which is given by triple $\mathrm f:=(\mathrm f_{\mathrm l},\mathrm Z,\mathrm f_{\mathrm r})$ and $\varphi$ is a morphism of sheaves: $\mathrm f_{\mathrm r!}\mathrm f_{\mathrm l}^*\EuScript K_1\rightarrow \EuScript K_2.$ Composition is defined by taking fibered products and base change theorem. We denote this category by $\mathrm H_{\mathrm {top}}^{\star!}.$ With the diagram $\mathrm X_{\EuScript J}\colon \EuScript J\longrightarrow \mathrm {Top},$ we can associate diagram $\mathrm {cospan}(\mathrm X_{\EuScript J})\colon\mathrm{cospan}(\EuScript J^{\circ})\longrightarrow\mathrm{span}(\mathrm {Top}).$ We have cofibered category $\mathrm H_{\EuScript J}^{\star!}:=\mathrm {cospan}(\EuScript J^{\circ}) \times_{\mathrm{span}(\mathrm {Top})} \mathrm{H}^{\star!}_{\mathrm {top}}$ over $\mathrm {cospan}(\EuScript J^{\circ}).$
\begin{Def}\label{d5.4.1}
\underline{Category of $!\star$-sheaves} $\mathrm {Sh}^{!\star}(\mathrm X_{\EuScript J})$ on the diagram $\mathrm X_{\EuScript J}$ is the total category of cofibration $\mathrm H_{{\star!}}^{\EuScript J}\longrightarrow \mathrm{cospan}({\EuScript J}^{\circ}).$
\end{Def}
This category is also abelian and derived category will be denoted by $\mathrm D_{}(\underline{\Gamma}({\mathrm H}_{\star!}^{{\EuScript J}})).$ Note that we actually have a bifibration $(\mathrm H_{{\EuScript J}}^{\star!},\mathrm H^{{\EuScript J}}_{!\star})$ over $\mathrm {cospan}({\EuScript J}^{\circ}).$

\begin{Prop}\label{p5.4.1} Quadruple $(\mathrm H_{\star}^{{\EuScript J}},\mathrm H_{{\EuScript J}}^{\star},\mathrm H_{!}^{{\EuScript J}},\mathrm H_{\mathrm {\EuScript J}}^!)$ is the Grothendieck cross functor, associated with ringed $\EuScript J^{\circ}$-topos $(\mathbf {Sh}^{\EuScript J}_{\mathrm {top}},\mathrm A)).$ Corresponding Mackey $\star!$-functor is pseudo-functor $\mathrm H_{{{\EuScript J}}}^{\star!}$ and Mackey $!\star$-functor is $\mathrm H_{!\star}^{ {\EuScript J}}.$
\end{Prop}
\begin{proof} Evident.

\end{proof}
We have associated covariant Verdier duality functors:
\begin{equation}\label{e5.4.1}
\mathbb V_{\star\mapsto!}\colon \mathrm D_{\mathrm{cocart}}(\underline{\Gamma}(\mathrm H^{\star}_{{\EuScript J}}))\longleftrightarrow\mathrm  D_{\mathrm{cart}}(\underline{\Gamma}(\mathrm H^{!}_{{\EuScript J}}))\colon \mathbb V_{!\mapsto\star},
\end{equation}
Let $\mathrm w_{{\EuScript J^{\circ}}}\in \mathrm  D_{\mathrm{cart}}(\mathrm H^{!}_{{\EuScript J}})$ be an object defined as dualizing complex $\mathrm w_{\mathrm i}\in \mathrm D(\mathrm X_{\mathrm i},\mathrm A)$ for every $\mathrm i\in \EuScript J,$ which we call $!$-dualizing objects on diagram $\mathrm X_{\EuScript J}.$ We have evident corollary from Proposition \ref{p5.4.1}:
\begin{Cor}\label{p5.4.2} $(\mathrm H_{\star}^{{\EuScript J}},\mathrm H_{{\EuScript J}}^{\star},\mathrm H_{!}^{{\EuScript J}},\mathrm H_{\mathrm {\EuScript J}}^!,\mathrm w_{\EuScript J^{\circ}})$ is Grothendieck cross functor with dualizing object.
\end{Cor}
If we have stratification $\EuScript S$ on diagram $\mathrm X_{\EuScript J},$ then we have constructible Grothendieck cross functor with corresponding duality functors:
\begin{equation}\label{e5.4.2}
\mathbf D_{\star\mapsto!}\colon \mathrm D_{\mathrm{cocart},\mathrm c}(\underline{\Gamma}(\mathrm H^{\star}_{{\EuScript J}}))\longleftrightarrow\mathrm  D_{\mathrm{cart},\mathrm c}(\underline{\Gamma}(\mathrm H^{!}_{{\EuScript J}}))^{\circ}\colon \mathbf D_{!\mapsto\star},
\end{equation}
and contravariant Verdier duality functor for $!$-sheaves:
\begin{equation}\label{e5.4.4}
\mathbb  D_{\mathrm H^!_{{\EuScript J}}}\colon \mathrm D_{\mathrm{cart},\mathrm c}(\underline{\Gamma}(\mathrm H^{!}_{{\EuScript J}}))\longrightarrow\mathrm  D_{\mathrm{cart},\mathrm c}(\underline{\Gamma}(\mathrm H^{!}_{{\EuScript J}}))^{\circ}
\end{equation}
\begin{remark}\label{rc} Let $\mathrm X\in \mathrm {Top}$ be a topological space with the filtration:
\begin{equation}\label{}
\mathrm X_1\subset\mathrm X_2\subset\dots\subset \mathrm X_{\mathrm n}=\mathrm X,
\end{equation}
where $\mathrm X_{\mathrm i}$ are closed subspaces. We can consider space $\mathrm X$ with such filtration as diagram from finite category $[\mathrm n]\in \mathrm {Cat}.$ With every object $\EuScript K\in \mathrm D^{\mathrm b}(\mathrm X,\mathrm A)$ we can associate cohomological Postnikov system \cite{BBD}, which plays role of corresponding $\star$-sheaf and homological Postnikov system, which plays role of $!$-sheaf. Then contravariant Verdier duality functor is identity and functor $\mathbb  D_{\mathrm H^!_{{[\mathrm n]}}}$
is given by usual Verdier duality, which interchanges these two Postnikov systems.

We can consider space $\mathrm X$ as stratified space $(\mathrm X,\mathrm S).$ Suppose that stratified spaces $(\mathrm X,\mathrm S)$ is equipped with perversity function $\mathrm p\colon\mathrm S\longrightarrow \mathbb Z.$ Construction from Proposition \ref{pts}  in this case defines usual perverse $\mathrm t$-structure, whose heart is an abelian category of $\mathrm p$-perverse sheaves $\mathrm {Perv}(\mathrm X,\mathrm S).$
\end{remark}
In the same spirit we can define pseudo-functors $\mathrm H_{!!}^{\EuScript J}$ and $\mathrm H^{\star\star}_{\EuScript J}$ with corresponding functors $\Xi_{\star}$ and $\Xi_!.$
\subsection{Operations on $!$-sheaves}

Let $\mathrm f\colon \mathrm X_{\EuScript J}\longrightarrow \mathrm Y_{\EuScript J},$ be a morphism of diagrams, then we have a morphism between $\mathrm D_{\mathrm{cart},\mathrm c}(\underline{\Gamma}(\mathrm H^!_{\mathrm X_{\EuScript J}}))$ and $\mathrm D_{\mathrm{cart},\mathrm c}(\underline{\Gamma}(\mathrm H^!_{\mathrm Y_{\EuScript J}})),$ given by the pair of adjoint $!$-functors:
 \begin{equation}\label{}
\mathrm f_{!}\colon \mathrm D_{\mathrm{cart},\mathrm c}(\underline{\Gamma}(\mathrm H^!_{\mathrm X_{\EuScript J}}))\longleftrightarrow \mathrm D_{\mathrm{cart},\mathrm c}(\underline{\Gamma}(\mathrm H^!_{\mathrm Y_{\EuScript J}}))\colon \mathrm f^{!},
 \end{equation}
where functor $\mathrm f_{!}$ is defined as composition of the componentwise $!$-pushforward functor $\underline{\mathrm R}(\mathrm f_{\mathrm i!})\colon \mathrm {D}_{}(\mathrm X_{\mathrm i},\mathrm A)\rightarrow \mathrm {D}_{}(\mathrm Y_{\mathrm i},\mathrm A),$ $\mathrm i\in \EuScript J$ and functor $\Xi_!.$ Functor $\mathrm f^{!}$ is defined as componentwise $!$-pullback functor $\mathrm f^!(\EuScript K)_{\mathrm i}:=\mathrm f^!_{\mathrm i}(\EuScript K_{\mathrm i}),$ where $\mathrm f^!_{\mathrm i}\colon \mathrm D(\mathrm Y_{\mathrm i},\mathrm A)\longrightarrow \mathrm D(\mathrm X_{\mathrm i},\mathrm A)$.

For a morphism of diagrams $\mathrm f\colon \mathrm X_{\EuScript J}\longrightarrow \mathrm Y_{\EuScript J}$ we can also define $\star$-operations for $!$-sheaves:
\begin{equation}\label{e5.5.2}
\mathrm f_{\star}\colon \mathrm D_{\mathrm{cart},\mathrm c}(\underline{\Gamma}(\mathrm H^!_{\mathrm X_{\EuScript J}}))\longleftrightarrow \mathrm D_{\mathrm{cart},\mathrm c}(\underline{\Gamma}(\mathrm H^!_{\mathrm Y_{\EuScript J}}))\colon \mathrm f^{\star}
\end{equation}
by the rule:
\begin{equation}\label{}
\mathrm f_{\star}:=\mathbb V_{\star\mapsto!}\circ \Xi_{\star}\circ\underline{\mathrm R}(\mathrm f_{*})\circ \mathbb V_{!\mapsto\star}\qquad\mathrm f_{}^{\star}:=\mathbb V_{\star\mapsto!}\circ \underline{\mathrm L}(\mathrm f^*)\circ \mathbb V_{!\mapsto\star}.
\end{equation}
\begin{Lemma}\label{l5.5.1} On constructible subcategory we have isomorphisms of functors:
\begin{equation}\label{e5.5.3}
\mathrm f_{\star}\overset{\sim}{\longrightarrow} \mathbb D_{\mathrm H^!_{\mathrm Y_{\EuScript J}}}\circ\mathrm f_!\circ\mathbb D_{\mathrm H^!_{\mathrm X_{\EuScript J}}},\qquad \mathrm f^{\star}\overset{\sim}{\longrightarrow} \mathbb D_{\mathrm H^!_{\mathrm Y_{\EuScript J}}}\circ\mathrm f^!\circ\mathbb D_{\mathrm H^!_{\mathrm X_{\EuScript J}}}.
\end{equation}

\end{Lemma}
\begin{proof} By commutativity property \ref{e4.2.13} of $*$-operations and $!$-operations with duality we get:
\begin{equation}\label{}
\mathbb D_{\mathrm H^!_{\mathrm Y_{\EuScript J}}}\circ\mathrm f_!\circ\mathbb D_{\mathrm H^!_{\mathrm X_{\EuScript J}}}\overset{\sim}{\longrightarrow} \mathbb D_{\mathrm H^!_{\mathrm Y_{\EuScript J}}}\circ\mathbf D_{\star\mapsto!}\circ\Xi_{\star}\circ\underline{\mathrm R}(\mathrm f_{*}) \circ \mathbb V_{!\mapsto\star}
\end{equation}
Then by Proposition \ref{p4.3.6} we have:
\begin{equation}\label{}
\mathbb D_{\mathrm H^!_{\mathrm Y_{\EuScript J}}}\circ\mathbf D_{\star\mapsto!}\circ\Xi_{\star}\circ\underline{\mathrm R}(\mathrm f_{*}) \circ \mathbb V_{!\mapsto\star}\overset{\sim}{\longrightarrow}\mathbb V_{\star\mapsto!}\circ \mathbf D_{!\mapsto\star}\circ \mathbf D_{\star\mapsto !}\circ\Xi_{\star}\circ\underline{\mathrm R}(\mathrm f_{*}) \circ \mathbb V_{!\mapsto\star}.
\end{equation}
Thus since biduality morphism is equivalence we obtain desired result. Case of $\star$-pullback can be proved analogically.
\end{proof}

\begin{Ex} For diagrams $\mathrm X_{\EuScript J}$ with subdiagram $\mathrm j\colon\mathrm Z_{\EuScript J}\hookrightarrow\mathrm X_{\EuScript J}$ important example of $\star$-functors is given by hyperbolic restriction functor $\Phi_{\mathrm Z_{\EuScript J}}.$ It is defined as:
\begin{equation}\label{}
\Phi_{\mathrm Z_{\EuScript J}}:=\mathrm a_{\mathrm Z_{\EuScript J}\star}\circ\mathrm j^!\colon \mathrm D_{\mathrm{cart},\mathrm c}(\underline{\Gamma}(\mathrm H^!_{\mathrm X_{\EuScript J}}))\longrightarrow \mathrm D(\mathrm A),
\end{equation}
where $\mathrm a_{\mathrm Z_{\EuScript J}}\colon \mathrm Z_{\EuScript J}\longrightarrow \{\mathrm {pt}\},$ is the canonical morphism to the point. In the setting of Ran prestack of $\mathbb A^1$ hyperbolic restriction to the Ran prestack of $\mathbb A^{\mathbb R}$ plays important role in the construction of Hopf algebras from factorizable sheaves. In this case functor $\Phi_{\mathbb A^{\mathbb R}_{\EuScript S^{}}}$ obeys usual properties, such as commutativity with Verdier duality and exactness with respect to middle perverse $\mathrm t$-structure.

\end{Ex}
Additional operations are given by Kan extension functors:
\begin{Def} Let $\mathrm j\colon\EuScript I\longrightarrow\EuScript J$ be a functor. We define pair of adjoint functors:
\begin{equation}\label{}
\mathrm j_!\colon \mathrm D_{\mathrm{cart},\mathrm c}(\underline{\Gamma}(\mathrm H^!_{\mathrm X_{\EuScript I}}))\longrightarrow\mathrm D_{\mathrm{cart},\mathrm c}(\underline{\Gamma}(\mathrm H^!_{\mathrm X_{\EuScript J}}))\colon\mathrm j^*,
\end{equation}
where functor $\mathrm j^*$ is given by restriction along functor $\mathrm j$ and functor $\mathrm j_!$ is defined as composition $\Xi_!\circ \underline{\mathrm L}(\mathrm j_!).$
\end{Def}

Category $\mathrm D_{\mathrm{cart},\mathrm c}(\Gamma(\mathrm H^!_{\mathrm X_{\EuScript J}}))$ carries symmetric monoidal product denoted by:
\begin{equation}\label{tp1}
\otimes^!\colon \mathrm D_{\mathrm{cart},\mathrm c}(\underline{\Gamma}(\mathrm H^!_{\mathrm X_{\EuScript J}}))\times\mathrm D_{\mathrm{cart},\mathrm c}(\Gamma(\mathrm H^!_{\mathrm X_{\EuScript J}})) \longrightarrow \mathrm D_{\mathrm{cart},\mathrm c}(\underline{\Gamma}(\mathrm H^!_{\mathrm X_{\EuScript J}}))
\end{equation}
It is given by componentwise $!$-tensor product of sheaves. Note that this tensor product is unital and unit is given by $\mathrm w_{\mathrm X_{\EuScript J^{\circ}}}.$

\begin{remark} If category $\EuScript J$ is monoidal category, with tensor product $\ast\colon \EuScript J\times \EuScript J\longrightarrow \EuScript J$ then we can define additional symmetric tensor structure on $\mathrm  D_{\mathrm{cart}}(\mathrm H^!_{\EuScript J}),$ by Day convolution product:
\begin{equation}\label{}
\otimes^{\ast}\colon \mathrm  D_{\mathrm{cart}}(\underline{\Gamma}(\mathrm H^!_{\mathrm X_{\EuScript J}}))\times \mathrm  D_{\mathrm{cart}}(\mathrm H^!_{\mathrm X_{\EuScript J}})\longrightarrow\mathrm  D_{\mathrm{cart}}(\underline{\Gamma}(\mathrm H^!_{\mathrm X_{\EuScript J}}))
\end{equation}
Examples of such tensor products are given by tensor structures on the Ran prestack, introduced in \cite {BD}.
\end{remark}
Let $\EuScript J$ be a category, denote by $\mathrm i\colon \EuScript J\longrightarrow \EuScript J\times \EuScript J$ canonical diagonal embedding functor. Let $\mathrm X_{\EuScript J}\colon \EuScript J\longrightarrow\mathrm {Top}$ be a diagram. We have diagrams $\mathrm i_!\mathrm X_{\EuScript J},$ which is defined as Kan extension along morphism $\mathrm i.$ Denote by $\mathrm X_{\EuScript J}\times \mathrm X_{\EuScript J}\colon \EuScript J\times \EuScript J\longrightarrow \mathrm {Top}$ diagram, which is defined as composition of product diagram $\mathrm X_{\EuScript J}\times \mathrm X_{\EuScript J},$ with cartesian monoidal structure on $\mathrm {Top}.$ Then by universal property of colimits we have following morphism of diagrams:
\begin{equation}\label{}
\mathrm {diag}\colon \mathrm i_!\mathrm X_{\EuScript J}\longrightarrow \mathrm X_{\EuScript J}\times \mathrm X_{\EuScript J},
\end{equation}
Note that morphism $\mathrm {diag}$ is closed embedding. Contravariant Verdier duality functor for $!$-sheaves enjoys following properties:
\begin{Prop} $(\mathrm {i})$ Space of morphisms between $\EuScript F\boxtimes \EuScript K$ and $\mathrm {diag}_!\mathrm w_{\mathrm X_{\EuScript J}},$ is represented\footnote{This property is taken as definition in \cite{Gai1}} by contravariant Verdier duality functor:
\begin{equation}\label{vadjp}
\mathrm {Hom}_{\mathrm D_{\mathrm{cart},\mathrm c}(\mathrm H^!_{{\EuScript J\times \EuScript J}})}(\EuScript K\boxtimes\EuScript G,\mathrm {diag}_!\mathrm w_{\mathrm X_{\EuScript J}})=\mathrm {Hom}_{\mathrm D_{\mathrm{cart},\mathrm c}(\mathrm H^!_{{\EuScript J}})}(\EuScript K,\mathbb D_{\mathrm H^!_{\mathrm {\EuScript J}}}\EuScript G)
\end{equation}

$(\mathrm {ii})$ We have canonical biduality morphism:
\begin{equation}\label{}
\mathrm {Id}_{\mathrm D_{\mathrm{cart},\mathrm c}(\mathrm H^!_{{\EuScript J}})}\longrightarrow \mathbb D_{\mathrm H^!_{\mathrm {\EuScript J}}}\circ \mathbb D_{\mathrm H^!_{\mathrm {\EuScript J}}}
\end{equation}

\end{Prop}
\begin{proof} Evident. 

\end{proof}

\begin{Ex}\label{rrs} Let $\EuScript S$ be a category of finite sets and surjective morphisms between them. Following \cite{BD} with the topological spaces $\mathrm X$ we associate diagram $\mathrm X_{\EuScript S}.$  Denote by $\EuScript S^{\emptyset}$ subcategory of finite sets and isomorphisms, with canonical inclusion functor $\mathrm {inc}\colon\EuScript S^{\emptyset}\longrightarrow \EuScript S.$ We have homotopy left Kan extension functors, which includes in the following commutative diagram:
\begin{equation}\label{e4.3.11}
\begin{diagram}[height=2.8em,width=5.3em]
                                                                   &                         & \mathrm D_{\mathrm{cart}}(\underline{\Gamma}(\mathrm H^!_{\EuScript S^{\emptyset}}))  &     & \\
                                                                   & \ldTo^{\mathrm {inc}_!^{\circ}} &                                                           &  \rdTo^{\mathrm {inc}_!} & \\
\mathrm D_{\mathrm{cart}}(\underline{\Gamma}(\mathrm H^!_{\EuScript S^{\circ}})) & & \rTo^{\mathbb V_{!\mapsto \star}} & & \mathrm D_{\mathrm{cocart}}(\underline{\Gamma}(\mathrm H^{\star}_{\EuScript S^{\circ}}))\\
\end{diagram}
\end{equation}
Functor $\mathrm {inc}^{\circ}_!$ induces equivalence between category $\mathrm D_{\mathrm{cart}}(\underline{\Gamma}(\mathrm H^!_{\EuScript S^{\emptyset}}))$ and subcategory of graded $!$-sheaves on Ran space $\mathrm D_{\mathrm{cart}}^{\mathrm {gr}}(\underline{\Gamma}(\mathrm H^!_{\EuScript S^{\circ}})),$ which consists of objects $\EuScript K\in \mathrm D_{\mathrm{cart}}(\underline{\Gamma}(\mathrm H^!_{\EuScript S^{\circ}}))$ such that
\begin{equation}\label{}
\EuScript K_{\mathrm I}\in \mathrm D^{\mathrm {gr}}(\mathrm X^{\mathrm I})\quad\mathrm {supp}(\mathrm {gr}_{\mathrm n}\EuScript K_{\mathrm I})\subset\mathrm X^{\mathrm n}\quad \mathrm n<|\mathrm I|
\end{equation}
Let us consider diagrams $\mathrm X_{\EuScript S}$ equipped with diagonal stratification, suppose that we have perversity function $\mathrm p$ on $\mathrm X_{\EuScript S}.$ Commutativity of above diagrams implies that functor $\mathrm {inc}^{\circ}_!$ is perverse $\mathrm t$-exact, thus we define category of graded $\mathrm p$-perverse $!$-sheaves on $\mathrm X_{\EuScript S}$ as image of functor $\mathrm {inc}^{\circ}_!:$
\begin{equation}\label{}
\mathrm {inc}^{\circ}_!\colon\mathrm {Perv}^!(\mathrm X_{\EuScript S^{\emptyset}},\mathrm p) \overset{\sim}{\longrightarrow} \mathrm {Perv}^{!\mathrm {gr}}_{\mathrm {cart}}(\mathrm X_{\EuScript S^{\circ}},\mathrm p)
\end{equation}
Under above equivalence category of graded $!$-sheaves on the Ran spaces of $\mathrm X$ can be equipped with the new tensor structure, which is defined by the $!$-tensor product \ref{tp1} on $\mathrm D_{\mathrm{cart}}(\underline{\Gamma}(\mathrm H^!_{\EuScript S^{\emptyset}})):$
\begin{equation}\label{}
\circ\colon \mathrm D_{\mathrm{cart}}^{\mathrm {gr}}(\underline{\Gamma}(\mathrm H^!_{\EuScript S^{\circ}}))\times \mathrm D_{\mathrm{cart}}^{\mathrm {gr}}(\underline{\Gamma}(\mathrm H^!_{\EuScript S^{\circ}}))\longrightarrow\mathrm D_{\mathrm{cart}}^{\mathrm {gr}}(\underline{\Gamma}(\mathrm H^!_{\EuScript S^{\circ}}))
\end{equation}
This tensor structure is unital and unit is given by $\mathrm {inc}_!(\mathrm w_{\EuScript S^{\emptyset}}).$ In the case when $\mathrm X=\mathbb A^{\mathbb R}$ and perversity function is given by minimal perversity function $\mathrm p_{\mathrm {min}}$ tensor product $\circ$ on the category $\mathrm {Perv}^{!\mathrm {gr}}_{\mathrm {cart}}(\mathrm X_{{\EuScript S}},\mathrm p_{\mathrm {min}})$ is exact and corresponds to the white product of quadratic algebras \cite{PP} and $\mathrm {inc}_!(\mathrm w_{\EuScript S^{\emptyset}})$ is the free associative algebra on one variable. Note that contravariant Verdier duality induces anti-equivalence of category of graded perverse $!$-sheaves:
\begin{equation}\label{}
\mathbb  D_{\mathrm H^!_{{{{\EuScript S^{\circ}}}}}}\colon \mathrm {Perv}^{!\mathrm {gr}}_{\mathrm {cart}}(\mathrm X_{{\EuScript S}},\mathrm p_{\mathrm {min}})\overset{\sim}{\longrightarrow} \mathrm {Perv}^{!\mathrm {gr}}_{\mathrm {cart}}(\mathrm X_{{\EuScript S}},\mathrm p_{\mathrm {max}})^{\circ}
\end{equation}
Thus we can also define another unital tensor structure on $\mathrm {Perv}^{!\mathrm {gr}}_{\mathrm {cart}}(\mathrm X_{{\EuScript S}},\mathrm p_{\mathrm {min}}):$
\begin{equation}\label{}
\EuScript K\bullet\EuScript E:=\mathbb  D_{\mathrm H^!_{{{{\EuScript S^{\circ}}}}}}^{-1}(\mathbb  D_{\mathrm H^!_{{{{\EuScript S^{\circ}}}}}}\EuScript K\cdot\mathbb  D_{\mathrm H^!_{{{{\EuScript S^{\circ}}}}}}\EuScript E)),
\end{equation}
Product $\cdot$ is defined analogically to $\circ,$ instead of $!$-product we use ordinary product of sheaves. In the case when $\mathrm X=\mathbb A^{\mathbb R}$ tensor product $\bullet$ on the category $\mathrm {Perv}^{!{\mathrm {gr}}}_{\mathrm {cart}}(\mathrm X_{{\EuScript S^{\circ}}},\mathrm p_{\mathrm {min}})$ corresponds to the black product of quadratic algebras and the unit is given by $\mathrm {inc}_!(\mathrm j_*\mathrm j^*\mathrm w_{{\EuScript S}^{\emptyset}}),$ where for every finite set $\mathrm I,$ sheaf $\mathrm j_*\mathrm j^* \mathrm w_{{\EuScript S}^{\emptyset}\mathrm I}$ is defined as $\mathrm j_{\mathrm I*} {\mathrm j^*}_{\mathrm I}\mathrm w_{\mathrm X^{\mathrm I}},$ and $\mathrm j_{\mathrm I}\colon \mathrm U\hookrightarrow \mathrm X^{\mathrm I}$ is open complement, transversal to the smallest diagonal $\mathrm X\subset \mathrm X^{\mathrm I}.$ Under correspondence between factorizable sheaves and graded algebras this sheaf corresponds to the algebra of dual numbers.
\end{Ex}

\begin{remark} Generally in the case of $!$-sheaves we does not have a notion of inner hom (compare with Example \ref{rrs}), but contravariant Verderer enjoys hom like adjunction property \ref{vadjp}. On the another hand in the case of $\star$-sheaves we have notion of inner hom $\mathrm {Hom}^{\star}$ (see \ref{innerh}). Reader can compare this with the case of $\EuScript D$-modules on the scheme.

\end{remark}

\newpage
\bibliographystyle{alphanum}
\bibliography{tt}

\end{document}